\documentclass[10pt,twoside]{amsart}
\def\YEAR{\year}\newcount\VOL\VOL=\YEAR\advance\VOL by-1995
\def\firstpage{1}\def\lastpage{1000}
\def\received{}\def\revised{}
\def\communicated{}

\makeatletter
\def\magnification{\afterassignment\m@g\count@}
\def\m@g{\mag=\count@\hsize6.5truein\vsize8.9truein\dimen\footins8truein}
\makeatother

\font\eightrm=cmr8
\font\caps=cmcsc10                    
\font\Caps=cmcsc10 scaled \magstep1   

\makeatletter
\@specialpagefalse\headheight=8.5pt
\def\DocMath{}
\renewcommand{\@evenhead}{%
    \ifnum\thepage>\lastpage\rlap{\thepage}\hfill%
    \else\rlap{\thepage}\slshape\leftmark\hfill{\caps\SAuthor}\hfill\fi}%
\renewcommand{\@oddhead}{%
    \ifnum\thepage=\firstpage{\DocMath\hfill\llap{\thepage}}%
    \else{\slshape\rightmark}\hfill{\caps\STitle}\hfill\llap{\thepage}\fi}%
\makeatother

\def\TSkip{\bigskip}
\newbox\TheTitle{\obeylines\gdef\GetTitle #1
\ShortTitle  #2
\SubTitle    #3
\Author      #4
\ShortAuthor #5
\EndTitle
{\setbox\TheTitle=\vbox{\baselineskip=20pt\let\par=\cr\obeylines%
\halign{\centerline{\Caps##}\cr\noalign{\medskip}\cr#1\cr}}%
	\copy\TheTitle\TSkip\TSkip%
\def\next{#2}\ifx\next\empty\gdef\STitle{#1}\else\gdef\STitle{#2}\fi%
\def\next{#3}\ifx\next\empty%
    \else\setbox\TheTitle=\vbox{\baselineskip=20pt\let\par=\cr\obeylines%
    \halign{\centerline{\caps##} #3\cr}}\copy\TheTitle\TSkip\TSkip\fi%
\centerline{\caps #4}\TSkip\TSkip%
\def\next{#5}\ifx\next\empty\gdef\SAuthor{#4}\else\gdef\SAuthor{#5}\fi%
\ifx\received\empty\relax
    \else\centerline{\eightrm Received: \received}\fi%
\ifx\revised\empty\TSkip%
    \else\centerline{\eightrm Revised: \revised}\TSkip\fi%
\ifx\communicated\empty\relax
    \else\centerline{\eightrm Communicated by \communicated}\fi\TSkip\TSkip%
\catcode'015=5}}\def\Title{\obeylines\GetTitle}
\def\Abstract{\begingroup\narrower
    \parskip=\medskipamount\parindent=0pt{\caps Abstract. }}
\def\EndAbstract{\par\endgroup\TSkip}

\long\def\MSC#1\EndMSC{\def\arg{#1}\ifx\arg\empty\relax\else
     {\par\narrower\noindent%
     2000 Mathematics Subject Classification: #1\par}\fi}

\long\def\KEY#1\EndKEY{\def\arg{#1}\ifx\arg\empty\relax\else
	{\par\narrower\noindent Keywords and Phrases: #1\par}\fi\TSkip}

\newbox\TheAdd\def\Addresses{\vfill\copy\TheAdd\vfill
    \ifodd\number\lastpage\vfill\eject\phantom{.}\vfill\eject\fi}
{\obeylines\gdef\GetAddress #1
\Address #2
\Address #3
\Address #4
\EndAddress
{\def\xs{4.3truecm}\parindent=0pt
\setbox0=\vtop{{\obeylines\hsize=\xs#1\par}}\def\next{#2}
\ifx\next\empty 
     \setbox\TheAdd=\hbox to\hsize{\hfill\copy0\hfill}
\else\setbox1=\vtop{{\obeylines\hsize=\xs#2\par}}\def\next{#3}
\ifx\next\empty 
     \setbox\TheAdd=\hbox to\hsize{\hfill\copy0\hfill\copy1\hfill}
\else\setbox2=\vtop{{\obeylines\hsize=\xs#3\par}}\def\next{#4}
\ifx\next\empty\ 
     \setbox\TheAdd=\vtop{\hbox to\hsize{\hfill\copy0\hfill\copy1\hfill}
                \vskip20pt\hbox to\hsize{\hfill\copy2\hfill}}
\else\setbox3=\vtop{{\obeylines\hsize=\xs#4\par}}
     \setbox\TheAdd=\vtop{\hbox to\hsize{\hfill\copy0\hfill\copy1\hfill}
	        \vskip20pt\hbox to\hsize{\hfill\copy2\hfill\copy3\hfill}}
\fi\fi\fi\catcode'015=5}}\gdef\Address{\obeylines\GetAddress}

\usepackage{amssymb}
\usepackage{amscd}
\usepackage{amsmath}
\usepackage{amsrefs}

\newtheorem{theorem}{Theorem}[section]
\newtheorem{prop}{Proposition}[section]
\newtheorem{lemma}{Lemma}
\newtheorem{corollary}{Corollary}
\newtheorem{conjecture}{Conjecture}
\newtheorem{definition}{Definition}

\theoremstyle{remark}
\newtheorem{remark}{Remark}
\newtheorem*{remarks}{Remarks}

\newcommand{\tr}{\tilde{\br}}

\newcommand{\bq}{\mathbb Q}
\newcommand{\bqpur}{\hat{\mathbb Q}_p^{ur}}
\newcommand{\bz}{\mathbb Z}
\newcommand{\br}{\mathbb R}
\newcommand{\bc}{{\mathbb C}}
\newcommand{\bg}{\mathbb G}

\newcommand{\bof}{\mathbb F}

\newcommand{\BC}{{\mathbb C}}

\newcommand{\F}{\mathcal F}
\newcommand{\T}{\mathcal T}
\newcommand{\G}{\mathcal G}

\newcommand{\co}{\mathcal O}

\newcommand{\et}{\mathrm{et}}

\newcommand{\E}{\mathcal E}
\newcommand{\X}{\mathcal X}
\newcommand{\Y}{\mathcal Y}
\newcommand{\p}{\mathfrak p}
\newcommand{\beq}{\begin{equation}}
\newcommand{\eeq}{\end{equation}}
\newcommand{\spe}{\mathrm sp}
\newcommand{\A}{\mathcal A}

\newcommand{\gr}{{\mathrm gr}}
\newcommand{\sbar}{{\bar{s}}}
\newcommand{\etabar}{{\bar{\eta}}}
\newcommand{\mydet}{\mathrm{det}}

 \DeclareMathOperator{\Spec}{Spec}
\DeclareMathOperator{\Pic}{Pic} \DeclareMathOperator{\Cl}{Cl}

\DeclareMathOperator{\Hom}{Hom}

\DeclareMathOperator{\slope}{slope}

\DeclareMathOperator{\Gal}{Gal}
\DeclareMathOperator{\ord}{ord}
\DeclareMathOperator{\Tot}{Tot}\DeclareMathOperator{\Frob}{Frob}

\newcommand{\fonc}[5]{
 \begin{array}{cccc}
 #1: & #2 & \longrightarrow & #3\\
     & #4 & \longmapsto & #5
 \end{array}
}
\newcommand{\appl}[4]{
 \begin{array}{cccc}
   #1 & \longrightarrow & #2\\
   #3 & \longmapsto & #4
 \end{array}
}

\begin{document}

\Title
On the Weil-\'etale topos of regular arithmetic schemes
\ShortTitle
\SubTitle
\Author
M. Flach and B. Morin
\ShortAuthor
\EndTitle
\Abstract
We define and study a Weil-\'etale topos for any regular, proper scheme $\X$ over $\Spec(\bz)$ which has some of the properties suggested by Lichtenbaum for such a topos. In particular, the cohomology with $\tr$-coefficients has the expected relation to $\zeta(\X,s)$ at $s=0$ if the Hasse-Weil L-functions $L(h^i(\X_\bq),s)$ have the expected meromorphic continuation and functional equation.  If $\X$ has characteristic $p$ the cohomology with $\bz$-coefficients also has the expected relation to $\zeta(\X,s)$ and our cohomology groups recover those previously studied by Lichtenbaum and Geisser.
\EndAbstract
\MSC
Primary: 14F20, 11S40, Secondary: 11G40, 18F10
\EndMSC
\KEY
\EndKEY
\Address
Department of Mathematics, Caltech, Pasadena CA 91125, USA
\Address
Department of Mathematics, Caltech, Pasadena CA 91125, USA
\Address
\Address
\EndAddress

\section{Introduction} In \cite{li04} Lichtenbaum suggested the
existence of Weil-\'etale cohomology groups for arithmetic schemes
$\X$ (i.e. separated schemes of finite type over $\Spec(\bz)$) which
are related to the zeta-function $\zeta(\X,s)$ of $\X$ as follows.
\begin{itemize}
\item[a)] The compact support cohomology groups $H^i_c(\X_W,\tr)$
are finite dimensional vector spaces over $\br$, vanish for almost
all $i$ and satisfy
\[ \sum_{i\in\bz}(-1)^i\dim_\br H^i_c(\X_W,\tr)=0.\]
\item[b)] The function $\zeta(\X,s)$ has a meromorphic continuation to $s=0$ and
\[ \ord_{s=0}\zeta(\X,s)=\sum_{i\in\bz}(-1)^i\cdot i\cdot\dim_\br
H^i_c(\X_W,\tr).\]
\item[c)] There exists a canonical class $\theta\in H^1(\X_W,\tr)$ so
that the sequence
\[
\cdots\xrightarrow{\cup\theta}H^i_c(\X_W,\tr)\xrightarrow{\cup\theta}H^{i+1}_c(\X_W,\tr)\xrightarrow{\cup\theta}\cdots\]
is exact.
\item[d)] The compact support cohomology groups $H^i_c(\X_W,\bz)$
are finitely generated over $\bz$ and vanish for almost all $i$.
\item[e)] The natural map from $\bz$ to $\tr$-coefficients induces an isomorphism
\[ H^i_c(\X_W,\bz)\otimes_\bz\br\xrightarrow{\sim}H^i_c(\X_W,\tr).\]
\item[f)]  If $\zeta^*(\X,0)$ denotes the leading Taylor-coefficient of $\zeta(\X,s)$ at $s=0$
and \[
\lambda:\br\cong\bigotimes_{i\in\bz}\mydet_{\br}H^i_c(\X_W,\tr)^{(-1)^i}\]
the isomorphism induced by c) then
\[
\bz\cdot\lambda(\zeta^*(\X,0))=\bigotimes_{i\in\bz}\mydet_{\bz}H^i_c(\X_W,\bz)^{(-1)^i}\]
where the determinant is understood in the sense of \cite{knumum}.
\end{itemize}

If $\X$ has finite characteristic these groups are well defined and
well understood by work of Lichtenbaum \cite{li01} and Geisser
\cites{geisser04,geisser05}. In particular all the above properties a)-f) hold
for $\dim(\X)\leq 2$ and in general under resolution of
singularities. Lichtenbaum also defined such groups
for $\X=\Spec(\co_F)$ where $F$ is a number field and showed that
a)-f) hold if one artificially redefines $H^i_c(\Spec(\co_F)_W,\bz)$
to be zero for $i\geq 4$. In \cite{flach06-2} it was then shown that
$H^i_c(\Spec(\co_F)_W,\bz)$ as defined by Lichtenbaum does indeed
vanish for odd $i\geq 5$ but is an abelian group of infinite rank
for even $i\geq 4$.

In any case, in Lichtenbaum's definition the groups
$H_c^i(\Spec(\co_F)_W,\bz)$ and $H_c^i(\Spec(\co_F)_W,\tr)$ are
defined via an Artin-Verdier type compactification
$\overline{\Spec(\co_F)}$ of $\Spec(\co_F)$ \cite{av}, where however
$H^i(\overline{\Spec(\co_F)}_W,\F)$ is not the cohomology group of a
topos but rather a direct limit of such. The first purpose of this
article is to give a definition of a topos $\overline{\Spec(\co_F)}_W$
which recovers Lichtenbaum's groups (see section \ref{wdef} below). This definition was proposed in
the second author's thesis \cite{morin} and is a natural
modification of Lichtenbaum's idea which is suggested by a closer
look at the \'etale topos $\Spec(\co_F)_\et$.

In \cite{av} Artin and Verdier defined a topos $\overline{\X}_\et$ for any arithmetic scheme $\X\to\Spec(\bz)$
so that there are complementary open and closed immersions
\[ \X_\et\to \overline{\X}_\et \leftarrow Sh(\X_\infty) \]
the sense of topos theory \cite{sga4}. Here $\X_\infty$ is the topological quotient space $\X(\bc)/G_\br$ where $\X(\bc)$ is the set of complex points with its standard Euclidean topology and $G_\br=\Gal(\bc/\br)$. If $\X$ is an arithmetic scheme and
$\Y$ denotes either $\X$ or $\overline{\X}$ we define the
Weil-\'etale topos of $\Y$ by
\[ \Y_W:=\Y_\et\times_{\overline{\Spec(\bz)}_\et}\overline{\Spec(\bz)}_W,\]
a fibre product in the 2-category of topoi. This definition is
suggested by the fact that the Weil-\'etale topos defined by Lichtenbaum for varieties over finite fields
is isomorphic to a similar fibre product, as was
shown in the second author's thesis \cite{morin} and will be recalled in section \ref{fibre} below. The work of Geisser \cite{geisser05} shows that Lichtenbaums's definition
is only reasonable (i.e. satisfies a)-f)) for smooth, proper varieties over finite fields. Correspondingly, one can only expect our fibre product definition to be reasonable for proper {\em regular} arithmetic schemes.

The second purpose of this article is to show that this is indeed the case as far as $\tr$-coefficients are concerned. Our main result is the
following

\begin{theorem} Let $\X$ be a regular scheme, proper over $\Spec(\bz)$.
\begin{itemize} \item[i)] For $\X=\Spec(\co_F)$ one has
\[ \overline{\Spec(\co_F)}_W\cong \overline{\Spec(\co_F)}_\et\times_{\overline{\Spec(\bz)}_\et}\overline{\Spec(\bz)}_W,\]
where $\overline{\Spec(\co_F)}_W$ is the topos defined in section \ref{wdef} below, based on Lichtenbaum's idea of replacing Galois groups by Weil groups.
\item[ii)] If $\X\to\Spec(\bof_p)$ has characteristic $p$ then our groups agree with those of
Lichtenbaum and Geisser and a)-f) hold for $\X$.
\item[iii)] If $\X$ is flat over $\Spec(\bz)$ and the Hasse-Weil
L-functions $L(h^i(\X_\bq),s)$ of all motives $h^i(\X_\bq)$ satisfy
the expected meromorphic continuation and functional equation. Then
a)-c) hold for $\X$.
\end{itemize}
\label{introtheo}\end{theorem}

The assumptions of iii) are satisfied, for example, if $\X$ is a regular
model of a Shimura curve, or of a self product $E\times\cdots \times
E$ where $E$ is an elliptic curve, over a totally real field $F$.

Unfortunately, properties d) and e) do not hold with our fibre
product definition, even in low degrees, and we also do not expect them
to hold with any similar definition (see the remarks in section \ref{rems}). The right definition of
Weil-\'etale cohomology with $\bz$-coefficients for schemes of
characteristic zero will require a key new idea, as is already
apparent for $\X=\Spec(\co_F)$.

We briefly describe the content of this article. In section \ref{prem} we recall preliminaries on sites, topoi and classifying topoi. Section \ref{fibre} contains the proof that Lichtenbaum's Weil-\'etale topos in characteristic $p$ is a fibre product via a method that is different from the one in the second author's thesis \cite{morin}. In section \ref{sect-AVetaletopos} we recall the definition of $\overline{\X}_\et$ and the corresponding compact support cohomology groups $H^i_c(\X_\et,\F)$. In section \ref{wdef} we define $\overline{\Spec(\co_F)}_W$ and give the proof of Theorem \ref{introtheo} i) (see Proposition \ref{prop-fiberproduct-makes-sense}). In section \ref{section-xwdef} we define $\overline{\X}_W$, describe its fibres above all places $p\leq\infty$ and its generic point. In section \ref{section-cohomology-XW} we compute the cohomology of $\overline{\X}_W$ with $\tr$-coefficients following Lichtenbaum's method of studying the Leray spectral sequence from the generic point. This section is the technical heart of this article. In section \ref{comp-support} we compute the compact support cohomology $H^i_c(\X_W,\tr)$ via the natural morphism $\overline{\X}_W\to\overline{\X}_\et$ and prove properties a) and c) (see Theorem \ref{ac-theo}). The class $\theta$ in c) is defined in subsection \ref{funclass}.

Section \ref{zeta} introduces Hasse-Weil L-functions of varieties over $\bq$ as well as Zeta-functions of arithmetic schemes and contains the proof of Theorem \ref{introtheo} ii) (see Theorem \ref{charp-theo}) and of property b) (see Theorem \ref{b-theo}), thereby concluding the proof Theorem \ref{introtheo} iii). In subsection \ref{tama} we show that property f) for $\zeta(\X,s)$ is compatible with the Tamagawa number conjecture of Bloch and Kato \cite{bk88} (or rather of Fontaine and Perrin-Riou \cite{fpr91}) for $\prod_{i\in \bz} L(h^i(\X_\bq),s)^{(-1)^i}$ at $s=0$. In order to do this we need to augment the list of properties a)-f) for Weil-\'etale cohomology with further natural assumptions g)-j) of which g) and h) hold in characteristic $p$, and we need to assume a number of conjectures which are preliminary to the formulation of the Tamagawa number conjecture.
Finally, in section \ref{loc-inv-cycles} we prove some results related to the so called local theorem of invariant cycles in $l$-adic cohomology, and we formulate analogous conjectures in $p$-adic cohomology. These results may be of some interest independently of Weil-\'etale cohomology, and are necessary to establish the equality of vanishing orders
\[\ord_{s=0}\zeta(\X,s)=\ord_{s=0}\prod_{i\in \bz} L(h^i(\X_\bq),s)^{(-1)^i}\]
for regular schemes $\X$ proper and flat over $\Spec(\bz)$.
\medskip

{\em Acknowledgements:} The first author is supported by grant DMS-0701029 from the National Science Foundation. He would also like to thank Spencer Bloch for a helpful discussion about the material in section \ref{loc-inv-cycles} and the MPI Bonn for its hospitality during the final preparation of this paper.

\section{Preliminaries}\label{prem}

In this paper, a topos is a Grothendieck topos over $\underline{Set}$, and a morphism of topoi is a geometric morphism. A pseudo-commutative diagram of topoi is said to be commutative. Finally, we suppress any mention of universes.

\subsection{Left exact sites}

Recall that a Grothendieck topology $\mathcal{J}$ on a category
$\mathcal{C}$ is said to be \emph{sub-canonical} if $\mathcal{J}$ is
coarser than the canonical topology, i.e. if any representable presheaf on $\mathcal{C}$ is a sheaf for the
topology $\mathcal{J}$. A category $\mathcal{C}$ is said to be
\emph{left exact} when finite projective limits exist in
$\mathcal{C}$, i.e. when $\mathcal{C}$ has a final object and fiber
products. A functor between left exact categories is said to be left exact
if it commutes with finite projective limits.
\begin{definition}
A Grothendieck site $(\mathcal{C},\mathcal{J})$ is said to be \emph{left
exact} if $\mathcal{C}$ is a left exact category endowed with a
subcanonical topology $\mathcal{J}$. A \emph{morphism of left exact sites}
$(\mathcal{C}',\mathcal{J}')\rightarrow (\mathcal{C},\mathcal{J})$
is a continuous left exact functor
$\mathcal{C}'\rightarrow\mathcal{C}$.
\end{definition}
Note that any Grothendieck topos, i.e. any category satisfying Giraud's axioms, is equivalent to the category of sheaves of sets on a left exact site. Note also that a Grothendieck site $(\mathcal{C},\mathcal{J})$ is left exact if and only if the canonical functor (given in general by Yoneda and sheafification)
$$y:\mathcal{C}\longrightarrow\widetilde{(\mathcal{C},\mathcal{J})}$$
identifies $\mathcal{C}$ with a left exact full subcategory of $\widetilde{(\mathcal{C},\mathcal{J})}$.
The following result is proven in \cite{sga4} IV.4.9.
\begin{lemma}
A \emph{morphism of left exact sites}
$f^*:(\mathcal{C}',\mathcal{J}')\rightarrow
(\mathcal{C},\mathcal{J})$ induces a morphism of topoi
$f:(\widetilde{\mathcal{C},\mathcal{J}})\rightarrow
(\widetilde{\mathcal{C}',\mathcal{J}'})$. Moreover we have a
commutative diagram
\begin{equation*}\begin{CD}
@. (\widetilde{\mathcal{C},\mathcal{J}}) @<{f^*}<<
(\widetilde{\mathcal{C}',\mathcal{J}'})\\
@. @AA{y_{\mathcal{C}}}A @AA{y_{\mathcal{C}'}}A @.\\
@. \mathcal{C} @<{f^*}<< \mathcal{C}' @. {}
\end{CD}\end{equation*}
where the vertical arrows are the fully faithful Yoneda functors.
\end{lemma}
\subsection{The topos $\T$}

We denote by $Top^{lc}$ (respectively by $Top^c$) the category of locally compact topological spaces (respectively of compact spaces). A locally compact space is assumed to be Hausdorff. The category $Top^{lc}$ is endowed with the open cover topology $\mathcal{J}_{op}$, which is subcanonical. We denote by $\T$ the topos of sheaves of sets on the site $(Top^{lc},\mathcal{J}_{op})$. The Yoneda functor $$y:Top^{lc}\longrightarrow\T$$ is fully faithful, and $Top^{lc}$ is viewed as a generating full subcategory of $\T$. For any object $T$ of $Top^{lc}$, $T$ is locally compact hence there exist morphisms
$$\coprod yU_i\rightarrow \coprod yK_i\rightarrow yT$$
where $\{U_i\subset T\}$ is an open covering, and $K_i$ is a compact subspace of $T$. It follows that
$\coprod yU_i\rightarrow yT$ is an epimorphism in $\T$, hence so is $\coprod yK_i\rightarrow yT$. This shows that the category of compact spaces $Top^c$ is a generating full subcategory of $\T$.

The unique morphism $t:\T\rightarrow\underline{Set}$ has a section $s:\underline{Set}\rightarrow\T$ such that $t_*=s^*$ hence we have three adjoint functors $t^*,\,t_*=s^*,\,s_*$. In particular $t_*$ is exact hence we have $H^n(\T,\mathcal{A})=H^n(\underline{Set},\mathcal{A}(*))=0$ for any $n\geq1$ and any abelian object $\mathcal{A}$.

\subsection{Classifying topoi}

\subsubsection{General case.}
For any topos $\mathcal{S}$ and any group object $G$ in $\mathcal{S}$, we denote by $B_G$ the category of left $G$-object in $\mathcal{S}$. Then $B_G$ is a topos, as it follows from Giraud's axioms, and $B_G$ is endowed with a canonical morphism $B_G\rightarrow\mathcal{S}$, whose inverse image functor sends an object $F$ of $\mathcal{S}$ to $F$ with trivial $G$-action. If there is a risk of ambiguity, the topos $B_G$ is denoted by $B_{\mathcal{S}}(G)$. The topos $B_G$ is said to be the classifying topos of $G$ since for any topos $f:\mathcal{E}\rightarrow\mathcal{S}$  over $\mathcal{S}$, the category $\underline{Homtop}_{\mathcal{S}}\,(\mathcal{E},B_G)$ is equivalent to the category of $f^*G$-torsors in $\mathcal{E}$ (see \cite{sga4} IV. Exercice 5.9).

\subsubsection{Examples.}

Let $G$ be a discrete group, i.e. a group object of the final topos $\underline{Set}$. Then $B_{\underline{Set}}G$ is the category of left $G$-sets, and the cohomology groups $H^*(B_{\underline{Set}}G,A)$, where $A$ is an abelian object of $B_G$ i.e. a $G$-module, is precisely the cohomology of the discrete group $G$. Here $B_{\underline{Set}}G$ is called the \emph{small classifying topos of the discrete group $G$} and is denoted by $B_G^{sm}$. If $G$ is the profinite group, the \emph{small classifying topos $B^{sm}_G$ of the profinite group $G$}  is the category of continuous $G$-sets.

Let $G$ be a locally compact topological group. Then $G$ represents a group object of $\T$, where $\T$ is defined above. Then $B_G$ is the classifying topos of the topological group $G$, and the cohomology groups $H^*(B_G,\mathcal{A})$, where $\mathcal{A}$ is an abelian object of $B_G$ (e.g. a topological $G$-module) is the cohomology of the topological group $G$. If $G$ is not locally compact, then we just need to replace $\T$ with the category of sheaves on $(Top,\mathcal{J}_{op})$.

Let $S$ be a scheme and let $G$ be a smooth group scheme over $S$. We denote by $S_{Et}$ the big \'etale topos of $S$. Then $G$ represents a group object of $S_{Et}$ and $B_G$ is the classifying topos of $G$. The cohomology groups $H^*(B_G,\mathcal{A})$, where $\mathcal{A}$ is an abelian object of $B_G$ (e.g. an abelian group scheme over $S$ endowed with a $G$-action) is the \'etale cohomology of the $S$-group scheme $G$.

\subsubsection{The local section site.}

For $G$ any locally compact topological group, we denote by $B_{Top^{lc}}G$ the category
of $G$-equivariant locally compact topological spaces endowed with the local section topology $\mathcal{J}_{ls}$ (see
\cite{li04} section 1). The Yoneda functor yields a canonical fully faithful functor
$$B_{Top^{lc}}G\longrightarrow B_G.$$
Then one can show that the local section topology $\mathcal{J}_{ls}$ on $B_{Top^{lc}}G$ is the topology induced by the canonical topology of $B_G$. Moreover $B_{Top^{lc}}G$ is a generating family of $B_G$. It follows that the morphism
$$B_G\longrightarrow\widetilde{(B_{Top^{lc}}G,\mathcal{J}_{ls})}$$
is an equivalence. In other words the site $(B_{Top^{lc}}G,\mathcal{J}_{ls})$ is a site for the classifying topos $B_G$ (see \cite{flach06-2} for more details).

\subsubsection{The classifying topos of a strict topological pro-group.}

A \emph{locally compact topological pro-group} $\underline{G}$ is a pro-object in the category of locally compact topological groups, i.e. a functor $I^{op}\rightarrow Gr(Top^{lc})$, where $I$ is a filtered category and $Gr(Top^{lc})$ is the category of locally compact topological groups. A locally compact topological pro-group $\underline{G}$ is said to be \emph{strict} if the transition maps $G_j\rightarrow G_i$ have local sections. We define the limit of $\underline{G}$ in the 2-category of topoi as follows.
\begin{definition}
The classifying topos of a strict topological pro-group $\underline{G}$ is defined as
$$B_{\underline{G}}:=\underleftarrow{lim}_{I}\, B_{G_i},$$
where the the projective limit is computed in the 2-category of topoi.
\end{definition}

\subsubsection{}In order to ease the notations, we will simply denote by $Top$ the category of locally compact spaces. For any locally compact group $G$, we denote by $B_{Top}G$ the category of locally compact spaces endowed with a continuous $G$-action.

\subsection{Fiber products of topoi}
The class of topoi forms a 2-category. In particular, $\underline{Homtop}\,(\mathcal{E},\mathcal{F})$ is a category for any of topoi $\mathcal{E}$ and $\mathcal{F}$. If $f,g:\mathcal{E}\rightrightarrows\mathcal{F}$ are two objects of $\underline{Homtop}\,(\mathcal{E},\mathcal{F})$, then a morphism $\sigma:f\rightarrow g$ is a natural transformation $\sigma:f_*\rightarrow g_*$. Consider now two morphisms of topoi with the same target $f:\mathcal{E}\rightarrow\mathcal{S}$ and $g:\mathcal{F}\rightarrow\mathcal{S}$. For any topos $\mathcal{G}$, we define the category $$\underline{Homtop}\,(\mathcal{G},\mathcal{E})\times_{\underline{Homtop}\,(\mathcal{G},\mathcal{S})}
\underline{Homtop}\,(\mathcal{G},\mathcal{F})$$
whose objects are given by triples of the form $(a,b,\alpha)$, where $a$ and $b$ are objects of $\underline{Homtop}\,(\mathcal{G},\mathcal{E})$ and $\underline{Homtop}\,(\mathcal{G},\mathcal{F})$ respectively, and
$$\alpha:f\circ a\cong g\circ b$$
is an isomorphism in the category $\underline{Homtop}\,(\mathcal{G},\mathcal{S})$.

A fiber product  $\mathcal{E}\times_\mathcal{S}\mathcal{F}$ in the 2-category of topoi is a topos endowed with canonical projections $p_1:\mathcal{E}\times_\mathcal{S}\mathcal{F}\rightarrow\mathcal{E}$, $p_2:\mathcal{E}\times_\mathcal{S}\mathcal{F}\rightarrow\mathcal{F}$ and an isomorphism $\alpha:f\circ p_1\cong g\circ p_2$ satisfying the following universal condition.
For any topos $\mathcal{G}$ the natural functor
$$\label{equi-fiber-prod-topoi}
\appl{\underline{Homtop}\,(\mathcal{G},\mathcal{E}\times_\mathcal{S}\mathcal{F})}
{\underline{Homtop}\,(\mathcal{G},\mathcal{E})\times_{\underline{Homtop}\,(\mathcal{G},\mathcal{S})}
\underline{Homtop}\,(\mathcal{G},\mathcal{F})}{d}{(p_1\circ d,p_2\circ d,\alpha\circ d_*)}$$
is an equivalence. It is known that fiber products of topoi always exist (see \cite{illusie09} for example). The universal condition implies that such a fiber product is unique up to equivalence.
A product of topoi is a fiber product over the final topos $$\mathcal{E}\times\mathcal{F}=\mathcal{E}\times_{\underline{Set}}\mathcal{F}.$$ A square of topoi
\begin{equation*}\begin{CD}
@.\mathcal{E'} @>>> \mathcal{S}'\\
@. @VVV @VVV @.\\
@.\mathcal{E} @>>>\mathcal{S} @.
\end{CD}\end{equation*}
is said to be a \emph{pull-back} if it is commutative and if the morphism
$$\mathcal{E}'\longrightarrow\mathcal{E}\times_\mathcal{S}\mathcal{S}',$$
given by the universal condition for the fiber product, is an equivalence.
The following examples will be used in this paper. Let $f:\mathcal{E}\rightarrow\mathcal{S}$ be a morphism of topoi. For any object $X$ of $\mathcal{S}$, the commutative diagram
\begin{equation}\begin{CD}\label{pull-back-localization}
@. \mathcal{E}/f^*X@>>> \mathcal{S}/X\\
@. @VVV @VVV @.\\
@.\mathcal{E} @>f>>\mathcal{S} @.
\end{CD}\end{equation}
is a pull-back  (see \cite{sga4} IV Proposition 5.11). For any group-object $G$ in $\mathcal{S}$, the commutative diagram
\begin{equation}\begin{CD}\label{pull-back-classifying-topos}
@. B_{\mathcal{E}}(f^*G)@>>> B_{\mathcal{S}}(G)\\
@. @VVV @VVV @.\\
@.\mathcal{E} @>f>>\mathcal{S} @.
\end{CD}\end{equation}
is a pull-back. This follows from the fact that $B_{\mathcal{S}}(G)$ classifies $G$-torsors.

\section{The Weil-\'etale topos in characteristic p is a fiber product}\label{fibre}

For any scheme $Y$, we denote by $Y_{et}$ the (small) \'etale topos of $Y$, i.e. the category of sheaves of sets on the \'etale site on $Y$. Let $G$ be a discrete group acting on a scheme $Y$. An \'etale sheaf $\mathcal{F}$ on $Y$ is $G$-equivariant if $\mathcal{F}$ is endowed with a family of morphisms
$\{\varphi_g:g_*\mathcal{F}\rightarrow \mathcal{F};\, g\in G \}$
satisfying $\varphi_{1_G}=Id_{\mathcal{F}}$ and
$\varphi_{gh}=\varphi_g\circ g_*(\varphi_h)$, for any $g,h\in G$.
The category $\mathcal{S}(G;Y_{et})$ of $G$-equivariant \'etale
sheaves on $Y$ is a topos, as it follows from Giraud's axioms. The cohomology $H^*(\mathcal{S}(G;Y_{et}),\mathcal{A})$, for any $G$-equivariant abelian \'etale sheaf on $Y$, is the equivariant \'etale cohomology for the action $(G,Y)$.

An equivariant map of $G$-schemes $u:X\rightarrow Y$ induces a morphism
of topoi $\mathcal{S}(G;X_\et)\rightarrow\mathcal{S}(G;Y_\et)$.
Let $Y$ be a scheme separated and of finite type over a field $k$, let $\overline{k}/k$ be a separable closure and let $\mathcal{F}$ be an \'etale sheaf on $Y\otimes_k\overline{k}$. An action of the
Galois group $G_k$ on $\mathcal{F}$ is said to be \emph{continuous}
when the induced action of the profinite group $G_k$ on the discrete set
$\mathcal{F}(U\times_k\overline{k})$ is continuous, for any $U$
\'etale and quasi-compact over $Y$. It is well known that
the \'etale topos $Y_{\et}$ is equivalent to the category
$\mathcal{S}(G_k,\overline{Y}_{\et})$ of \'etale sheaves on
$\overline{Y}:=Y\otimes_k\overline{k}$ endowed with a continuous
action of the Galois group $G_k$.

Let $Y$ be a separated scheme of finite type over a finite field
$k=\mathbb{F}_q$. Let $\overline{k}/k$ be an algebraic closure. Let
$W_k$ and $G_k$ be the Weil group and the Galois group of $k$
respectively. The small classifying topos $B^{sm}_{W_{k}}$ is
defined as the category of $W_{k}$-sets, while $B^{sm}_{G_k}$ is the
category of continuous $G_{k}$-sets. We denote by $Y^{sm}_W$ the Weil-\'etale
topos of the scheme $Y$, which is defined as follows. We consider
the scheme $\overline{Y}=Y\otimes_k\overline{k}$ endowed with the
action of $W_k$. Then the Weil-\'etale topos $Y^{sm}_W$ is the topos  of
$W_k$-equivariant sheaves of sets on $\overline{Y}$. We have a
morphism
$$\gamma_Y:Y^{sm}_W:=\mathcal{S}(W_k,\overline{Y}_{\et})\longrightarrow\mathcal{S}(G_k,\overline{Y}_{\et})\cong Y_{et}.$$
Indeed, consider the functor $\gamma_Y^*$ which takes an \'etale
sheaf $\mathcal{F}$ on $\overline{Y}$ endowed with a continuous
$G_k$-action to the sheaf $\mathcal{F}$ endowed with the induced
$W_k$-action via the canonical map $W_k\rightarrow G_k$. Then
$\gamma_Y^*$ commutes with arbitrary inductive limits and with
projective limits. Hence $\gamma_Y^*$ is the inverse image of a
morphism of topoi $\gamma_Y$. This morphism has been defined and
studied by T. Geisser in \cite{geisser04}. Note that the Weil-\'etale topos of
$\Spec(k)$ is precisely $B^{sm}_{W_k}$ and that the \'etale topos
$\Spec(k)_{\et}$ is equivalent to $B^{sm}_{G_k}$. In this case the morphism $\gamma_k:=\alpha:B^{sm}_{W_{k}}\rightarrow B^{sm}_{G_k}$, from the Weil-\'etale topos of $\Spec(k)$ to its \'etale topos is the morphism induced by the canonical map $W_{k}\rightarrow G_k$.
The structure map $Y\rightarrow\Spec(k)$ gives a $W_k$-equivariant morphism of schemes
$\overline{Y}\rightarrow\Spec(\overline{k})$,
inducing in turn a morphism $Y^{sm}_W\rightarrow
B^{sm}_{W_k}$. This structure map also induces a morphism of \'etale topos $Y_\et\rightarrow B^{sm}_{G_k}$.
The diagram
\begin{equation}\begin{CD}\label{comm-square-PB}
@.Y^{sm}_W  @>\gamma_Y>>Y_{\et}\\
@. @VVV @VVV @.\\
@.B_{W_k} @>\alpha>>B_{G_k} @. {}
\end{CD}\end{equation}
is commutative, where $\alpha$ is induced by the morphism
$W_k\rightarrow G_k$. The aim of this section is to prove that the previous diagram is a pull-back of topoi. Our proof is based on a descent argument. We need some basic facts concerning truncated simplicial topoi. A truncated simplicial topos $\mathcal{S}_{\bullet}$ is given by the usual diagram
$$\mathcal{S}_2\rightrightarrows\rightarrow\mathcal{S}_1\rightrightarrows\leftarrow\mathcal{S}_0$$
Given such truncated simplicial topos $\mathcal{S}_{\bullet}$, we define the category $Desc(\mathcal{S}_{\bullet})$ of objects of $S_0$ endowed with a descent data. By \cite{moerdijk88}, the category $Desc(\mathcal{S}_{\bullet})$ is a topos. More precisely, $Desc(\mathcal{S}_{\bullet})$ is the inductive limit of the diagram $\mathcal{S}_{\bullet}$ in the 2-category of topoi.
The most simple example is the following. Let $\mathcal{S}$ be a topos and let $X$ be an object of $\mathcal{S}$. We consider the truncated simplicial topos
$$(\mathcal{S},X)_{\bullet}:\,\,\,\mathcal{S}/(X\times X\times X)\rightrightarrows\rightarrow\mathcal{S}/(X\times X)\rightrightarrows\leftarrow\mathcal{S}/X$$
where these morphisms of topoi are induced by the projections maps (of the form $X\times X\times X\rightarrow X\times X$ and $X\times X\rightarrow X$) and by the diagonal map $X\rightarrow X\times X$. It is well known that, if $X$ covers the final object of $\mathcal{S}$ (i.e. $X\rightarrow e_{\mathcal{S}}$ is epimorphic where $e_{\mathcal{S}}$ is the final object of $\mathcal{S}$), then the natural morphism
$$Desc(\mathcal{S},X)_{\bullet}\longrightarrow\mathcal{S}$$
is an equivalence (see \cite{fgiknv05} Chapter 4 Example 4.1). In other words $\mathcal{S}/X\rightarrow\mathcal{S}$ is an effective descent morphism for any $X$ covering the final object of $\mathcal{S}$.
\begin{lemma}\label{lem-descent}
Let $f:\mathcal{E}\rightarrow\mathcal{S}$ be a morphism of topoi and let $X$ be an object of $\mathcal{S}$ covering the final object. The morphism $f$ is an equivalence if and only if the induced morphism
$$f/X:\mathcal{E}/f^*X\longrightarrow\mathcal{S}/X$$
is an equivalence.
\end{lemma}
\begin{proof}
The condition is clearly necessary. Assume that $f/X$ is an equivalence.
We have $\mathcal{S}/(X\times X)=(\mathcal{S}/X)/(X\times X)$ and $\mathcal{S}/(X\times X\times X)=(\mathcal{S}/X)/(X\times X\times X)$, for any projection maps $X\times X\rightarrow X$ and  $X\times X\times X\rightarrow X$. Hence the triple of morphisms
$(f/X\times X\times X,f/X\times X,f/X)$ yields an equivalence of truncated simplicial topoi
$$f/:(\mathcal{E},f^*X)_{\bullet}\longrightarrow (\mathcal{S},X)_{\bullet}$$
This equivalence induces an equivalence of descent topoi
$$Desc(f/):Desc(\mathcal{E},f^*X)_{\bullet}\longrightarrow Desc(\mathcal{S},X)_{\bullet}$$
such that the following square is commutative
\begin{equation*}\begin{CD}
@.Desc(\mathcal{E},f^*X)_{\bullet}  @>Desc(f/)>>Desc(\mathcal{S},X)_{\bullet}\\
@. @VVV @VVV @.\\
@.\mathcal{E} @>f>>\mathcal{S} @. {}
\end{CD}\end{equation*}
This shows that $f$ is an equivalence since the vertical maps are equivalences.
\end{proof}

\begin{theorem}
Let $Y$ be a  scheme separated and of finite type over a finite field $k$. The canonical morphism
$$Y^{sm}_W\longrightarrow Y_{et}\times_{B^{sm}_{G_k}}B^{sm}_{W_k}$$
is an equivalence.
\end{theorem}
\begin{proof}
The morphism
$$f:Y^{sm}_W\longrightarrow Y_{et}\times_{B^{sm}_{G_k}}B^{sm}_{W_k}$$
is defined by the commutative square (\ref{comm-square-PB}). Let $p:Y_{et}\times_{B^{sm}_{G_k}}B^{sm}_{W_k}\rightarrow B^{sm}_{W_k}$ be the second projection. Consider the object $EW_k$ of $B^{sm}_{W_k}$ defined by the action of $W_k$ on itself by multiplication, and let $p^*EW_k$ be its pull-back in $Y_{et}\times_{B^{sm}_{G_k}}B^{sm}_{W_k}$. It is enough to show that the morphism
$$f/p^*EW_k:Y^{sm}_W/f^*p^*EW_k\longrightarrow (Y_{et}\times_{B^{sm}_{G_k}}B^{sm}_{W_k})/p^*EW_k$$
is an equivalence.

Recall that $Y^{sm}_W:=\mathcal{S}_{et}(W_k,\overline{Y})$ is the topos of $W_k$-equivariant \'etale sheaves on $\overline{Y}$. The object $f^*p^*EW_k$ is represented by the $W_k$-equivariant \'etale $\overline{Y}$-scheme $\coprod_{W_k}\overline{Y}\rightarrow\overline{Y}$. One has the following equivalences
$$Y^{sm}_W/f^*p^*EW_k=\mathcal{S}_{et}(W_k,\overline{Y})/\coprod_{W_k}y\overline{Y}\cong
\mathcal{S}_{et}(W_k,\coprod_{W_k}\overline{Y})\cong \overline{Y}_{et}.$$
Consider now the localization $(Y_{et}\times_{B^{sm}_{G_k}}B^{sm}_{W_k})/p^*EW_k$.
We have the following canonical equivalences:
\begin{align}
(Y_{et}\times_{B^{sm}_{G_k}}B^{sm}_{W_k})/p^*EW_k
\label{equi1}&\cong Y_{et}\times_{B^{sm}_{G_k}}B^{sm}_{W_k}\times_{B^{sm}_{W_k}}\underline{Set}\\
\label{equi2}&\cong Y_{et}\times_{B^{sm}_{G_k}}\underline{Set}\\
\label{equi3}&\cong \underleftarrow{lim}\,(Y_{et}\times_{B^{sm}_{G_{k'/k}}}\underline{Set})\\
\label{equi3.5}&\cong \underleftarrow{lim}\,(Y_{et}\times_{B^{sm}_{G_{k'/k}}}(B^{sm}_{G_{k'/k}}/EG_{k'/k}))\\
\label{equi4}&\cong \underleftarrow{lim}\,(Y_{et}/Y')\\
\label{equi5}&\cong \underleftarrow{lim}\,Y'_{et}\\
\label{equi6}&\cong (\underleftarrow{lim}\,Y')_{et}= \overline{Y}_{et}
\end{align}
Indeed, (\ref{equi1}) follows from the canonical equivalence $B^{sm}_{W_k}/EW_k\cong\underline{Set}$.  The inverse limit in  (\ref{equi3}) is taken over the Galois extensions $k'/k$. Using the natural equivalence $$B^{sm}_{G_k}\cong\underleftarrow{lim}\,B^{sm}_{G(k'/k)}$$
(\ref{equi3}) follow from the universal property of limits of topoi. For (\ref{equi3.5}) we use again $$B^{sm}_{G(k'/k)}/EG(k'/k)\cong\underline{Set}.$$ Then (\ref{equi4}) follows from the fact that the inverse image of $EG(k'/k)$ in the \'etale topos $Y_{et}$ is the sheaf represented by the \'etale $Y$-scheme $Y':=Y\times_kk'$. Then (\ref{equi5}) is given by (\cite{sga4} III Proposition 5.4), and (\ref{equi6}) is given by (\cite{morin} Lemma 8.3), since the schemes $Y'$ are all quasi-compact and quasi-separated. We obtain a commutative square
\begin{equation}\begin{CD}
@.\overline{Y}_{et}  @>Id>>\overline{Y}_{et}\\
@. @VVV @VVV @.\\
@.Y^{sm}_W/f^*p^*EW_k@>f/p^*EW_k>>(Y_{et}\times_{B^{sm}_{G_k}}B^{sm}_{W_k})/p^*EW_k @. {}
\end{CD}\end{equation}
where the vertical maps are the equivalences defined above. It follows that $f/p^*EW_k$ is an equivalence, and so is $f$ by Lemma \ref{lem-descent}.
\end{proof}

\begin{corollary}\label{cor-fiber-product}
There is a canonical equivalence
$$Y_{et}\times_{B^{sm}_{G_k}}B_{W_k}\cong Y^{sm}_W\times\T$$
\end{corollary}
\begin{proof}
The Weil group $W_k$ is a group of the final topos $\underline{Set}$. If
$u:\T\rightarrow \underline{Set}$ denotes the unique map, then
$u^*W_k$ is the group object of $\T$ represented by the discrete group
$W_k$. Hence one has (see the pull-back diagram (\ref{pull-back-classifying-topos})):
$$B^{sm}_{W_k}\times\T:=B_{\underline{Set}}(W_k)\times\T\cong B_{\T}(yW_k)=:B_{W_k}.$$
The previous theorem therefore yields
$$Y_{et}\times_{B^{sm}_{G_k}}B_{W_k}\cong Y_{et}\times_{B^{sm}_{G_k}}B^{sm}_{W_k}\times\T
\cong Y^{sm}_W\times\T.$$
\end{proof}
\begin{definition}
We define the \emph{big Weil-\'etale topos} of $Y$ as the fiber product
$$Y_W:=Y_{et}\times_{B^{sm}_{G_k}}B_{W_k}\cong Y^{sm}_{W}\times\T.$$
\end{definition}

\begin{corollary}\label{big-small}
Let $p_1:Y_W\rightarrow Y_W^{sm}$ and $p_2:Y_W\rightarrow\T$ be the projections. Then
for any abelian object $\mathcal{A}'$ of $Y_W$, one has
$$H^n(Y_W,\mathcal{A}')\cong H^n(Y^{sm}_W,p_{1*}\mathcal{A}').$$
If $\mathcal{A}$ is an abelian object of $\mathcal{T}$, then
$$H^n(Y_W,p_2^*\mathcal{A})\cong H^n(Y_W^{sm},\mathcal{A}(*)).$$
\end{corollary}
\begin{proof}
This follows from Corollary \ref{cohomology-basechange-overT}, using the equivalence $Y_W\cong Y^{sm}_{W}\times\T$.
\end{proof}

Define the sheaf $\tr$ on $Y_W$ as $p_2^*(y\br)$, where $y\br$ is the object of $\T$ represented by the standard topological group $\br$. Then we have canonical isomorphisms
$$H^n(Y_W,\tr)\cong H^n(Y_W^{sm},\br)\mbox{ and }H^n(Y_W,\bz)\cong H^n(Y_W^{sm},\bz)$$
as it follows from the previous corollary.

\begin{corollary}\label{cor-fiberproduct-classtopoi}
Let $\alpha:\mathcal{G}\rightarrow\mathcal{H}$ and  $\beta:\mathcal{G}'\rightarrow\mathcal{H}$ be two homomorphisms of group objects in a topos $\mathcal{S}$. If $\alpha$ is an epimorphism then the natural morphism
$$f:B_{\mathcal{G}\times_{\mathcal{H}}\mathcal{G}'}\longrightarrow B_{\mathcal{G}}\times_{B_{\mathcal{H}}}B_{\mathcal{G}'}$$
is an equivalence.
\end{corollary}
\begin{proof}
Let $e_{\mathcal{S}}$ be the final object in $\mathcal{S}$. The unique map $\mathcal{G}'\rightarrow e_{\mathcal{S}}$ is epimorphic, since the unit of $\mathcal{G}'$ yields a section $e_{\mathcal{S}}\rightarrow\mathcal{G}'$.
Therefore, the morphism $E\mathcal{G}'\rightarrow e_{\mathcal{G}'}$ in $B_{\mathcal{G}'}$, where $e_{\mathcal{G}'}$ is the final object of $B_{\mathcal{G}'}$, is epimorphic.  We denote the second projection by
$$p:B_{\mathcal{G}}\times_{B_{\mathcal{H}}}B_{\mathcal{G}'}\longrightarrow B_{\mathcal{G}'}$$
Let $\mathcal{K}$ be the kernel of $\alpha$, so that $\mathcal{G}/\mathcal{K}\cong\mathcal{H}$. On the one hand, we have the following canonical equivalences:
\begin{align*}
(B_{\mathcal{G}}\times_{B_{\mathcal{H}}}B_{\mathcal{G}'})/p^*E\mathcal{G}'
&\cong B_{\mathcal{G}}\times_{B_{\mathcal{H}}}(B_{\mathcal{G}'}/p^*E\mathcal{G}')\\
&\cong B_{\mathcal{G}}\times_{B_{\mathcal{H}}}\mathcal{S}\\
&\cong B_{\mathcal{G}}\times_{B_{\mathcal{H}}}(B_{\mathcal{H}}/E\mathcal{H})\\
&\cong B_{\mathcal{G}}/\alpha^*E\mathcal{H}\\
&\cong B_{\mathcal{G}}/(\mathcal{G}/\mathcal{K})\\
&\cong B_{\mathcal{K}}
\end{align*}
Here $\mathcal{G}/\mathcal{K}$ is endowed with its natural $\mathcal{G}$-action. The second, the third and the last equivalences are given by (\cite{sga4} IV.5.8), and the fourth equivalence is given by the pull-back diagram (\ref{pull-back-localization}).

On the other hand, we have an exact sequence of group objects in $\mathcal{S}$
$$1\rightarrow\mathcal{K}\rightarrow\mathcal{G}\times_{\mathcal{H}}\mathcal{G}'\rightarrow\mathcal{G}'\rightarrow1.$$
Indeed, the kernel of $\mathcal{G}\times_{\mathcal{H}}\mathcal{G}'\rightarrow\mathcal{G}'$ is given by
$$\mathcal{G}\times_{\mathcal{H}}\mathcal{G}'\times_{\mathcal{G}'}e_{\mathcal{S}}
=\mathcal{G}\times_{\mathcal{H}}e_{\mathcal{S}}=\mathcal{K}.$$
Moreover, $\mathcal{G}\times_{\mathcal{H}}\mathcal{G}'\rightarrow\mathcal{G}'$ is epimorphic, since epimorphisms are universal in a topos. We obtain
\begin{align*}
B_{\mathcal{G}\times_{\mathcal{H}}\mathcal{G}'}/f^*p^*E\mathcal{G}'
&=B_{\mathcal{G}\times_{\mathcal{H}}\mathcal{G}'}/(\mathcal{G}\times_{\mathcal{H}}\mathcal{G}'/\mathcal{K})\\
&= B_{\mathcal{K}}
\end{align*}
and we have a commutative square
\begin{equation}\begin{CD}
@. B_{\mathcal{K}}  @>Id>> B_{\mathcal{K}}\\
@. @VVV @VVV @.\\
@.B_{\mathcal{G}\times_{\mathcal{H}}\mathcal{G}'}/f^*p^*E\mathcal{G}'@>f/p^*E\mathcal{G}'>>(B_{\mathcal{G}}\times_{B_{\mathcal{H}}}B_{\mathcal{G}'})/p^*E\mathcal{G}' @. {}
\end{CD}\end{equation}
where the vertical maps are the equivalences defined above. Hence $f/p^*E\mathcal{G}'$ is an equivalence. By Lemma \ref{lem-descent}, $f$ is an equivalence as well, since $E\mathcal{G}'\rightarrow e_{\mathcal{G}'}$ is epimorphic.
\end{proof}

\begin{corollary}\label{cor-fiberproduct-classtopoi-topgrps}
Let $\alpha:G\rightarrow{H}$ and  $\beta:G'\rightarrow{H}$ be two morphisms of locally compact topological groups. If $\alpha$ has local sections then the natural morphism
$$f:B_{G\times_{H}G'}\longrightarrow B_{G}\times_{B_{H}}B_{G'}$$
is an equivalence.
\end{corollary}
\begin{proof}
Since $\alpha:G\rightarrow{H}$ has local sections, the induced morphism $y(G)\rightarrow y(H)$ is an epimorphism in $\mathcal{T}$. Hence the result follows from Corollary \ref{cor-fiberproduct-classtopoi}.
\end{proof}

\section{Artin-Verdier \'etale topos of an arithmetic scheme}\label{sect-AVetaletopos}

Let $\X$ be a scheme separated and of finite type over $\Spec(\bz)$.
We denote by $\X^{an}$ the complex analytic variety associated to
$\X\otimes_{\bz}{\bc}$, endowed with the standard complex topology.
The Galois group $G_{\br}$ of $\br$ acts on $\X^{an}$. The quotient
space $\X_{\infty}:=\X^{an}/G_{\br}$ is endowed with the quotient
topology. We consider the pair
$$\overline{\X}:=(\X,\X_{\infty}).$$
As a set, $\overline{\X}$ is the disjoint union $\X\coprod\X_{\infty}$. The Zariski topology on $\overline{\X}$ is defined as follows. An open subset $(\mathcal{U},D)$ of $\overline{\X}$ is given by a Zariski open subscheme $\mathcal{U}\subset\X$ and an open subspace $D\subset\mathcal{U}_{\infty}$ for the complex topology. We define the category $Et_{\overline{\X}}$ of \'etale
$\overline{\X}$-schemes as follows. An \'etale
$\overline{\X}$-scheme is an arrow
$f:(\mathcal{U},D)\rightarrow(\X,\X_{\infty})$, where
$\mathcal{U}\rightarrow\X$ is an \'etale morphism in the usual sense and $D$ is
an open subset of $\mathcal{U}_{\infty}$. The map
$f_{\infty}:D\rightarrow\X_{\infty}$ is supposed to be unramified in
the  sense that $f_{\infty}(d)\in\mathcal{X}(\br)$ if and only if
$d\in D\cap\mathcal{U}(\br)$. An \'etale
$\overline{\X}$-scheme $\overline{\mathcal{U}}$ is said to be connected (respectively irreducible) if it is connected (respectively irreducible) as a topological space. A morphism $(\mathcal{U},D)\rightarrow
(\mathcal{U}',D')$ in the category $Et_{\overline{\X}}$ is given by
a morphism of \'etale $\X$-schemes
$\mathcal{U}\rightarrow\mathcal{U}'$ inducing a map $D\rightarrow
D'$. The \'etale topology $\mathcal{J}_{et}$ on the category
$Et_{\overline{\X}}$ is the topology generated by the pretopology
for which a covering family is a surjective family. The
Artin-Verdier \'etale site is left exact.

\begin{definition}
The \emph{Artin-Verdier \'etale topos} of $\overline{\X}$ is the category
of sheaves of sets on the Artin-Verdier \'etale site:
$$\overline{\X}_{et}:=(\widetilde{Et_{\overline{\X}},\mathcal{J}_{et}}).$$
\end{definition}

The object $y\X:=y(\X,\emptyset)$ is a subobject of the final object
$y\overline{\X}$ of $\overline{\X}_{et}$. This yields an open
subtopos
$$\overline{\X}_{et}/y(\X,\emptyset)\hookrightarrow \overline{\X}_{et}.$$
We have the following canonical identifications (see \cite{sga4} III Proposition 5.4):
$$\overline{\X}_{et}/y(\X,\emptyset)\cong\widetilde{(Et_{\overline{\mathcal{X}}}/_{(\X,\emptyset)},\mathcal{J}_{ind})}
\cong(\widetilde{Et_{\X},\mathcal{J}_{et}})=\X_{et}$$ where
$\X_{et}$ is the usual \'etale topos of $\X$, and
$\mathcal{J}_{ind}$ is the topology on
$Et_{\overline{\mathcal{X}}}/(\X,\emptyset)$ induced by
$\mathcal{J}_{et}$ on $Et_{\overline{\mathcal{X}}}$ via the
forgetful functor
$Et_{\overline{\mathcal{X}}}/(\X,\emptyset)\rightarrow
Et_{\overline{\mathcal{X}}}$. We thus obtain an open embedding
$$\phi:{\X}_{et}\hookrightarrow\overline{\X}_{et}.$$
Let $Sh(\X_{\infty})$ be the topos of sheaves of sets on the
topological space $\X_{\infty}$, i.e. the category of \'etal\'e
spaces on $\X_{\infty}$. We consider $Sh(\X_{\infty})$ as a site
endowed with the canonical topology $\mathcal{J}_{can}$. There is a
morphism of left exact sites
$$
\fonc{u_{\infty}^*}{(Et_{\overline{\mathcal{X}}},\mathcal{J}_{et})}{(Sh(\X_{\infty}),\mathcal{J}_{can})}{(\mathcal{U},D)}{D\rightarrow\X_{\infty}}
$$
The resulting morphism of topoi
$$u_{\infty}:Sh(\X_{\infty})\longrightarrow\overline{\X}_{et}$$
is precisely the closed complement of the open subtopos
${\X}_{et}\hookrightarrow\overline{\X}_{et}$ defined above, i.e. we have the following result.
\begin{prop}\label{closed/open-decomp-etale}
There is an open-closed decomposition of topoi
$$\varphi:\X_{et}\longrightarrow\overline{\X}_{et}\longleftarrow Sh(\X_{\infty}):u_{\infty}$$
\end{prop}
The gluing functor $u_{\infty}^*\phi_*$ can be made more explicit as follows. There
is a canonical morphism of topoi
$$\alpha:Sh(G_{\br},\X^{an})\longrightarrow\X_{et}$$
where $Sh(G_{\br},\X^{an})$ is the topos of $G_{\br}$-equivariant
sheaves on the topological space $\X^{an}$, i.e. the category of
$G_{\br}$-equivariant \'etal\'e spaces on $\X^{an}$. The map
$\alpha$ is defined by the morphism of left exact sites which takes
an \'etale $\X$-scheme $\mathcal{U}$ to the $G_{\br}$-equivariant
\'etal\'e space $\mathcal{U}^{an}$ over $\X^{an}$ (note that
$\mathcal{U}^{an}\rightarrow\mathcal{X}^{an}$ is a
$G_{\br}$-equivariant local homeomorphism since the morphism
$\mathcal{U}\otimes_{\bz}{\bc}\rightarrow\mathcal{X}\otimes_{\bz}{\bc}$
is \'etale and compatible with complex conjugation).

The quotient map $\X^{an}\rightarrow \X^{an}/G_{\br}$ yields another
morphism of topoi
$$(\pi^*,\pi_*^{G_{\br}}):Sh(G_{\br},\X^{an})\longrightarrow
Sh(\X_{\infty}).$$ Here $\pi:\X^{an}\rightarrow\X_{\infty}$ is the
quotient map, $\pi^*$ is the usual inverse image and
$\pi_*^{G_{\br}}\mathcal{F}$ is the $G_{\br}$-invariant subsheaf of
the the direct image $\pi_*\mathcal{F}$, i.e. for any open $U\subset\X_{\infty}$ one has $$\pi_*^{G_{\br}}\mathcal{F}(U):=\mathcal{F}(\pi^{-1}U)^{G_{\br}}.$$
Then we have an
identification of functors
$$
u_{\infty}^*\varphi_*\cong\pi_*^{G_{\br}}\alpha^*:\X_{et}\longrightarrow
Sh(\X_{\infty})
$$
Let us consider the category
$(Sh(\X_{\infty})\,,\X_{et},\pi_*^{G_{\br}}\alpha^*)$ defined in
(\cite{sga4} IV.9.5.1) by Artin gluing. Recall that an object of this category is a triple $(F,E,\sigma)$, where $F$ is an object of $Sh(\X_{\infty})$, $E$ is an object of $\X_{et}$ and $\sigma$ is a map
$\sigma:{F}\rightarrow\pi_*^{G_{\br}}\alpha^*{E}$.

\begin{corollary}\label{Cor-Artin-Gluing}
The category $\overline{\X}_{et}$ is canonically equivalent to
$(Sh(\X_{\infty})\,,\X_{et},\pi_*^{G_{\br}}\alpha^*)$.
\end{corollary}
\begin{proof}
There is a canonical functor
$$
\fonc{\Phi}{\overline{\X}_{et}}{(Sh(\X_{\infty})\,,\X_{et},u_{\infty}^*\varphi_*)}
{\mathcal{F}}{(u_{\infty}^*\mathcal{F},\varphi^*\mathcal{F},\sigma)}
$$
where the morphism $$\sigma:u_{\infty}^*\mathcal{F}\longrightarrow
u_{\infty}^*\varphi_*(\varphi^*\mathcal{F})$$
is induced by the adjunction transformation $Id\rightarrow\varphi_*\varphi^*$. By (\cite{sga4} IV.9.5.4.a) the functor
$\Phi$ is an equivalence of categories, since
$u_{\infty}:Sh(\X_{\infty})\hookrightarrow\overline{\X}_{et}$ is the closed
complement of the open embedding $\phi:\X_{et}\hookrightarrow
\overline{\X}_{et}$. Hence the result follows from the isomorphism
$$u_{\infty}^*\varphi_*\cong\pi_*^{G_{\br}}\alpha^*.$$
\end{proof}

\begin{corollary}\label{pull-ba}
We denote by $\infty$ the archimedean place of $\bq$. The commutative square
\begin{equation}\begin{CD}
@. Sh(\X_{\infty})@>>>Sh(\infty)\\
@. @VVV @VVV @.\\
@.\overline{\X}_{et} @>f>>\overline{\Spec(\bz)}_{et} @. {}
\end{CD}\end{equation}
is a pull-back, where $Sh(\infty)=\underline{Set}$ is the category of sheaves on the one point space.
\end{corollary}
\begin{proof}
The map $\overline{\X}\rightarrow \overline{\Spec(\bz)}$ induces a morphism of \'etale topos $f$.
Consider the open embedding $\Spec(\bz)_{et}\hookrightarrow\overline{\Spec(\bz)}_{et}$. Its inverse image under the map $f$ is
$\X_{et}\hookrightarrow\overline{\X}_{et}$. The result therefore follows from Proposition \ref{closed/open-decomp-etale} and (\cite{sga4} IV Corollaire 9.4.3).
\end{proof}
\begin{prop}\label{pull-bu}
For any prime number $p$, we have a pull-back
\begin{equation}\begin{CD}
@.(\X\otimes_{\bz}\mathbb{F}_p)_{et}@>>>\Spec(\mathbb{F}_p)_{et}\\
@. @VVV @VVV @.\\
@.\overline{\X}_{et} @>f>>\overline{\Spec(\bz)}_{et} @. {}
\end{CD}\end{equation}
\end{prop}
\begin{proof}
The morphism $\Spec(\mathbb{F}_p)_{et}\rightarrow\overline{\Spec(\bz)}_{et}$ factors through $\Spec(\bz)_{et}$, hence one is reduced to show that $$(\X\otimes_{\bz}\mathbb{F}_p)_{et}\cong{\X}_{et}\times_{{\Spec(\bz)}_{et}}\Spec(\mathbb{F}_p)_{et}.$$
This follows from (\cite{sga4} IV Corollaire 9.4.3) since $\Spec(\mathbb{F}_p)\rightarrow\Spec(\bz)$ is a closed embedding.
\end{proof}

\subsection{\'Etale cohomology with compact support}
It follows from Corollary \ref{Cor-Artin-Gluing} that we have the usual sequences of adjoint functors (see \cite{sga4} IV.14)
$$\varphi_!,\,\varphi^*,\,\varphi_* \mbox{ and } u_{\infty}^*,\,u_{\infty*},\,u_{\infty}^!.$$
between the categories of abelian sheaves on $\overline{\X}_{et}$,
$\X_{et}$ and $Sh(\X_{\infty})$. In particular $u_{\infty*}$
is exact and $\varphi^*$ preserves injective objects since
$\varphi_!$ is exact. For any abelian sheaf $\mathcal{A}$ on
$\overline{\X}_{et}$, one has the exact sequence
\begin{equation}
0\rightarrow\varphi_!\varphi^*\mathcal{A}\rightarrow\mathcal{A}\rightarrow
u_{\infty*}u_{\infty}^*\mathcal{A}\rightarrow0,
\end{equation}
where the morphisms are given by adjunction.

\begin{definition}
Assume that $\X$ is proper over $\Spec(\mathbb{Z})$ and let $\mathcal{A}$ be an
abelian sheaf on $\X_{et}$. The \'etale cohomology with compact
support is defined by
$$H_c^n(\X_{et},\mathcal{A}):=H^n(\overline{\X}_{et},\varphi_!\mathcal{A}).$$
\end{definition}

\begin{prop}\label{prop-cpctspp-etale-coh}
Let $\X$ be a flat proper scheme over $\Spec(\bz)$. Assume that $\X$ is normal and connected. Then the $\br$-vector space $H_c^n(\X_{et},\mathbb{R})$ is finite dimensional, zero for $n$ large, and we have
\begin{align*}
H_c^n(\X_{et},\mathbb{R})&=0\mbox{ for $n=0$}\\
&= H^0(\X_{\infty},\mathbb{R})/\br \mbox{ for $n=1$}\\
&= H^{n-1}(\X_{\infty},\mathbb{R}) \mbox{ for $n\geq2$}\\
\end{align*}
\end{prop}
\begin{proof}
The exact sequence
$$0\rightarrow\varphi_!\mathbb{R}\rightarrow\mathbb{R}\rightarrow
u_{\infty*}\mathbb{R}\rightarrow0$$
and the fact that $u_{\infty*}$ is exact give a long exact sequence
$$0\rightarrow H_c^0(\X_{et},\mathbb{R})\rightarrow H^0(\overline{\X}_{et},\mathbb{R})\rightarrow H^0(\X_{\infty},\mathbb{R}) \rightarrow H_c^1(\X_{et},\mathbb{R})\rightarrow H^1(\overline{\X}_{et},\mathbb{R})\rightarrow$$
The inclusion of the generic point of $\X$ yields a morphism of topoi
$$\eta:(\Spec\,K(\X))_{et}\longrightarrow\overline{\X}_{et}.$$
We have immediately $R^n\eta_*\mathbb{R}=0$ for any $n\geq1$ since Galois cohomology is torsion and $\mathbb{R}$ is uniquely divisible. Moreover, we have $\eta_*\mathbb{R}=\mathbb{R}$. Indeed, the scheme $\X$ is normal hence the set of connected components of an \'etale $\overline{\X}$-scheme $\overline{\mathcal{U}}$ is in 1-1 correspondence with the set of connected components of $\mathcal{U}\times_{\X}\Spec\,K(\X)\,$, i.e. one has
$$\pi_0(\mathcal{U}\times_{\X}\Spec\,K(\X))=\pi_0(\mathcal{U})=\pi_0(\overline{\mathcal{U}}).$$
Therefore the Leray spectral sequence associated to the morphism $\eta$ gives
$$H^n(\overline{\X}_{et},\mathbb{R})=H^n(G_{K(\X)},\mathbb{R}).$$
We obtain $H^0(\overline{\X}_{et},\mathbb{R})=\mathbb{R}$ and $H^n(\overline{\X}_{et},\mathbb{R})=0$ for $n\geq1$, and the result follows.
\end{proof}

\section{The definition of $\overline{\Spec(\co_F)}_W$}\label{wdef}
Let $F$ be a number field. We consider the Arakelov compactification $\bar{X}=(\Spec\,\mathcal{O}_F,X_{\infty})$ of $X=\Spec\,\mathcal{O}_F$, where $X_{\infty}$ is the finite set of
archimedean places of $F$. Note that this is a special case of the previous section, since $X_{\infty}$ is the quotient of $X\otimes\mathbb{C}$ by complex conjugation. We endow $\bar{X}$ with the Zariski topology described previously.

If $\bar{F}/F$ is an algebraic closure and $\bar{F}/K/F$ a finite Galois extension then the relative Weil group $W_{K/F}$ is defined by the extension of topological groups
$$1\rightarrow C_K \rightarrow W_{K/F}\rightarrow G_{K/F}\rightarrow 1$$
corresponding to the fundamental class in $H^2(G_{K/F},C_K)$ given by class field theory, where $C_K$ is the id\`ele class group of $K$. A Weil group of $F$ is then defined as the projective limit $W_F:=\underleftarrow{lim}\,W_{K/F}$, computed in the category of topological groups. Alternatively, let $\bar{F}/K/F$ be a finite Galois extension and let $S$ be a finite set of places of $F$ containing all the places which ramify in $K$. Then the fundamental class in $$H^2(G_{K/F},C_K)\cong H^2(G_{K/F},C_{K,S})$$
yields a group extension
$$1\rightarrow C_{K,S} \rightarrow W_{K/F,S}\rightarrow G_{K/F}\rightarrow 1$$
where $C_{K,S}$ is the $S$-id\`ele class group of $K$. Then one has (see \cite{li04})
$$W_F:=\underleftarrow{lim}\,W_{K/F}=\underleftarrow{lim}\,W_{K/F,S}$$

\subsection{The Weil-\'etale topos}\label{section-Lichtenbaum-topos}
We choose an algebraic closure $\bar{F}/F$ and a Weil group $W_F$.
For any place $v$ of $F$, we choose an algebraic closure $\bar{F_v}/F_v$  and an embedding $\bar{F}\rightarrow \bar{F_v}$ over $F$.
Then we choose a local Weil group $W_{F_v}$ and a Weil map $\theta_v:W_{F_v}\rightarrow W_F$ compatible with $\bar{F}\rightarrow \bar{F_v}$.

Let $W^1_{F_v}$ be the maximal compact subgroup of $W_{F_v}$. For any closed point $v\in\bar{X}$ (ultrametric or archimedean), we define the Weil group of "the residue field at $v$" as follows
$$W_{k(v)}:=W_{F_v}/W^1_{F_v},$$
while the Galois group of the residue field at $v$ can be defined as
$G_{k(v)}:=G_{F_v}/{I_v}$. Note that $G_{k(v)}$ is the trivial group for $v$ archimedean. For any $v$, the Weil map $W_{F_v}\rightarrow G_{F_v}$ chosen above induces a morphism $W_{k(v)}\rightarrow G_{k(v)}$.
Finally, we denote by
$$q_v:W_{F_v}\longrightarrow W_{F_v}/W^1_{F_v}=:W_{k(v)}$$ the map from
the local Weil group $W_{F_v}$ to the Weil group of the residue
field at $v\in\bar{X}$.

\begin{definition}
Let $T_{\bar{X}}$ be the category of objects $(Z_0,Z_v,f_v)$ defined
as follows. The topological space $Z_0$ is endowed with a continuous
$W_F$-action. For any place $v$ of $F$, $Z_v$ is a topological space
endowed with a continuous $W_{k(v)}$-action. The continuous map
$f_v:Z_v\rightarrow Z_0$ is $W_{F_v}$-equivariant, when $Z_v$ and
$Z_0$ are seen as $W_{{F_v}}$-spaces via the maps
$\theta_v:W_{F_v}\rightarrow W_{F}$ and $q_v:W_{F_v}\rightarrow
W_{k(v)}$. Moreover, we require the following facts.
\begin{itemize}
\item The spaces $Z_v$ are \emph{locally compact}.
\item The map $f_v$ is an \emph{homeomorphism for almost all
places $v$ of $F$ and a continuous injective map for all places}.
\item The action of $W_F$ on $Z_0$ factors through $W_{K/F}$, for
some finite Galois subextension $\bar{F}/K/F$.
\end{itemize}
A \emph{morphism} $$\phi:(Z_0,Z_v,f_v)\rightarrow(Z'_0,Z'_v,f'_v)$$
in the category $T_{\bar{X}}$ is a continuous $W_F$-equivariant map
$\phi:Z_0\rightarrow Z'_0$ \emph{inducing} a continuous map
$\phi_v:Z_v\rightarrow Z_v$ for any place $v$. Then $\phi_v$ is $W_{k(v)}$-equivariant.

The category $T_{\bar{X}}$ is endowed with the local section
topology $\mathcal{J}_{ls}$, i.e. the topology generated by the
pretopology for which a family
$$\{\varphi_i:(Z_{i,0},Z_{i,v},f_{i,v})\rightarrow
(Z_0,Z_v,f_v),\,i\in I\}$$ is a covering family if
$\coprod_{i\in I} Z_{i,v}\rightarrow Z_v$ has local continuous sections, for any
place $v$.
\end{definition}

\begin{lemma}
The site $(T_{\bar{X}},\mathcal{J}_{ls})$ is left exact.
\end{lemma}
\begin{proof}
The category $T_{\bar{X}}$ has fiber products and a final object, hence finite projective limits are representable in $T_{\bar{X}}$. It
remains to show that $\mathcal{J}_{ls}$ is subcanonical. This follows easily from the fact
that, for any topological group $G$, the local section topology
$\mathcal{J}_{ls}$ on $B_{Top}{G}$ coincides with the open cover topology $\mathcal{J}_{op}$, which is
subcanonical.
\end{proof}
\begin{definition}\label{ofwdef}
We define the \emph{Weil-\'etale topos} $\bar{X}_{W}$ as the topos
of sheaves of sets on the site defined above:
$$\bar{X}_{W}:=\widetilde{(T_{\bar{X}},\mathcal{J}_{ls})}.$$
\end{definition}

\begin{prop}
We have a morphism of topoi
$$j:B_{W_{F}}\longrightarrow\bar{X}_W.$$
\end{prop}
\begin{proof}
By \cite{flach06-2} Corollary 2, the site
$(B_{Top}{W_{F}},\mathcal{J}_{ls})$ is a site for the classifying
topos $B_{W_{F}}$ is defined as the topos of
$y(W_{F})$-objects of $\mathcal{T}$. By \cite{flach06-2} Corollary 2, the site
$(B_{Top}{W_{F}},\mathcal{J}_{ls})$ is a site for $B_{W_{F}}$. The morphism of left
exact sites
$$\fonc{j^*}{(T_{\bar{X}},\mathcal{J}_{ls})}{(B_{Top}{W_{F}},\mathcal{J}_{ls})}{(Z_0,Z_v,f_v)}{Z_0}$$
induces the morphism of topoi $j$.
\end{proof}

\begin{prop}\label{generic-XW}
The morphism of topoi $j:B_{W_{F}}\rightarrow\bar{X}_W$ factors
through
$$B_{\underline{W}_{K/F,S}}:=\underleftarrow{lim}B_{{W}_{K/F,S}}.$$
The induced morphism
$i_0:B_{\underline{W}_{K/F,S}}\rightarrow\bar{X}_W$ is an embedding.
\end{prop}
\begin{proof}
Let $(Z_0,Z_v,f_v)$ be an object of $T_{\bar{X}}$. The action of
$W_F$ on $Z_0$ factors through $W_{K/F}$, for some finite Galois
sub-extension $\bar{F}/K/F$. Since $W_{K/F}$ and $Z_0$ are both locally compact,
this action is given by a continuous morphism
$$\rho:W_{K/F}\longrightarrow Aut(Z_0)$$
where $Aut(Z_0)$ is the homeomorphism group of $Z_0$ endowed with
the compact-open topology. The kernel of $\rho$ is a closed
normal subgroup of $W_{K/F}$ since $Aut(Z_0)$ is Hausdorff. Moreover, there exists an open subset $V$
of $\bar{X}$ such that $f_v:Z_v\rightarrow Z_0$ is an isomorphism of
$W_{F_v}$-spaces for any $v\in V$. Let $\widetilde{W}^1_{F_v}$ denotes
the image of the continuous morphism
$$W^1_{F_v}\longrightarrow W_{F_v}\longrightarrow W_{K/F},$$ endowed with the induced topology.
Then $\widetilde{W}^1_{F_v}$ is
in the kernel of $\rho$ for any $v\in V$. Let $N_V$ be the closed
normal subgroup of $W_{K/F}$ generated by the subgroups
$\widetilde{W}^1_{F_v}$ for any $v\in V$. Then $\rho$ induces a
continuous morphism
$$W_{K/F}/N_V\longrightarrow Aut(Z_0).$$
We choose $V$ small enough so that $K/F$ is unramified above $V$ and
we set $S:=\bar{X}-V$. Then we have $$N_V=\prod_{w\mid v,\,v\in
V}\mathcal{O}_{K_w}^{\times}\subseteq C_K \subseteq W_{K/F}\mbox{
and }W_{K/F}/N_V= W_{K/F,S}.$$
Hence the action of $W_F$ on $Z_0$ factors through $W_{K/F,S}$, for
some finite Galois sub-extension $\bar{F}/K/F$ and some finite set
$S$ of places of $F$ containing all the places which ramify in $K$.
The morphism of left exact sites
$$\fonc{j^*}{(T_{\bar{X}},\mathcal{J}_{ls})}{(B_{Top}{W_{F}},\mathcal{J}_{ls})}{(Z_0,Z_v,f_v)}{Z_0}.$$
therefore induces a morphism
$$\fonc{i_0^*}{(T_{\bar{X}},\mathcal{J}_{ls})}{(\underrightarrow{lim}\,B_{Top}{W_{K/F,S}},\mathcal{J}_{ls})}{(Z_0,Z_v,f_v)}{Z_0}$$
where $(\underrightarrow{lim}\,B_{Top}{W_{K/F,S}},\mathcal{J}_{ls})$ is the direct limit site. More precisely, $\underrightarrow{lim}\,B_{Top}{W_{K/F,S}}$ is the direct limit category endowed with the coarsest topology $\mathcal{J}$ such that the functors $B_{Top}{W_{K/F,S}}\rightarrow \underrightarrow{lim}\,B_{Top}{W_{K/F,S}}$ are all continuous, when $B_{Top}{W_{K/F,S}}$ is endowed with the local section topology. One can identify $\underrightarrow{lim}\,B_{Top}{W_{K/F,S}}$ with a full subcategory of $B_{Top}{W_{F}}$ and $\mathcal{J}$ with the local section topology $\mathcal{J}_{ls}$. By (\cite{sga4} VI.8.2.3), the direct limit site $(\underrightarrow{lim}\,B_{Top}{W_{K/F,S}},\mathcal{J}_{ls})$ is a site for the projective limit topos $B_{\underline{W}_{K/F,S}}$. We obtain a morphism of topoi
$$i_0:B_{\underline{W}_{K/F,S}}\longrightarrow\bar{X}_W.$$
It remains to show that this morphism is an embedding. Let
$\mathcal{F}$ be an object of $B_{\underline{W}_{K/F,S}}$. Then
$i_0^*i_{0*}\mathcal{F}$ is the sheaf associated with the presheaf
$$
\fonc{i_{0}^pi_{0*}\mathcal{F}}{\underrightarrow{lim}\,B_{Top}{W_{K/F,S}}}{\underline{Set}}
{Z}{\displaystyle{\lim_{Z\rightarrow
i_{0}^*(Y_0,Y_v,f_v)}}i_{0*}\mathcal{F}(Y_0,Y_v,f_v)}
$$
where the direct limit is taken over the category of arrows
$Z\rightarrow i_{0}^*(Y_0,Y_v,f_v)$. For any object $Z$ of $\underrightarrow{lim}\,B_{Top}{W_{K/F,S}}$, there exist a finite Galois extension $K_Z/F$ and a finite set $S_Z$ such that $Z$ is an object of $B_{Top}{W_{K_Z/F,S_Z}}$.
Consider the cofinal subcategory $I_Z$ of the category of arrows defined above, where $I_Z$ consists of the following objects. For any finite set $S$ of places of $F$ such that $S_Z\subseteq S$, we consider the map
$Z\rightarrow i_{0}^*(Z_0,Z_v,f_v)$ with $Z_0=Z$ as a $W_F$-space,
$Z_v=Z$ as a $W_{k(v)}$-space for any place $v$ not in  $S$ and $Z_v=\emptyset$ for any $v\in S$. We thus have
$$\displaystyle{\lim_{Z\rightarrow
i_{0}^*(Y_0,Y_v,f_v)}}i_{0*}\mathcal{F}(Y_0,Y_v,f_v)=\displaystyle{\lim_{I_Z}}\,\,i_{0*}\mathcal{F}(Z_0,Z_v,f_v)=\mathcal{F}(Z).$$
Hence $i_{0}^pi_{0*}\mathcal{F}$ is already a sheaf and we have
$$i_{0}^*i_{0*}\mathcal{F}=i_{0}^pi_{0*}\mathcal{F}=\mathcal{F}.$$
This shows that $i_{0*}$ is fully faithful, i.e. $i_{0}$ is an
embedding.

\end{proof}

\begin{prop}\label{map-to-BR}
There is canonical morphism of topoi
$$\mathfrak{f}:\bar{X}_{W}\longrightarrow B_{\mathbb{R}}.$$
\end{prop}
\begin{proof}
We have a commutative diagram of topological groups
\begin{equation}\begin{CD}
@. W_{F_v} @>>> W_{k(v)}\\
@. @VVV @VVV @.\\
@.W_F @>>>\mathbb{R} @. {}
\end{CD}\end{equation}
where $W_F\rightarrow\mathbb{R}$ is defined as the composition
$$W_{F}\rightarrow W_F^{ab}\cong{C_F}\rightarrow\mathbb{R}_+^{\times}\cong\mathbb{R}.$$
Hence there is a morphism of left exact sites
\begin{equation}\label{morphism-sites-toT}
\fonc{\mathfrak{f}^*}{(B_{Top}\mathbb{R},\mathcal{J}_{ls})}{(T_{\bar{X}},\mathcal{J}_{ls})}{Z}{(Z,Z,Id_Z)}
\end{equation}
where $Z$ is seen as $W_F$-space (respectively a $W_{k(v)}$-space) via the canonical morphism $W_F\rightarrow\mathbb{R}$ (respectively via $W_{k(v)}\rightarrow\mathbb{R}$). The result follows.
\end{proof}

\subsection{The morphism from the Weil-\'etale topos to the Artin-Verdier \'etale topos}

Let $\bar{X}$ be the Arakelov compactification of the number ring
$\mathcal{O}_F$. We consider below the Artin-Verdier \'etale site $(Et_{\bar{X}};\mathcal{J}_{et})$ and the Artin-Verdier \'etale topos $\bar{X}_{et}$ of the arithmetic curve $\bar{X}$.

\begin{prop}\label{prop-morph-sites-etale-loc-sections}
There exists a morphism of left exact sites
$$\fonc{\gamma^*}{(Et_{\bar{X}};\mathcal{J}_{et})}{(T_{\bar{X}};\mathcal{J}_{ls})}{\bar{U}}{(U_0,U_v,f_v)}.$$
The underlying functor $\gamma^*$ is fully faithful and its essential image consists exactly of objects
$(U_0,U_v,f_v)$ of $T_{\bar{X}}$ where $U_0$ is a finite $W_F$-set.
\end{prop}
This result is a reformation of \cite{morin} Proposition 4.61
and \cite{morin} Proposition 4.62. We give below a sketch of the
proof.
\begin{proof}
For any \'etale $\bar{X}$-scheme $\bar{U}$, we define an object
$\gamma^*(\bar{U})=(U_0,U_v,f_v)$ of $T_{\bar{X}}$ as follows.  The scheme $\bar{U}\times_{\bar{X}}\Spec\,F$
is the spectrum of an \'etale $F$-algebra and the
Grothendieck-Galois theory shows that this $F$-algebra is uniquely
determined by the finite $G_F$-set
$$U_0:=Hom_{\Spec\, F}(\Spec\, \bar{F},\bar{U}\times_{\bar{X}}\Spec\,F)=Hom_{\bar{X}}(\Spec\,\overline{F},\bar{U}).$$
Let $v$ be an ultrametric place of $F$. The maximal unramified
sub-extension of the algebraic closure $\bar{F_v}/F_v$ yields an algebraic closure of the residue
field $\overline{k(v)}/k(v)$. The scheme
$\bar{U}\times_{\bar{X}}\Spec\,k(v)$ is the spectrum of an \'etale
$k(v)$-algebra, corresponding to the finite $G_{k(v)}$-set
$$U_v:=Hom_{\Spec\, k(v)}(\Spec\, \overline{k(v)},\bar{U}\times_{\bar{X}}\Spec\,k(v))=Hom_{\bar{X}}(\Spec\,\overline{k(v)},\bar{U})$$
The chosen $F$-embedding $\bar{F}\rightarrow\bar{F_v}$ induces a $G_{F_v}$-equivariant map
$$f_v:U_v\longrightarrow U_0.$$
Consider now an archimedean place $v$ of $F$. Define
$$U_v:=Hom_{\bar{X}}(v,\bar{U})=\bar{U}\times_{\bar{X}}v$$
where the map $v\rightarrow\bar{X}$ is the closed embedding corresponding to the archimedean place $v$ of $F$. As above, the $F$-embedding $\bar{F}\rightarrow\bar{F_v}$ induces a $G_{F_v}$-equivariant map
$$f_v:U_v\longrightarrow U_0.$$
For any place $v$ of $F$, the set $U_v$ is viewed as a
$W_{k(v)}$-topological space via the morphism $W_{k(v)}\rightarrow
G_{k(v)}$. Respectively, $U_0$ is viewed as a $W_{F}$-topological
space via $W_{F}\rightarrow G_{F}$. Then the map $f_v$ defined above
is $W_{F_v}$-equivariant. We check that the map
$f_v$ is bijective for almost all valuations and injective for all
valuations (see \cite{morin} Proposition 4.62). We obtain a functor
$$\gamma^*:Et_{\bar{X}}\longrightarrow T_{\bar{X}}.$$
This functor is left exact by construction (i.e. it preserves the final objects and
fiber product) and continuous (i.e. it preserves covering families) since a surjective map of discrete sets is a local section cover.
The last claim of the proposition follows from Galois theory.
\end{proof}

\begin{corollary}
There is a morphism of topoi
$\gamma:\bar{X}_W\rightarrow\bar{X}_{et}$.
\end{corollary}
\begin{proof}
This follows from the fact that a morphism
of left exact sites induces a morphism of topoi.
\end{proof}

\subsection{Structure of $\bar{X}_W$ at the closed points.}

Let $v$ be a place of $F$. We consider the Weil group $W_{k(v)}$ and
the Galois group $G_{k(v)}$ of the residue field $k(v)$ at
$v\in\bar{X}$. Note that for $v$ archimedean one has
$W_{k(v)}\cong\mathbb{R}$ and $G_{k(v)}=\{1\}$. Consider the big
classifying topos $B_{W_{k(v)}}$, i.e. the category of
$y(W_{k(v)})$-objects in $\mathcal{T}$. We consider also the small
classifying topos $B^{sm}_{G_{k(v)}}$, which is defined as the
category of continuous $G_{k(v)}$-sets. The category of
locally compact $W_{k(v)}$-spaces $B_{Top}{W_{k(v)}}$ is endowed with
the local section topology $\mathcal{J}_{ls}$. Recall that the site
$(B_{Top}{W_{k(v)}},\mathcal{J}_{ls})$ is a site for the classifying
topos $B_{W_{k(v)}}$. We denote by $B_{fSets}G_{k(v)}$ the
category of finite $G_{k(v)}$-sets endowed with the canonical
topology $\mathcal{J}_{can}$. The site $(B_{fSets}G_{k(v)},\mathcal{J}_{can})$ is a site for the small
classifying topos $B^{sm}_{G_{k(v)}}$.

For any place $v$ of $F$, we have a morphism of left exact sites
$$\fonc{i_v^*}{(T_{\bar{X}},\mathcal{J}_{ls})}{(B_{Top}{W_{k(v)}},\mathcal{J}_{ls})}{(Z_0,Z_v,f_v)}{Z_v}$$
hence a morphism of topoi
$$i_v:B_{W_{k(v)}}\longrightarrow\bar{X}_W.$$
On the other hand one has morphism of topoi
$$u_v:B^{sm}_{G_{k(v)}}\longrightarrow\bar{X}_{et}$$
for any closed point $v$ of $\bar{X}$. For $v$ ultrametric, this morphism is induced by the closed embedding of schemes
$$\Spec\,k(v)\longrightarrow\bar{X}$$
since the \'etale topos of $\Spec\,k(v)$ is equivalent to the
category $B^{sm}_{G_{k(v)}}$ of continuous $G_{k(v)}$-sets. Note
that this equivalence is induced by the choice of an algebraic
closure of ${k(v)}$ made at the beginning of section \ref{section-Lichtenbaum-topos}. By Corollary \ref{closed/open-decomp-etale}, there is a closed embedding $$Sh(X_{\infty})=\coprod_{X_{\infty}}\underline{Set}\longrightarrow \bar{X}_{et}$$
which yields the closed embedding
$$u_v:B^{sm}_{G_{k(v)}}=\underline{Set}\longrightarrow\bar{X}_{et}$$
for any archimedean valuation $v$ of $F$. In both cases, we have a commutative diagram of left exact sites
\begin{equation*}\begin{CD}
@. (B_{Top}W_{k(v)},\mathcal{J}_{ls}) @<{\alpha^*_v}<<
(B_{fSets}{G_{k(v)}},\mathcal{J}_{can})\\
@. @AA{i^*_v}A @AA{u^*_v}A @.\\
@. (T_{\bar{X}},\mathcal{J}_{ls}) @<{\gamma^*}<<
(Et_{\bar{X}},\mathcal{J}_{et}) @. {}
\end{CD}\end{equation*}
where $u_v^*(\bar{U})$ is the finite $G_{k(v)}$-set
$Hom_{\bar{X}}(\Spec\,\overline{k(v)},\bar{U})$ (respectively $Hom_{\bar{X}}(v,\bar{U})$) for $v$ ultrametric (respectively archimedean). This commutative diagram of sites induces a commutative diagram of topoi.
\begin{theorem}\label{thm-pull-back}
For any closed point $v$ of $\bar{X}$, the following diagram is a
pull-back of topoi.
\begin{equation*}\begin{CD}
@.  B_{W_{k(v)}} @>{\alpha_v}>>
B^{sm}_{G_{k(v)}}\\
@. @VV{i_v}V @VV{u_v}V @.\\
@. \bar{X}_W @>{\gamma}>> \bar{X}_{et} @. {}
\end{CD}\end{equation*}
In particular, the morphism $i_v$ is a closed embedding.
\end{theorem}
\begin{proof}
We first prove a partial result.
\begin{lemma}\label{lemma-iv-embedding}
The morphism $i_v$ is an embedding, i.e. $i_{v*}$ is fully faithful.
\end{lemma}
\begin{proof}
We use below the fact that the \emph{full}
subcategory
$$W_{k(v)}\times Top\hookrightarrow B_{Top}W_{k(v)}$$
is a topologically generating subcategory of the site
$(B_{Top}W_{k(v)},\mathcal{J}_{ls})$. Here $W_{k(v)}\times Top$
consists in locally compact topological spaces of the form
$Z=W_{k(v)}\times T$ on which $W_{k(v)}$ acts by left
multiplication on the first factor. In particular, a sheaf
$\mathcal{F}$ of
$$B_{W_{k(v)}}=\widetilde{(B_{Top}W_{k(v)},\mathcal{J}_{ls})}$$
is completely determined by its values $\mathcal{F}(W_{k(v)}\times
T)$ on objects of $W_{k(v)}\times Top$.

Let $\mathcal{F}$ be an object of $B_{W_{k(v)}}$. Consider the
adjunction map
\begin{equation}\label{adjunction-iv}
i_v^*\circ i_{v*}\mathcal{F}\longrightarrow\mathcal{F}.
\end{equation}
The sheaf $i_v^*\circ i_{v*}\mathcal{F}$ is the sheaf on $(B_{Top}W_{k(v)},\mathcal{J}_{ls})$ associated to
the presheaf
$$Z\rightarrow \displaystyle{\lim_{Z\rightarrow i_v^*(Y_0,Y_w,f_w)}}i_{v*}\mathcal{F}(Y_0,Y_w,f_w)
=\displaystyle{\lim_{Z\rightarrow
i_v^*(Y_0,Y_w,f_w)}}\mathcal{F}(Y_v)$$ where the direct limit is
taken over the category of arrows $Z\rightarrow i_v^*(Y_0,Y_w,f_w)$
with $(Y_0,Y_w,f_w)$ an object of $T_{\bar{X}}$.

Let $\bar{F}/K/F$ be a finite Galois sub-extension, let $S$ be a
finite set of closed points of $\bar{X}$ such that $v\in S$ and let
$Z=W_{k(v)}\times T$ be an object of $W_{k(v)}\times Top$. Consider
the object of $T_{\bar{X}}$
$$\mathcal{Y}(K,S,Z)=(T_0,T_w,f_w)$$
defined as follows. We first define the topological space
$$T_0=W_{K/F,S}\times^{W_{F_v}}Z:=(W_{K/F,S}\times W_{k(v)}\times T)/{W_{F_v}}\cong(W_{K/F,S}/{{W}^1_{F_v}})\times T$$
endowed with its natural $W_F$-action. For any $w$ not in $S$, we consider
$T_w=W_{K/F,S}\times^{W_{F_v}}Z$ on which $W_{k(w)}$ acts via the
map
$$W_{k(w)}=W_{F_w}/W^1_{F_w}\longrightarrow W_{K/F,S}.$$
For any $w\in S$ such that $w\neq v$, we set $T_w=\emptyset$, and we
define $T_v=Z$. The map $f_w$ is the identity for any $w$ not in $S$
and $$f_v:Z\longrightarrow W_{K/F,S}\times^{W_{F_v}}Z$$ is the
canonical map. This map $f_v$ is continuous and injective. The image of ${W}^1_{F_v}$ in
$W_{K/F,S}$ is compact, and the spaces $T_0$ and $T_w$ are locally compact for any place $w$
so that $\mathcal{Y}(K,S,Z)$ is an object of $T_{\bar{X}}$.

On the one hand, the functor $Z\mapsto W_{K/F,S}\times^{W_{F_v}}Z$
is left adjoint to the forgetful functor
$B_{Top}W_{K/F,S}\rightarrow B_{Top}W_{F_v}$. On the other hand, for
any object $(Y_0,Y_w,f_w)$ of $T_{\bar{X}}$, the action of $W_F$ on
$Z_0$ factors through $W_{K/F,S}$ for some finite Galois extension
$K/F$ and some finite set $S$ of places of $F$. It follows that
$$\{\mathcal{Y}(K,S,Z),\mbox{ for $K/F$ Galois, }S\mbox{ finite }\}$$ yields a cofinal system in the category of arrows $Z\rightarrow i_v^*(Y_0,Y_w,f_w)$ considered above, for any fixed object $Z$ of $W_{k(v)}\times Top$. Hence $i_v^*\circ i_{v*}\mathcal{F}$ is the sheaf on $(B_{Top}W_{k(v)},\mathcal{J}_{ls})$
associated to the presheaf
$$
\appl{W_{k(v)}\times
Top}{\underline{Set}}{Z}{\displaystyle{\lim_{Z\rightarrow
i_v^*(Y_0,Y_w,f_w)}}\mathcal{F}(Y_v)=
\displaystyle{\lim_{Z\rightarrow
i_v^*\mathcal{Y}(K,S,Z)}}\mathcal{F}(Z)=\mathcal{F}(Z)}
$$
Since $W_{k(v)}\times Top$ is a topologically generating subcategory of $(B_{Top}W_{k(v)},\mathcal{J}_{ls})$, the sheaf on $B_{W_{k(v)}}$
associated to this presheaf is $\mathcal{F}$, and the adjunction morphism (\ref{adjunction-iv}) is an isomorphism. This shows that
$i_{v*}$ is fully faithful, i.e. $i_{v}$ is an embedding.
\end{proof}

\begin{definition}
Let $v$ be a closed point of $\bar{X}$. We consider the morphism $p_v:\T\rightarrow B_{W_{k(v)}}$ whose inverse image $p_v^*$ is the forgetful functor, and we denote by $i_{\overline{v}}$ the composite morphism
$$i_{\overline{v}}:=i_v\circ p_v:\T\longrightarrow B_{W_{k(v)}}\longrightarrow \bar{X}_W.$$
\end{definition}
For any object $Z=W_{k(v)}\times T$ of the full subcategory
$W_{k(v)}\times Top\hookrightarrow B_{Top}W_{k(v)}$ and for any sheaf $\mathcal{F}$ of $\bar{X}_W$, we have
\begin{equation}\label{i^p_v}
i^p_{v}\mathcal{F}(Z)=\displaystyle{\lim_{Z\rightarrow
i_v^*(Y_0,Y_w,f_w)}}\mathcal{F}(Y_0,Y_w,f_w)=
\displaystyle{\lim_{Z\rightarrow i_v^*\mathcal{Y}(K,S,Z)}}\mathcal{F}(\mathcal{Y}(K,S,Z))
\end{equation}
where we consider the pull-back presheaf $i^p_{v}\mathcal{F}$ on $B_{Top}W_{k(v)}$.
The morphism $p_v:\T\rightarrow B_{W_{k(v)}}$ is induced by the morphism of left exact sites given by the forgetful functor
$B_{Top}W_{k(v)}\rightarrow Top$. By adjunction, for any space $T$ of $Top$ and any presheaf $\mathcal{P}$ on $B_{Top}W_{k(v)}$ we have $$p_v^p\mathcal{P}(T)=\mathcal{P}(W_{k(v)}\times T).$$
Hence the isomorphism $i^p_{\overline{v}}\cong p^p_v\circ i^p_v$ gives
\begin{equation}\label{i^pbarv}
i^p_{\overline{v}}\mathcal{F}(T)=i^p_{v}\mathcal{F}(Z)
=\displaystyle{\lim_{Z\rightarrow
i_v^*\mathcal{Y}(K,S,Z)}}\mathcal{F}(\mathcal{Y}(K,S,Z))
\end{equation}
where $Z:=W_{k(v)}\times T$. We consider the category of compact spaces $Top^c$. The morphism of sites $(Top^c,\mathcal{J}_{op})\rightarrow (Top,\mathcal{J}_{op})$ induces an equivalence of topoi, hence one can restrict our attention to compact spaces. Let us show that $i^p_{\overline{v}}\mathcal{F}$ restricts to a sheaf on $(Top^c,\mathcal{J}_{op})$. Let $\{T_i\rightarrow T,\,i\in I\}$ be a covering family of $(Top^c,\mathcal{J}_{op})$, i.e. a local section cover of compact spaces. One can assume that $I$ is finite, since any covering family of $(Top^c,\mathcal{J}_{op})$ can be refined by a finite covering family.
For any $K/F$ and any $S$,
$$\{\mathcal{Y}(K,S,W_{k(v)}\times T_i)\rightarrow \mathcal{Y}(K,S,W_{k(v)}\times T)\}$$ is a covering family of
$(T_{\bar{X}},\mathcal{J}_{ls})$. Moreover the fiber product
$$\mathcal{Y}(K,S,W_{k(v)}\times T_i)\times_{ \mathcal{Y}(K,S,W_{k(v)}\times T)}\mathcal{Y}(K,S,W_{k(v)}\times T_j)$$
computed in the category $T_{\bar{X}}$, is isomorphic to $\mathcal{Y}(K,S,W_{k(v)}\times T_{ij})$, where $T_{ij}$ denotes $T_i\times_TT_i$. It follows that the diagram of sets
$$\mathcal{F}(\mathcal{Y}(K,S,W_{k(v)}\times T))\rightarrow \prod_i\mathcal{F}(\mathcal{Y}(K,S,W_{k(v)}\times T_i))
\rightrightarrows \prod_{i,j}\mathcal{F}(\mathcal{Y}(K,S,W_{k(v)}\times T_{ij}))$$
is exact. Passing to the inductive limit over $K$ and $S$, and using left exactness of filtered inductive limits (i.e. using the fact that filtered inductive limits commute with finite products and equalizers), we obtain an exact diagram of sets
$$i^p_{\overline{v}}\mathcal{F}(T)\rightarrow \prod_ii^p_{\overline{v}}\mathcal{F}(T_i)
\rightrightarrows \prod_{i,j}i^p_{\overline{v}}\mathcal{F}(T_{ij}),$$
as it follows from (\ref{i^pbarv}).
Hence $i^p_{\overline{v}}\mathcal{F}$ is a sheaf on $(Top^c,\mathcal{J}_{op})$. Therefore, for any compact space $T$, one has
\begin{equation}\label{i^*barv}
i^*_{\overline{v}}\mathcal{F}(T)=i^p_{\overline{v}}\mathcal{F}(T)=\displaystyle{\lim_{Z\rightarrow
i_v^*\mathcal{Y}(K,S,Z)}}\mathcal{F}(\mathcal{Y}(K,S,Z))
\end{equation}
where $Z=W_{k(v)}\times T$.

\begin{lemma}
The family of functors
$$\{i_{\overline{v}}^*:\bar{X}_W\rightarrow \T,\,v\in\bar{X}^0\}$$ is
conservative, where $\bar{X}^0$ is the set of closed points of $\bar{X}$.
\end{lemma}
\begin{proof}
Let $\mathcal{F}$ be an object of $\bar{X}_W$. We need to show that the adjunction map
\begin{equation}\label{inj-adj-map}
\mathcal{F}\longrightarrow \prod_{v\in\bar{X}^0}i_{\overline{v}*}i_{\overline{v}}^*\mathcal{F}.
\end{equation}
is injective. For any $(Z_0,Z_w,f_w)$ of $T_{\bar{X}}$, we have
$$\prod_{v\in\bar{X}^0}(i_{\overline{v}*}i_{\overline{v}}^*\mathcal{F})(Z_0,Z_w,f_w)=\prod_{v\in\bar{X}^0}i_{\overline{v}}^*\mathcal{F}(Z_v).$$
Note that, in the term on the right hand side of the equality above, $Z_v$ is considered as a topological space without any action. For any $v$, we choose a local section cover of the space $Z_v$:
$$\{T_{v,l}\hookrightarrow Z_v,\,l\in\Lambda_v\}$$ such that $T_{v,l}$ is a compact subspace of $Z_v$ for any index $l$. Such a local section cover exists since $Z_v$ is locally compact. The map
$$i_{\overline{v}}^*\mathcal{F}(Z_v)\longrightarrow\prod_{l\in\Lambda_v} i_{\overline{v}}^*\mathcal{F}(T_{v,l}).$$
is injective since $i_{\overline{v}}^*\mathcal{F}$ is a sheaf. It is therefore enough to show that the composite map
$$\kappa:\mathcal{F}(Z_0,Z_w,f_w)\longrightarrow\prod_{v\in\bar{X}^0}i_{\overline{v}}^*\mathcal{F}(Z_v)\longrightarrow \prod_{v\in\bar{X}^0,\,l\in\Lambda_v} i_{\overline{v}}^*\mathcal{F}(T_{v,l})$$
is injective. Let $\alpha,\beta\in\mathcal{F}(Z_0,Z_w,f_w)$ be two sections such that $\kappa(\alpha)=\kappa(\beta)$. For any pair $(v,l)$, we consider
$$\kappa_{v,l}:\mathcal{F}(Z_0,Z_w,f_w)\longrightarrow\prod_{v\in\bar{X}^0,\,l\in\Lambda_v} i_{\overline{v}}^*\mathcal{F}(T_{v,l})
\longrightarrow i_{\overline{v}}^*\mathcal{F}(T_{v,l}).$$
For any $(v,l)$, we have $\kappa_{v,l}(\alpha)=\kappa_{v,l}(\beta)$ and by (\ref{i^*barv})
$$i_{\overline{v}}^*\mathcal{F}(T_{v,l})=\underrightarrow{lim}\mathcal{F}(\mathcal{Y}(K,S,W_{k(v)}\times T_{v,l}))$$
where the direct limit is taken over the category of arrows
$$W_{k(v)}\times T_{v,l}\longrightarrow
i_v^*\mathcal{Y}(K,S,W_{k(v)}\times T_{v,l}).$$
The inclusion $T_{v,l}\subseteq Z_v$ gives a $W_{k(v)}$-equivariant continuous map $$W_{k(v)}\times T_{v,l}\longrightarrow
i_v^*(Z_0,Z_w,f_w)=Z_v.$$
Thus for any pair $(v,l)$, there is an object $\mathcal{Y}(K,S,W_{k(v)}\times T_{v,l})$ and a morphism $$\mathcal{Y}(K,S,W_{k(v)}\times T_{v,l})\longrightarrow (Z_0,Z_w,f_w)$$ in the category $T_{\bar{X}}$ inducing the previous map
$$W_{k(v)}\times T_{v,l}= i_v^*\mathcal{Y}(K,S,W_{k(v)}\times T_{v,l})\longrightarrow
i_v^*(Z_0,Z_w,f_w)=Z_v$$
and such that $\alpha_{\mid(v,l)}=\beta_{\mid(v,l)}$, where $\alpha_{\mid(v,l)}$ (respectively $\beta_{\mid(v,l)}$) denotes the restriction of $\alpha$ (respectively of $\beta$) to $\mathcal{Y}(K,S,W_{k(v)}\times T_{v,l})$. We obtain a local section cover
$$\{\mathcal{Y}(K,S,W_{k(v)}\times T_{v,l})\rightarrow (Z_0,Z_w,f_w),\,v\in\bar{X}^0,\,l\in\Lambda_v)\}$$
in the site $(T_{\bar{X}},\mathcal{J}_{ls})$ such that $\alpha_{\mid(v,l)}=\beta_{\mid(v,l)}$ for any $(v,l)$.
It follows that $\alpha=\beta$ since $\mathcal{F}$ is a sheaf. Hence $\kappa$ is injective and so is the adjunction map (\ref{inj-adj-map}).

\end{proof}
A morphism of topoi $f$ is said to be surjective if its inverse image functor $f^*$ is faithful.
\begin{corollary}
The following morphism is surjective:
$$(i_{v})_{v\in\bar{X}^0}:\,\coprod_{v\in\bar{X}^0}B_{W_{k(v)}}\longrightarrow\bar{X}_W.$$
\end{corollary}
\begin{proof}
The morphism of topoi
$$(i_{\overline{v}})_{v\in\bar{X}^0}:\,\coprod_{v\in\bar{X}^0}\T\longrightarrow\bar{X}_W$$
is surjective since its inverse image is faithful by the previous result. But $(i_{\overline{v}})_{v\in\bar{X}^0}$ factors through $(i_{v})_{v\in\bar{X}^0}$, hence $(i_{v})_{v\in\bar{X}^0}$ is surjective as well.
\end{proof}
\hspace{-0.4cm}\emph{Proof of Theorem \ref{thm-pull-back}.} Since the morphism $i_v$ is an embedding, we have in fact two embeddings of topoi
$$B_{W_{k(v)}}\longrightarrow\bar{X}_W\times_{\bar{X}_{et}}B^{sm}_{G_{k(v)}}\longrightarrow\bar{X}_{W}$$
where the fiber product
$\bar{X}_W\times_{\bar{X}_{et}}B^{sm}_{G_{k(v)}}$ is defined as the inverse image $\gamma^{-1}(B^{sm}_{G_{k(v)}})$ of the
closed sub-topos $B^{sm}_{G_{k(v)}}\hookrightarrow\bar{X}_{et}$ under the morphism $\gamma$ (see \cite{sga4} IV. Corollaire 9.4.3). Therefore $B_{W_{k(v)}}$ is equivalent
to a full subcategory of
$\bar{X}_W\times_{\bar{X}_{et}}B^{sm}_{G_{k(v)}}$. This
fiber product is the closed complement of the open subtopos
$\bar{Y}_W\hookrightarrow\bar{X}_W$ where $\bar{Y}:=\bar{X}-v$ (see
the next section for the definition of $\bar{Y}_W$). In other words, the strictly full subcategory
$\bar{X}_W\times_{\bar{X}_{et}}B^{sm}_{G_{k(v)}}$ of $\bar{X}_W$
consists in objects $\mathcal{G}$ such that $\mathcal{G}\times\gamma^*\bar{Y}$ is the final object of
$\bar{Y}_W$. It follows that
$$i_{v*}\mathcal{F}\times\gamma^*\bar{Y}$$
is the final object of $\bar{Y}_W$, for any object $\mathcal{F}$ of
$B_{W_{k(v)}}$.

We have to prove that $B_{W_{k(v)}}$ is in fact
equivalent to $\bar{X}_W\times_{\bar{X}_{et}}B^{sm}_{G_{k(v)}}$. Let $\mathcal{G}$ be an object of this fiber product, i.e. an object
of $\bar{X}_W$ such that $\mathcal{G}\times\gamma^*\bar{Y}$ is the
final object. Consider the adjunction map
$$\mathcal{G}\longrightarrow i_{v*}i_{v}^*\mathcal{G}.$$
If $w$ is a closed point of $\bar{X}$ such that $w\neq v$, then the
morphism $i_{w}$ factors through $\bar{Y}_W$:
$$i_w:B_{W_{k(w)}}\longrightarrow\bar{Y}_W\longrightarrow\bar{X}_W.$$
We denote by $i_{\bar{Y},w}:B_{W_{k(v)}}\rightarrow\bar{Y}_W$ the
induced map. Hence
$$i_{w}^*\mathcal{G}=i_{\bar{Y},w}^*(\mathcal{G}\times\bar{Y})$$ is the
final object of $B_{W_{k(w)}}$, since $\mathcal{G}\times\bar{Y}$ is
the final object of $\bar{Y}_W$ and $i_{\bar{Y},w}^*$ is left exact.
On the other hand
$$i_{w}^*i_{v*}i_{v}^*\mathcal{G}=i_{\bar{Y},w}^*(i_{v*}i_{v}^*\mathcal{G}\times\bar{Y})$$
is the final object of $B_{W_{k(w)}}$, since
$i_{v*}i_{v}^*\mathcal{G}\times\bar{Y}$ is the final object of
$\bar{Y}_W$. Hence the map
$$i_w^*(\mathcal{G})\longrightarrow i_w^*(i_{v*}i_{v}^*\mathcal{G})$$
is an isomorphism for any closed point $w\neq v$ of $\bar{X}$. Suppose now that
$w=v$. Then the map
$$i_v^*(\mathcal{G})\longrightarrow i_v^*(i_{v*}i_{v}^*\mathcal{G})=(i_v^*i_{v*})i_{v}^*\mathcal{G}=i_{v}^*\mathcal{G}$$
is an isomorphism by Lemma \ref{lemma-iv-embedding}. Hence the morphism
$$i_w^*(\mathcal{G})\longrightarrow
i_w^*(i_{v*}i_{v}^*\mathcal{G})$$ induced by the adjunction map
$\mathcal{G}\rightarrow i_{v*}i_{v}^*\mathcal{G}$ is an isomorphism
for any closed point $w$ of $\bar{X}$. Since the family of functors
$$\{i_w^*:\bar{X}_W\rightarrow B_{W_{k(w)}},\,w\in\bar{X}\}$$ is
conservative, the adjunction map $\mathcal{G}\rightarrow
i_{v*}i_{v}^*\mathcal{G}$ is an isomorphism for any object
$\mathcal{G}$ of $\gamma^{-1}(B^{sm}_{G_{k(v)}})$. Hence any object
of $\gamma^{-1}(B^{sm}_{G_{k(v)}})$ is in the essential image of
$i_{v*}$. This shows that the morphism
$$B_{W_{k(v)}}\longrightarrow\bar{X}_W\times_{\bar{X}_{et}}B^{sm}_{G_{k(v)}}$$
is an equivalence (this is a connected embedding). Theorem
\ref{thm-pull-back} follows.
\end{proof}

We consider the morphism
$$\bar{X}_W=\overline{\Spec(\mathcal{O}_F)}_W\longrightarrow\overline{\Spec(\bz)}_W$$
induced by the map $\overline{\Spec(\mathcal{O}_F)}\rightarrow\overline{\Spec(\bz)}$.
\begin{prop}\label{prop-fiberproduct-makes-sense}
The canonical morphism
$$\delta_{\bar{X}}:\bar{X}_W\longrightarrow \bar{X}_{et}\times_{\overline{\Spec(\bz)}_\et}\overline{\Spec(\bz)}_W$$
is an equivalence.
\end{prop}
\begin{proof}
Let $\bar{X}'$ be the open subscheme of $\bar{X}$ consisting of the points of $\bar{X}$ where the map $\bar{X}\rightarrow \overline{\Spec(\bz)}$ is \'etale. Let $Y\rightarrow\bar{X}$ be the complementary reduced closed subscheme.\\

\textbf{(i) The morphism $\delta_{\bar{X}}$ is an equivalence over $\bar{X}'$ and over $Y$.}
The canonical morphism
$$\bar{X}'_W\rightarrow \bar{X}'_{et}\times_{\overline{\Spec(\bz)}_\et}\overline{\Spec(\bz)}_W$$
is an equivalence. Indeed, the morphism $\bar{X}'\rightarrow \overline{\Spec(\bz)}$ is \'etale hence we have
\begin{align*}
\bar{X'}_{et}\times_{\overline{\Spec(\bz)}_\et}\overline{\Spec(\bz)}_W
&\cong(\overline{\Spec(\bz)}_\et/\bar{X}')\times_{\overline{\Spec(\bz)}_\et}\overline{\Spec(\bz)}_W\\
&\cong\overline{\Spec(\bz)}_\et\times_{\overline{\Spec(\bz)}_\et}(\overline{\Spec(\bz)}_W/\gamma^*\bar{X}')\\
&\cong\overline{\Spec(\bz)}_W/\gamma^*\bar{X}'\\
&\cong\bar{X}'_W
\end{align*}
Let $Y'$ be the image of $Y$ in $\overline{\Spec(\bz)}$, such that $Y'\times_{\overline{\Spec(\bz)}}\Spec(\bz)$ is given with a structure of reduced closed subscheme of $\Spec(\bz)$. The morphism of \'etale topoi
$Y_{et}\rightarrow\overline{\Spec(\bz)}_\et$ factors through $Y'_{et}$. It follows from Theorem \ref{thm-pull-back}
that one has
\begin{align*}
Y_{et}\times_{\overline{\Spec(\bz)}_\et}\overline{\Spec(\bz)}_W&\cong Y_{et}\times_{Y'_\et}Y'_{et}\times_{\overline{\Spec(\bz)}_\et}\overline{\Spec(\bz)}_W\\
&\cong
Y_{et}\times_{Y'_\et}Y'_{W}.
\end{align*}
We have the following equivalences $Y_{et}\cong\coprod_{v\in Y}B^{sm}_{G_{k(v)}}$,
$Y'_\et\cong\coprod_{p\in Y'}B^{sm}_{G_{k(p)}}$ and $Y'_W:=\coprod_{p\in Y'}B_{W_{k(p)}}$. We obtain
\begin{align*}
Y_{et}\times_{\overline{\Spec(\bz)}_\et}\overline{\Spec(\bz)}_W&\cong Y_{et}\times_{Y'_\et}Y'_W\\
&\cong
\coprod_{v\in Y}B^{sm}_{G_{k(v)}}\times_{(\coprod_{p\in Y'}B^{sm}_{G_{k(p)}})}\coprod_{p\in Y'}B_{W_{k(p)}}\\
&\cong
\coprod_{v\in Y}(B^{sm}_{G_{k(v)}}\times_{B^{sm}_{G_{k(p)}}}B_{W_{k(p)}})\\
&\cong
\coprod_{v\in Y}B_{W_{k(v)}}=Y_{W}
\end{align*}
In view of the pull-back squarre (\ref{pull-back-localization}), the last equivalence above follows from the fact that
$$B^{sm}_{G_{k(v)}}\cong  B^{sm}_{G_{k(p)}}/(G_{k(p)}/G_{k(v)}) \longrightarrow B^{sm}_{G_{k(p)}}$$
is a localization morphism.

\textbf{(ii) The natural transformation $t$ between the glueing functors.}
The previous step (i) shows that there is an open-closed decomposition of topoi
$$\mathrm{j}:\bar{X}'_W\rightarrow\bar{X}_{et}\times_{\overline{\Spec(\bz)}_\et}\overline{\Spec(\bz)}_W\leftarrow Y_W:\mathrm{i}$$
By Theorem \ref{thm-pull-back}, we have another open-closed decomposition
$$j:\bar{X}'_W\rightarrow\bar{X}_{W}\leftarrow Y_W:i$$
The glueing functors associated to these open-closed decompositions are given by
$\mathrm{i}^*\mathrm{j}_*$ and $i^*j_*$. The map  $\bar{X}_W\rightarrow \bar{X}_{et}\times_{\overline{\Spec(\bz)}_\et}\overline{\Spec(\bz)}_W$ induces a natural transformation
\begin{equation}\label{nat-transformation}
t:\mathrm{i}^*\mathrm{j}_*\longrightarrow i^*j_*.
\end{equation}
Indeed, the following commutative diagram
\begin{equation}\begin{CD}\label{diagram-for-fiberproduct}
@.\bar{X}'_W @>j>> \bar{X}_W @<i<<  Y_W \\
@. @VVId V @VV\delta_{\bar{X}}V @VVId V @.\\
@.\bar{X}'_W @>\mathrm{j}>> \bar{X}_{et}\times_{\overline{\Spec(\bz)}_\et}\overline{\Spec(\bz)}_W @ <\mathrm{i}<< Y_W @. {}
\end{CD}\end{equation}
gives $\delta_{\bar{X}}\circ i=\mathrm{i}$ and $\delta_{\bar{X}}\circ j=\mathrm{j}$. Then the natural transformation (\ref{nat-transformation}) is induced by the adjunction transformation
$\delta_{\bar{X}}^*\delta_{\bar{X}*}\rightarrow Id$ as follows :
$$\mathrm{i}^*\mathrm{j}_*\cong i^*\delta_{\bar{X}}^*\delta_{\bar{X}*}j_*\longrightarrow i^*j_*.$$\\

\textbf{(iii) The glueing functors are naturally isomorphic.} Since the disjoint sum topos
$Y_{W}=\coprod_{v\in S}B_{W_{k(v)}}$ is given by the direct product of the categories $B_{W_{k(v)}}$, it is enough to show that the natural transformation
\begin{equation}\label{nat-transformation-at-v}
\mathrm{i}_{v}^*\mathrm{j}_*\longrightarrow i_{v}^*j_*
\end{equation}
is an isomorphism for any $v\in Y$.

Let $\mathcal{F}$ be an object of $\bar{X}'_W$. The sheaf $i_v^*\circ j_*\mathcal{F}$ (respectively $\mathrm{i}_v^*\circ \mathrm{j}_*\mathcal{F}$) is the sheaf on $(W_{k(v)}\times
Top,\mathcal{J}_{ls})$
associated to the presheaf $i_v^p\circ j_*\mathcal{F}$ (respectively to the presheaf $\mathrm{i}_v^p\circ \mathrm{j}_*\mathcal{F}$). Recall that $W_{k(v)}\times Top$ is a topologically generating subcategory of $(B_{Top}W_{k(v)},\mathcal{J}_{ls})$. It is therefore enough to show that the natural map
\begin{equation}\label{iso-presheaves}
\mathrm{i}_v^p\circ \mathrm{j}_*\mathcal{F}\rightarrow i_v^p\circ j_*\mathcal{F},
\end{equation} of
presheaves on $W_{k(v)}\times
Top$, is an isomorphism.

On the one hand, for any object $\mathcal{F}$ of $\bar{X}'_W$ we have
\begin{align}\label{first-ind-lim}
i_{v}^pj_*\mathcal{F}(W_{k(v)}\times T)&=\displaystyle{\lim_{W_{k(v)}\times T\rightarrow
i_v^*(Y_0,Y_w,f_w)}}\mathcal{F}((Y_0,Y_w,f_w)\times\bar{X}')\\
\label{presk-first-ind-lim}&= \underrightarrow{lim}_{_{K/F,S}}\mathcal{F}(\mathcal{Y}(K/F,S,W_{k(v)}\times T)\times \bar{X})
\end{align}
where (\ref{presk-first-ind-lim}) is given by (\ref{i^p_v}). See the proof of Lemma \ref{lemma-iv-embedding} for the definition of $\mathcal{Y}(K/F,S,W_{k(v)}\times T)$. On the other hand, for any object $\mathcal{F}$ of $\bar{X}'_W$ we have
\begin{align*}
\mathrm{i}_{v}^p\mathrm{j}_*\mathcal{F}(W_{k(v)}\times T)&=\underrightarrow{lim}\,\mathrm{j}_*\mathcal{F}((Z_0,Z_w,f_w)\rightarrow V\leftarrow U)\\
&=\underrightarrow{lim}\,\mathcal{F}((Z_0,Z_w,f_w)\times_{V}U\times\bar{X}')
\end{align*}
where the direct limit is taken over the category of arrows
\begin{equation}\label{one-arrow}
W_{k(v)}\times T\rightarrow
\mathrm{i}_v^*((Z_0,Z_w,f_w)\rightarrow V\leftarrow U)=Z_p\times_{V_p}U_v.
\end{equation}
Here, $((Z_0,Z_w,f_w)\rightarrow V\leftarrow U)$ is an object of the fiber product site $\mathcal{C}_{\bar{X}}$, i.e. $(Z_0,Z_w,f_w)$, $V$ and $U$ are objects of the sites $T_{\overline{\Spec(\bz)}}$, $Et_{\overline{\Spec(\bz)}}$ and $Et_{\bar{X}}$ respectively. Then $(Z_0,Z_w,f_w)\times_{V}U$ is seen as an object of $T_{\bar{X}}$. Finally, the place $p$ is defined as the image of $v\in\bar{X}$ in $\overline{\Spec(\bz)}$. We refer to \cite{illusie09} and section \ref{section-cohomology-XW} for the definition of the site $\mathcal{C}_{\bar{X}}$.

There is a natural functor from the category of arrows of the form (\ref{one-arrow}) to the category of arrows $(W_{k(v)}\times T)\rightarrow
i_v^*(Y_0,Y_w,f_w)$ sending $((Z_0,Z_w,f_w)\rightarrow V\leftarrow V')$ to $(Z_0,Z_w,f_w)\times_{V}V'$. This provides us with the natural map
\begin{equation}\label{one-iso}
\mathrm{i}_{v}^p\mathrm{j}_*\mathcal{F}(W_{k(v)}\times T)\longrightarrow i_{v}^pj_*\mathcal{F}(W_{k(v)}\times T).
\end{equation}
In order to show that (\ref{one-iso}) is an isomorphism, we have to show that the system
$$W_{k(v)}\times T\longrightarrow
i_v^*((Z_0,Z_w,f_w)\times_V U),$$
where  $(Z_0,Z_w,f_w)\rightarrow V\leftarrow U)$ runs over the class of objects in $C_{\bar{X}}$, is cofinal in the category of arrows $\mathcal{A}_{v,T}$ :
$$W_{k(v)}\times T\rightarrow i_v^*(Y_0,Y_w,f_w).$$
We know that the system given by the $\mathcal{Y}(K,S,W_{k(v)}\times T)$'s is cofinal in $\mathcal{A}_{v,T}$. Here $v\in S$ and $K/F$ is unramified outside $S$. One can choose $S$ large enough so that $S$ contains $Y$. Let $S'$ be the image of $S$ in $\overline{\Spec(\bz)}$. Then $K/\bq$ is unramified outside $S'$. If we denote by $L/\mathbb{Q}$ the Galois closure of $K/\mathbb{Q}$ (in the fixed algebraic closure $\overline{\bq}/\bq$), then $L/\mathbb{Q}$ remains unramified outside $S'$, and $L/F$ is Galois and unramified outside $S$. Moreover, we have a morphism
$$\mathcal{Y}(L/F,S,W_{k(v)}\times T)\rightarrow \mathcal{Y}(K/F,S,W_{k(v)}\times T)\mbox{ in } \mathcal{A}_{v,T}.$$
Hence one can restrict our attention to the objects of $\mathcal{A}_{v,T}$ of the form
$$W_{k(v)}\times T\rightarrow i_v^*\mathcal{Y}(K/F,S,W_{k(v)}\times T)$$
where $K/\bq$ is a Galois extension unramified outside $S'$. We denote again by $p$ the image of $v$ in $\overline{\Spec(\bz)}$ and we consider the object
$$\mathcal{Y}(K/\bq,S',W_{k(p)}\times T)\rightarrow V\leftarrow U\mbox{ in }\mathcal{C}_{\bar{X}},$$
where $V$ and $U$ are defined as follows. Using Proposition \ref{prop-morph-sites-etale-loc-sections}, the \'etale $\overline{\Spec(\bz)}$-scheme $V$ is given by the $G_{\bq}$-set $G_{K/\bq}/I_p$, with  no point over $S'-\{p\}$, and exactly one point over the place $p$ corresponding to the distinguished $G_{\bq_p}$-orbit of $G_{K/\bq}/I_p$ on which the inertia group $I_p$ acts trivially. The \'etale $\bar{X}$-scheme $U$ is given by the $G_{F}$-set $G_{K/F}/I_v$, with  no point over $S-\{v\}$, and exactly one point over the place $v$ corresponding to the distinguished $G_{F_v}$-orbit of $G_{K/F}/I_v$ on which the inertia group $I_v$ acts trivially. Finally, we enlarge $S$ so that $S$ is the inverse image of $S'$ (which is the image of $S$) along the map $\bar{X}\rightarrow\overline{\Spec(\bz)}$. Then the map $U\rightarrow V$ is well defined.

Assume that one has an identification
\begin{equation}\label{key}
\mathcal{Y}(K/F,S,W_{k(v)}\times T)=\mathcal{Y}(K/\bq,S',W_{k(p)}\times T)\times_V U
\end{equation}
in the category $T_{\bar{X}}$. It would follow that the system of objects
$$W_{k(v)}\times T\rightarrow
i_v^*((Z_0,Z_w,f_w)\times_V U)$$
is cofinal in the category $\mathcal{A}_{v,T}$. The map (\ref{one-iso}) would be an isomorphism for any $T$ and any $\mathcal{F}$, hence (\ref{iso-presheaves}) would be an isomorphism of presheaves for any $\mathcal{F}$. This would show that the transformation (\ref{nat-transformation-at-v}) is an isomorphism. Hence the transformation (\ref{nat-transformation}) would be an isomorphism as well.

It is therefore enough to show (\ref{key}). One has
$$\mathcal{Y}(K/F,S,W_{k(v)}\times T)=\mathcal{Y}(K/F,S,W_{k(v)})\times (T,T,Id_T)$$  and
$$\mathcal{Y}(K/\bq,S',W_{k(p)}\times T)=\mathcal{Y}(K/\bq,S',W_{k(p)})\times (T,T,Id_T)$$ in the category $T_{\bar{X}}$, hence one can assume that $T=*$ is the point.
We have a map in $T_{\bar{X}}$
\begin{equation}\label{lastmapfor(v)}
\mathcal{Y}(K/F,S,W_{k(v)})\longrightarrow\mathcal{Y}(K/\bq,S',W_{k(p)})\times_V U.
\end{equation}
and we need to show that it is an isomorphism. Let $w$ be a point of $\bar{X}$. If $w\in S$ and $w\neq v$, then the $w$-component of both the right hand side and the left hand side in $(\ref{lastmapfor(v)})$ are empty. Assume that $w$ is not in $S$. Then the $w$-components of $\mathcal{Y}(K/F,S,W_{k(v)})$, $\mathcal{Y}(K/\bq,S',W_{k(p)})$, $V$ and $U$ are the $W_{k(w)}$-spaces $W_{K/F,S}/W_{F_v}^1$, $W_{K/\bq,S'}/W_{\bq_p}^1$, $G_{K/\bq}/I_p$ and $G_{K/F}/I_v$ respectively. But we have an $W_{k(w)}$-equivariant homemorphism
$$W_{K/F,S}/W_{F_v}^1\cong (W_{K/\bq,S'}/W_{\bq_p}^1)\times_{(G_{K/\bq}/I_p)}(G_{K/F}/I_v).$$
Moreover, the $v$-component of $\mathcal{Y}(K/F,S,W_{k(v)})$, $\mathcal{Y}(K/\bq,S',W_{k(p)})$, $V$ and $U$ are the $W_{k(v)}$-spaces
$W_{k(v)}$, $W_{k(p)}$, $G_{k(p)}/G_{k(u)}$ and $G_{k(v)}/G_{k(u)}$, where $u$ the unique point of $U$ lying over $v$. But we have
an $W_{k(v)}$-equivariant homemorphism
$$W_{k(v)}\cong W_{k(p)}\times_{(G_{k(p)}/G_{k(u)})}(G_{k(v)}/G_{k(u)}).$$
This shows that (\ref{lastmapfor(v)}) is an isomorphism in $T_{\bar{X}}$, and (\ref{key}) follows.\\

\textbf{(v) The morphism $\delta_{\bar{X}}$ is an equivalence.} We consider the glued topoi $(Y_W,\bar{X}'_W,{i}^*{j}_*)$ and $(Y_W,\bar{X}'_W,\mathrm{i}^*\mathrm{j}_*)$. Recall that an object of $(Y_W,\bar{X}'_W,{i}^*{j}_*)$ is a triple $(E,F,\sigma)$ with $E\in Y_W$, $F\in\bar{X}'_W$ and $\sigma:E\rightarrow{i}^*{j}_*F$ (see \cite{sga4} IV.9.5.3). There is a canonical functor
$$
\appl{\bar{X}_W}{(Y_W,\bar{X}'_W,{i}^*{j}_*)}{\mathcal{F}}{(i^*\mathcal{F},j^*\mathcal{F},i^*\mathcal{F}\rightarrow i^*j_*j^*\mathcal{F})}
$$
where $i^*\mathcal{F}\rightarrow i^*j_*j^*\mathcal{F}$ is given by adjunction. By (\cite{sga4} IV Theorem 9.5.4), this functor is an equivalence, and the same is true for the canonical functor
$$\bar{X}_{et}\times_{\overline{\Spec(\bz)}_\et}\overline{\Spec(\bz)}_W\longrightarrow(Y_W,\bar{X}'_W,\mathrm{i}^*\mathrm{j}_*)$$
Under these identifications, the inverse image functor $\delta^*_{\bar{X}}$ is given by (see diagram (\ref{diagram-for-fiberproduct}))
$$
\fonc{\delta^*_{\bar{X}}} {(Y_W,\bar{X}'_W,\mathrm{i}^*\mathrm{j}_*)} {(Y_W,\bar{X}'_W,{i}^*{j}_*)}{(E,F,\tau)}{(E,F,t_F\circ\tau)}
$$
Here $t$ is the transformation defined in step (ii), and $t_F\circ\tau$ denotes the following composition :
$$t_F\circ\tau: E\rightarrow\mathrm{i}^*\mathrm{j}_*F\rightarrow{i}^*{j}_*F.$$ Since $t$ is an isomorphism of functors, the inverse image functor $\delta_{\bar{X}}^*$ is an equivalence, hence so is the morphism $\delta_{\bar{X}}$.
\end{proof}

\section{The definition of $\overline{\X}_W$}\label{section-xwdef}

\subsection{}
Let $\X$ be a scheme separated and of finite type over $\Spec(\bz)$. Recall the defining site $Et_{\overline{\X}}$ of the Artin-Verdier \'etale topos $\overline{\X}_{\et}$ from section \ref{sect-AVetaletopos}. For any object $\overline{\mathcal{U}}$ of $Et_{\overline{\X}}$ one has the induced topos
\[ \overline{\mathcal{U}}_\et=\overline{\X}_{\et}/\overline{\mathcal{U}}\cong \widetilde{(Et_{\overline{\X}}/\overline{\mathcal{U}},\mathcal{J}_{ind})}.\]

\begin{definition} For any object $\overline{\mathcal{U}}$ of $Et_{\overline{\X}}$ we define the Weil-\'etale topos of $\overline{\mathcal{U}}$ as the fiber product
\[\overline{\mathcal{U}}_W:=\overline{\mathcal{U}}_\et\times_{\overline{\Spec(\bz)}_\et}\overline{\Spec(\bz)}_W.\]
\label{xwdef}\end{definition}
This topos is defined by a universal property in the 2-category of
topoi. As a consequence, it is well defined up to a canonical
equivalence. We point out two special cases. If $\overline{\mathcal{U}}=(\X,\X_\infty)=\overline{\X}$ is the final object we obtain the definition of $\overline{\X}_W$ and if $\overline{\mathcal{U}}=(\X,\emptyset)$ we obtain the definition of $\X_W$. The topos $\X_W$ will play no role in this paper but $\overline{\X}_W$ is our central object of study in case $\X$ is proper and regular.

Note also that for $\X=\Spec(\co_F)$ Definition \ref{xwdef} is consistent with Definition \ref{ofwdef} by Proposition \ref{prop-fiberproduct-makes-sense}.

\begin{prop}
The first projection yields a canonical morphism
$$\gamma_{\overline{\X}}:\overline{\X}_W\longrightarrow\overline{\X}_\et.$$
\end{prop}

\begin{prop}
There is a canonical morphism
$$\mathfrak{f}_{\overline{\X}}:\overline{\X}_W\longrightarrow B_{\mathbb{R}}.$$
\end{prop}
\begin{proof}
The morphism $\mathfrak{f}_{\overline{\X}}$ is defined as the composition
$$\overline{\X}_W\longrightarrow \overline{\Spec(\bz)}_W\longrightarrow B_{\mathbb{R}}$$
where the first arrow is the projection and the second is the
morphism of Proposition \ref{map-to-BR}.
\end{proof}
The structure of the topos $\overline{\X}_W$
over any closed point of $\overline{\Spec(\bz)}$ is made explicit below. Note that $\X\otimes_{\mathbb{Z}}\mathbb{F}_p$ is not assumed to be regular.
\begin{prop}\label{prop-pullback-at-p}
Let $\Spec(\mathbb{F}_p)$ be a closed point of $\Spec(\bz)$. Then
$$\overline{\X}_W\times_{\overline{\Spec(\bz)}_\et}{\Spec(\mathbb{F}_p)_\et}\cong (\X\otimes_{\mathbb{Z}}\mathbb{F}_p)_W\cong (\X\otimes_{\mathbb{Z}}\mathbb{F}_p)^{sm}_W\times\mathcal{T}\\
$$
where $(\X\otimes_{\mathbb{Z}}\mathbb{F}_p)_W$ denotes the big Weil-\'etale topos of $\X\otimes_{\mathbb{Z}}\mathbb{F}_p$.
\end{prop}
\begin{proof}
The result follows from the following equivalences.
\begin{align*}
\overline{\X}_W\times_{\overline{\Spec(\bz)}_\et}{\Spec(\mathbb{F}_p)_\et}&\cong
\overline{\X}_\et\times_{\overline{\Spec(\bz)}_\et}\overline{\Spec(\bz)}_W
\times_{\overline{\Spec(\bz)}_\et}{\Spec(\mathbb{F}_p)_\et}\\
&\cong
\overline{\X}_\et\times_{\overline{\Spec(\bz)}_\et}B_{W_{\mathbb{F}_p}}\\
&\cong
\overline{\X}_\et\times_{\overline{\Spec(\bz)}_\et}{\Spec(\mathbb{F}_p)_{et}}\times_{\Spec(\mathbb{F}_p)_{et}}B_{W_{\mathbb{F}_p}}\\
&\cong
(\X\otimes_{\mathbb{Z}}\mathbb{F}_p)_\et\times_{B^{sm}_{G_{\mathbb{F}_p}}}B_{W_{\mathbb{F}_p}}\\
&\cong (\X\otimes_{\mathbb{Z}}\mathbb{F}_p)_W
\end{align*}
The second equivalence, the fourth and the last one are given by Theorem \ref{thm-pull-back}, Proposition \ref{pull-bu} and Corollary \ref{cor-fiber-product} respectively.
\end{proof}
\begin{corollary}
The closed immersion of schemes
$(\X\otimes_{\mathbb{Z}}\mathbb{F}_p)\rightarrow\X$
induces a closed embedding of topoi
$$(\X\otimes_{\mathbb{Z}}\mathbb{F}_p)_W\longrightarrow \X_W.$$
\end{corollary}
We denote by $\infty$ the closed point of $\overline{\Spec(\bz)}$
corresponding to the archimedean place of $\bq$. This point yields a
closed embedding of topoi 
$$\underline{Set}=Sh(\infty)\longrightarrow\overline{\Spec(\bz)}_{et}.$$
This paper suggests the following definition.
\begin{definition}\label{xinftydef}
We define the Weil-\'etale topos of $\X_{\infty}$ as follows:
$$\X_{\infty,W}:=Sh(\X_{\infty})\times B_{\mathbb{R}}.$$
\end{definition}
The argument of Proposition \ref{prop-pullback-at-p} is also valid for the archimedean fiber.
\begin{prop}\label{prop-pullback-archW}
We have a pull-back square of topoi:
\begin{equation}\begin{CD}
@. \X_{\infty,W} @>>>\underline{Set}\\
@. @VV{i_{\infty}}V @VVV @.\\
@. \overline{\X}_W @>>>\overline{\Spec(\bz)}_{et} @. {}
\end{CD}\end{equation}
In particular $i_{\infty}$ is a closed embedding.
\end{prop}
\begin{proof}
The result follows from the following equivalences.
\begin{align*}
\overline{\X}_W\times_{\overline{\Spec(\bz)}_\et}{\underline{Set}}&\cong
\overline{\X}_\et\times_{\overline{\Spec(\bz)}_\et}\overline{\Spec(\bz)}_W
\times_{\overline{\Spec(\bz)}_\et}{\underline{Set}}\\
&\cong
\overline{\X}_\et\times_{\overline{\Spec(\bz)}_\et}B_{\mathbb{R}}\\
&\cong
\overline{\X}_\et\times_{\overline{\Spec(\bz)}_\et}{\underline{Set}}\times_{\underline{Set}}B_{\br}\\
&\cong
Sh(\X_{\infty})\times_{\underline{Set}}B_{\br}\\
&\cong Sh(\X_{\infty})\times B_{\br}\\
&=\X_{\infty,W}
\end{align*}
Indeed, the second (respectively the fourth) equivalence above is given by Theorem \ref{thm-pull-back} (respectively by Corollary \ref{pull-ba}).
\end{proof}

\subsection{} We assume here that $\X$ is irreducible and flat over $Spec(\bz)$. Let us study the structure of $\overline{\X}_W$ at the generic point of $\X$. We denote by $K(\X)$ the function field of
the irreducible scheme $\X$. Let $\overline{K(\X)}/K(\X)$ be an
algebraic closure. The algebraic closure $\overline{\bq}/\bq$ is
taken as a sub-extension of $\overline{K(\X)}/\bq$. Then we have a
continuous morphism $G_{K(\X)}\rightarrow G_{\bq}$.
\begin{definition}Let $\X$ be an irreducible scheme wich is flat, separated and of finite type over $Spec(\bz)$.
We consider the locally compact topological group
$$W_{K(\X)}:=G_{K(\X)}\times_{G_{\bq}}W_{\bq}$$
defined as a fiber product in the category of topological
groups.
\end{definition}
\begin{prop}\label{prop-morphism-j}
Let $\X$ be an irreducible scheme wich is flat, separated and of finite type over $Spec(\bz)$. There is a canonical morphism
$j_{\overline{\X}}:B_{W_{K(\X)}}\rightarrow\overline{\X}_W$.
\end{prop}
\begin{proof}
The continuous morphism $W_{K(\X)}\rightarrow G_{K(\X)}$ induces a
morphism $$B_{W_{K(\X)}}\rightarrow B_{G_{K(\X)}}\rightarrow
B^{sm}_{G_{K(\X)}}.$$ Here the second map is the canonical morphism
from the big classifying topos of $G_{K(\X)}$ to its small
classifying topos, whose inverse image sends a continuous $G_{K(\X)}$-set $E$ to the sheaf represented by the discrete $G_{K(\X)}$-space $E$ (see \cite{flach06-2} Section 7). The generic point of the irreducible scheme $\X$
and the previous choice of the algebraic closure $\overline{K(\X)}/K(\X)$ yield an
embedding $B^{sm}_{G_{K(\X)}}\hookrightarrow\overline{\X}_\et$. We
obtain a morphism
$$B_{W_{K(\X)}}\longrightarrow\overline{\X}_\et.$$
On the other hand we have maps
$$B_{W_{K(\X)}}\longrightarrow B_{W_{\bq}}\longrightarrow\overline{\Spec(\bz)}_W$$
and a commutative diagram
\begin{equation*}\begin{CD}
@.B_{W_{K(\X)}} @>>>\overline{\Spec(\bz)}_W\\
@. @VVV @VVV @.\\
@. \overline{\X}_\et @>>>\overline{\Spec(\bz)}_{et} @. {}
\end{CD}\end{equation*}
The result therefore follows from the very definition of
$\overline{\X}_W$.
\end{proof}
Unfortunately the morphism $j_{\overline{\X}}$ is not an embedding. The structure of $\overline{\X}_W$ at the generic point is more subtle, as it is shown below.
We assume again that $\X$ is irreducible, flat, separated and of finite type over $Spec(\bz)$. The generic point
$\Spec(\bq)\rightarrow\overline{\Spec(\bz)}$ and
$\overline{\bq}/\bq$ induce an embedding
$$B^{sm}_{G_{\bq}}\cong\Spec(\bq)_\et\hookrightarrow\overline{\Spec(\bz)}_\et.$$
The corresponding subtopos of $\overline{\Spec(\bz)}_W$ is the
classifying topos of the topological pro-group (see Proposition
\ref{generic-XW})
$$\underline{W}_{K/\bq,S}:=\{W_{K/\bq,S},\,\mbox{for $\overline{\bq}/K/\bq$ finite Galois and $S$ finite} \}.$$ Recall that
we have
$$B_{\underline{W}_{K/\bq,S}}:=\underleftarrow{lim}\,B_{W_{K/\bq,S}}$$
where the projective limit is understood in the 2-category of topoi.
In other words, there is a pull-back
\begin{equation*}\begin{CD}
@.B_{\underline{W}_{K/\bq,S}}  @>>>B^{sm}_{G_{\bq}}\\
@. @VV{i_{0}}V @VV{u_{0}}V @.\\
@.\overline{\Spec(\bz)}_W @>>>\overline{\Spec(\bz)}_{et} @. {}
\end{CD}\end{equation*}
The generic point of the irreducible scheme $\X$ and an algebraic
closure $\overline{K(\X)}/K(\X)$ yield an embedding
$B^{sm}_{G_{K(\X)}}\hookrightarrow\overline{\X}_\et$. 
We obtain
\begin{align*}
\overline{\X}_W\times_{\overline{\X}_\et}B^{sm}_{G_{K(\X)}}&=
\overline{\Spec(\bz)}_W\times_{\overline{\Spec(\bz)}_\et}\overline{\X}_\et\times_{\overline{\X}_\et}B^{sm}_{G_{K(\X)}}
\\
&\cong
\overline{\Spec(\bz)}_W\times_{\overline{\Spec(\bz)}_\et}B^{sm}_{G_{K(\X)}}\\
&\cong
\overline{\Spec(\bz)}_W\times_{\overline{\Spec(\bz)}_\et}B^{sm}_{G_{\bq}}\times_{B^{sm}_{G_{\bq}}}B^{sm}_{G_{K(\X)}}\\
&\cong
B_{\underline{W}_{K/\bq,S}}\times_{B^{sm}_{G_{\bq}}}B^{sm}_{G_{K(\X)}}
\end{align*}
The small classifying topos $B^{sm}_{G_{K(\X)}}$ is the projective
limit $\underleftarrow{lim}\,B^{sm}_{G_{L/K(\X)}}$ where $L/K(\X)$
runs over the finite Galois sub-extension of
$\overline{K(\X)}/K(\X)$. For such $L$ we set
$L':=L\cap\overline{\bq}$. Then the same is true for
$B^{sm}_{G_{\bq}}$, i.e. we have
$B^{sm}_{G_{\bq}}=\underleftarrow{lim}\,B^{sm}_{G_{L'/\bq}}$. Since
projective limits commute between themselves, we have
\begin{align*}
B^{sm}_{G_{K(\X)}}\times_{B^{sm}_{G_{\bq}}}B_{\underline{W}_{K/\bq,S}}&=
\underleftarrow{lim}\,B^{sm}_{G_{L/K(\X)}}\times_{\underleftarrow{lim}\,B^{sm}_{G_{L'/\bq}}}\underleftarrow{lim}\,B_{W_{L'/\bq,S}}\\
&=\underleftarrow{lim}_{_{L,S}}\,\,(B^{sm}_{G_{L/K(\X)}}\times_{B^{sm}_{G_{L'/\bq}}}B_{W_{L'/\bq,S}})\\
\end{align*}
By Corollary \ref{cor-fiberproduct-classtopoi-topgrps}, the fiber product
$B^{sm}_{G_{L/K(\X)}}\times_{B^{sm}_{G_{L'/\bq}}}B_{W_{L'/\bq,S}}$
is equivalent to the classifying topos of the topological group
$G_{L/K(\X)}\times_{G_{L'/\bq}}W_{L'/\bq,S}$ where the fiber product
is in turn computed in the category of topological groups. Note that $W_{L'/\bq,S}\rightarrow G_{L'/\bq}$ has local sections since $G_{L'/\bq}$ is profinite (see \cite{flach06-2} Proposition 2.1).

\begin{definition}
Let $\overline{K(\X)}/L/K(\X)$ be a finite Galois sub-extension and
let $S$ be a finite set of places of $\bq$ containing all the places
which ramify in $L'=L\cap\overline{\bq}$. We consider
the locally compact topological group
$$W_{L/K(\X),S}:=G_{L/K(\X)}\times_{G_{L'/\bq}}W_{L'/\bq,S}$$
defined as a fiber product in the category of topological groups.
\end{definition}
We have obtained the following result.
\begin{prop}Let $\X$ be an irreducible scheme wich is flat, separated and of finite type over $Spec(\bz)$. We have a pull-back square of topoi
\begin{equation*}\begin{CD}
@.\underleftarrow{lim}_{_{L,S}}\,B_{W_{L/K(\X),S}} @>>>B^{sm}_{G_{K(\X)}}\\
@. @VVV @VVV @.\\
@.\overline{\X}_W @>>>\overline{\X}_{et} @. {}
\end{CD}\end{equation*}
where the vertical arrows are embedding.
\end{prop}

\section{Cohomology of $\overline{\X}_W$ with $\tilde{\br}$-coefficients}\label{section-cohomology-XW}

The fiber product topos $\overline{\X}_W$, as defined in section \ref{section-xwdef}, is equivalent to the category of sheaves on a site $(\mathcal{C}_{\overline{\X}},\mathcal{J}_{\overline{\X}})$ lying in
a \emph{non-commutative} diagram of sites
\begin{equation*}\begin{CD}
@.(\mathcal{C}_{\overline{\X}},\mathcal{J}_{\overline{\X}}) @<<< (T_{\overline{\Spec(\bz)}},\mathcal{J}_{ls})\\
@. @AAA @AA{\gamma^*}A @.\\
@.(Et_{\overline{\X}},\mathcal{J}_\et) @<{f^*}<<
(Et_{\overline{\Spec(\bz)}},\mathcal{J}_\et) @. {}
\end{CD}\end{equation*}
The site $(\mathcal{C}_{\overline{\X}},\mathcal{J}_{\overline{\X}})$
is defined as follows (see \cite{illusie09}). The category $\mathcal{C}_{\overline{\X}}$ is the category
of pairs of morphisms $\mathcal{U}\rightarrow V\leftarrow
\mathcal{Z}$, where $\mathcal{U}$ is an object of
$Et_{\overline{\X}}$, $V$ is an object of
$Et_{\overline{\Spec(\bz)}}$ and $\mathcal{Z}$ is an object of
$T_{\overline{\Spec(\bz)}}$. The map $\mathcal{U}\rightarrow V$
(respectively $\mathcal{Z}\rightarrow V$) is understood as a
morphism $\mathcal{U}\rightarrow f^*V$ in $Et_{\overline{\X}}$
(respectively as a morphism $\mathcal{Z}\rightarrow \gamma^*V$ in
$T_{\overline{X}}$).

The topology $\mathcal{J}_{\overline{\X}}$ is generated by the
covering families $$\{(\mathcal{U}_i\rightarrow V_i\leftarrow
\mathcal{Z}_i)\rightarrow (\mathcal{U}\rightarrow V\leftarrow
\mathcal{Z}),\,i\in I\}$$ of the following types:

(a) $\mathcal{U}_i=\mathcal{U}$, $V_i=V$ and
$\{\mathcal{Z}_i\rightarrow \mathcal{Z}\}$ is a covering family.

(b) $\mathcal{Z}_i=\mathcal{Z}$, $V_i=V$ and
$\{\mathcal{U}_i\rightarrow \mathcal{U}\}$ is a covering family.

(c) $\{(\mathcal{U}'\rightarrow V'\leftarrow
\mathcal{Z}')\rightarrow (\mathcal{U}\rightarrow V\leftarrow
\mathcal{Z})\}$ with $\mathcal{U}'=\mathcal{U}$, and
$\mathcal{Z}'\rightarrow \mathcal{Z}$ is obtained by base change
from the map $V'\rightarrow V$ of $Et_{\overline{\Spec(\bz)}}$.

(d) $\{(\mathcal{U}'\rightarrow V'\leftarrow
\mathcal{Z}')\rightarrow (\mathcal{U}\rightarrow V\leftarrow
\mathcal{Z})\}$ with $\mathcal{Z}'=\mathcal{Z}$, and
$\mathcal{U}'\rightarrow \mathcal{U}$ is obtained by base change
from the map $V'\rightarrow V$ of $Et_{\overline{\Spec(\bz)}}$.

Then $(\mathcal{C}_{\overline{\X}},\mathcal{J}_{\overline{\X}})$ is a
defining site for the fiber product topos $\overline{\X}_W$. The topology $\mathcal{J}_{\overline{\X}}$ \emph{is not subcanonical}.

\begin{definition}
For any $\T$-topos $t:\mathcal{E}\rightarrow\T$, we define the sheaf of \emph{continuous real valued functions on $\mathcal{E}$} as follows:
$$\tr:=t^*(y\mathbb{R})$$
Here $y\mathbb{R}$ is the abelian object of $\T$ represented by the standard topological group $\mathbb{R}$.
\end{definition}
For an irreducible scheme $\X$ which is flat, separated and of finite type over $\Spec(\bz)$, we consider the morphism
$j_{\overline{\X}}:B_{W_{K(\X)}}\rightarrow\overline{\X}_W$ defined in Proposition \ref{prop-morphism-j}.
\begin{prop}\label{prop-acyclic-inclusion}
Let $\X$ be an irreducible scheme wich is flat, separated and of finite type over $Spec(\bz)$. We have $R^nj_{\overline{\X},*}\tr=0$ for any
$n\geq1$.
\end{prop}
\begin{proof}
Recall that the morphism $j_{\overline{\X}}$ is defined by the
following commutative diagram of topoi.
\begin{equation*}\begin{CD}
@.B_{W_{K(\X)}} @>b>>\overline{\Spec(\bz)}_W\\
@. @VaVV @V\gamma VV @.\\
@. \overline{\X}_\et @>f>>\overline{\Spec(\bz)}_{et} @. {}
\end{CD}\end{equation*}
The site $(B_{Top}W_{K(\X)},\mathcal{J}_{ls})$ is a defining site
for $B_{W_{K(\X)}}$, and we denote by $a^*$, $b^*$, $\gamma^*$ and
$f^*$ the morphism of sites inducing the morphism of topos $a$, $b$,
$\gamma$ and $f$. The morphism
$j_{\overline{\X}}:B_{W_{K(\X)}}\rightarrow \overline{\X}_W$ is
induced by the morphism of sites :
$$
\appl{\mathcal{C}_{\overline{\X}}}{B_{Top}W_{K(\X)}}{(\mathcal{U}\rightarrow
V\leftarrow\mathcal{Z})}{a^*\mathcal{U}\times_{a^*f^*V}
b^*\mathcal{Z}}
$$
Note that one has an identification $a^*\mathcal{U}\times_{a^*f^*V}
b^*\mathcal{Z}= a^*\mathcal{U}\times_{b^*\gamma^*V} b^*\mathcal{Z}$.
Consider the object of $T_{\overline{\Spec(\bz)}}$ whose components are all
given by the action of $W_{\bq}$ on $W_{\bq}/W_{\bq}^1\cong\mathbb{R}$:
$$(\mathbb{R},\mathbb{R},Id_{\mathbb{R}})=\mathfrak{f}^*E\mathbb{R}$$
This object $\mathfrak{f}^*E\mathbb{R}$ is a
covering of the final object in $T_{\overline{\Spec(\bz)}}$ for the
local section topology, hence
$$\mathfrak{f}^*_{\overline{\X}}E\mathbb{R}=(*\rightarrow *\leftarrow \mathfrak{f}^*E\mathbb{R})\longrightarrow(*\rightarrow*\leftarrow*)$$ is
a covering of the final object of $\mathcal{C}_{\overline{\X}}$ for
the topology $\mathcal{J}_{\overline{\X}}$.

The sheaf $R^nj_{\overline{\X},*}\tr$ is the
sheaf on $(\mathcal{C}_\X,\mathcal{J}_\X)$ associated to the
presheaf
$$
\fonc{P^nj_{\overline{\X},*}\tr}{\mathcal{C}_\X}{Ab}{(\mathcal{U}\rightarrow
V\leftarrow
\mathcal{Z})}{H^n(B_{W_{K(\X)}},a^*\mathcal{U}\times_{a^*f^*V}
b^*\mathcal{Z},\tr)}
$$
Since the object $\mathfrak{f}^*_{\overline{\X}}E\mathbb{R}$ defined above
covers the final object of $\mathcal{C}_{\X}$, we can restrict our
attention to the slice category
$\mathcal{C}_{\X}/\mathfrak{f}^*_{\overline{\X}}E\mathbb{R}$. Let
$(\mathcal{U}\rightarrow V\leftarrow \mathcal{Z})$ be an object of
$\mathcal{C}_{\X}/\mathfrak{f}^*_{\overline{\X}}E\mathbb{R}$, i.e. $(\mathcal{U}\rightarrow V\leftarrow \mathcal{Z})$ is given with a map $\mathcal{Z}\rightarrow\mathfrak{f}^*_{\overline{\X}}E\mathbb{R}$
in $T_{\overline{\Spec(\bz)}}$. We obtain a morphism
$$a^*\mathcal{U}\times_{a^*f^*V}
b^*\mathcal{Z}\longrightarrow W_{K(\X)}/W^1_{K(\X)}$$ in the category
$B_{Top}{W_{K(\X)}}$, where the homogeneous space
$(W_{K(\X)}/W^1_{K(\X)})\cong \mathbb{R}$ is seen as an object of
$B_{Top}{W_{K(\X)}}$.

On the other hand the continuous morphism
$$W_{K(\X)}\longrightarrow W_{K(\X)}/W^1_{K(\X)}=\mathbb{R}$$
has a global continuous section.
This gives an isomorphism in $\mathcal{T}$
$$y(W_{K(\X)}/W^1_{K(\X)})=yW_{K(\X)}/yW^1_{K(\X)}$$ and a canonical
equivalence
$$B_{W_{K(\X)}}/y(W_{K(\X)}/W^1_{K(\X)})\cong B_{W_{K(\X)}}/(yW_{K(\X)}/yW^1_{K(\X)})\cong B_{W^1_{K(\X)}}.$$
Under this equivalence the object represented by
$$\alpha:a^*\mathcal{U}\times_{a^*f^*V} b^*\mathcal{Z}\longrightarrow
(W_{K(\X)}/W^1_{K(\X)})\cong\mathbb{R}$$ corresponds to the object of $B_{W^1_{K(\X)}}$ represented by the subspace
$$\alpha^{-1}(0)\subset a^*\mathcal{U}\times_{a^*f^*V} b^*\mathcal{Z}$$
endowed with the induced continuous action of $W^1_{K(\X)}$. Thus
one has $$H^n(B_{W_{K(\X)}},a^*\mathcal{U}\times_{a^*f^*V}
b^*\mathcal{Z},\tr)=H^n(B_{W^1_{K(\X)}},\alpha^{-1}(0),\tr).$$
Therefore it is enough to prove
\begin{equation}\label{vanishing-RgR}
H^n(B_{W^1_{K(\X)}},Z,\tr):=H^n(B_{W^1_{K(\X)}}/yZ,\tr\times
yZ)=0,
\end{equation}
for any object $Z$ of $B_{Top}W^1_{K(\X)}$ and any $n\geq1$. We have
two canonical equivalences
$$B_{W^1_{K(\X)}}/EW^1_{K(\X)}\cong\mathcal{T}\mbox{ and } (B_{W^1_{K(\X)}}/yZ)/(EW^1_{K(\X)}\times yZ)\cong\mathcal{T}/yZ.$$
We obtain a pull-back square
\begin{equation*}\begin{CD}
@.\mathcal{T}/yZ @>l'>>\mathcal{T}\\
@. @Vh'VV @VhVV @.\\
@. B_{W^1_{K(\X)}}/yZ @>l>>B_{W^1_{K(\X)}} @. {}
\end{CD}\end{equation*}
where all the maps are localization morphisms (local homeomorphisms
of topoi in the modern language). It follows easilly that this pull-back square satisfies the Beck-Chevalley condition
$$h^*l_*\cong l'_*h'^{*}.$$
Moreover the functor $h'^{*}$, being a localization functor, preserves injective abelian objects. We
obtain
\begin{equation}\label{une-beck-che}
h^*R^n(l_*)\mathcal{A}\cong R^n(l'_*)h'^{*}\mathcal{A}
\end{equation}
for any abelian object $\A$ of $B_{W^1_{K(\X)}}/yZ $ and any
$n\geq0$. The forgetful functor $h^*$ takes an object $\mathcal{F}$
of $\mathcal{T}$ endowed with an action of $yW^1_{K(\X)}$ to
$\mathcal{F}$. Hence $R^n(l_*)\mathcal{A}$ is the object
$R^n(l'_*)\mathcal{A}$ endowed with the induced
$yW^1_{K(\X)}$-action.
\begin{lemma}
We have $R^n(l'_*)(\tr\times yZ)=0$ for any
$n\geq1$.
\end{lemma}
\begin{proof}
We consider the morphism $l':\mathcal{T}/yZ\rightarrow\mathcal{T}$.
The sheaf $R^n(l'_*)(\tr\times yZ)$ on
$\mathcal{T}=(Top^{lc},\mathcal{J}_{ls})$ is the sheaf associated to the
presheaf
$$
\fonc{P^n(l'_*)(\tr\times
yZ)}{Top^{lc}}{Ab}{T}{H^n(\mathcal{T}/y(Z\times T),\tr\times
y(Z\times T))}
$$
It is enough to show that
$H^n(\mathcal{T}/yT',\tr\times yT')=0$ for any
locally compact topological space $T'=Z\times T$. By (\cite{sga4} IV.4.10.5) we have a canonical isomorphism
$$H^n(\mathcal{T}/yT',\tr\times yT')=H^n(T',\mathcal{C}^0(T',\mathbb{R}))$$
where the right hand side is the usual sheaf cohomology of the
\emph{paracompact} space $T'$ with values in the sheaf
$\mathcal{C}^0(T',\mathbb{R})$ of continuous real valued functions on
$T'$. It is well known that the sheaf $\mathcal{C}^0(T',\mathbb{R})$
is fine, hence acyclic for the global section functor. The Lemma
follows.
\end{proof}
Therefore the sheaf $$h^*R^n(l_*)(\tr\times
yZ)\cong R^n(l'_*)h'^{*}(\tr\times yZ)$$
vanishes for any $n\geq1$, hence so does
$R^n(l_*)(\tr\times yZ)$. The spectral sequence
$$H^p(B_{W^1_{K(\X)}},R^q(l_*)(\tr\times yZ))\Rightarrow H^{p+q}(B_{W^1_{K(\X)}}/yZ,\tr\times yZ)$$
degenerates and yields an isomorphism
$$H^n(B_{W^1_{K(\X)}},l_*(\tr\times yZ))\cong
H^{n}(B_{W^1_{K(\X)}}/yZ,\tr\times yZ)$$ for any
$n\geq0$. The sheaf $l_*(\tr\times yZ)$ is given
by the object $l'_*(\tr\times yZ)$ of $\T$
endowed with the induced action of $yW^1_{K(\X)}$, as it follows from (\ref{une-beck-che}). Furthermore, one
has
$$l'_*(\tr\times
yZ)=l'_*l'^*(\tr)=\underline{Hom}_{\mathcal{T}}(yZ,\tr)$$ where
the right hand side is the internal Hom-object in $\mathcal{T}$ (see \cite{sga4} IV Corollaire 10.8). The
sheaf $\underline{Hom}_{\mathcal{T}}(yZ,\tr)$ is
represented by the abelian topological group
$\underline{Hom}_{Top}(Z,\mathbb{R})$ of continuous maps from $Z$ to
$\mathbb{R}$ endowed with the compact-open topology, since $Z$ is
locally compact. The compact-open topology on
$\underline{Hom}_{Top}(Z,\mathbb{R})$ is the topology of uniform
convergence on compact sets, since $\mathbb{R}$ is a metric space.
The real vector space $\underline{Hom}_{Top}(Z,\mathbb{R})$ is
locally convex, Hausdorff and complete (see \cite{bourbaki} X.16. Corollaire 3). Note that the action of
$W^1_{K(\X)}$ on $\underline{Hom}_{Top}(Z,\mathbb{R})$ is induced by
the action on $Z$, and that the group $W^1_{K(\X)}$ is compact. By
\cite{flach06-2} Corollary 8, one has
$$H^n(B_{W^1_{K(\X)}},\underline{Hom}_{Top}(Z,\mathbb{R}))=0.$$
In summary, for any locally compact topological space $Z$ with a
continuous action of $W^1_{K(\X)}$ and any $n\geq1$, one has
\begin{align*}
H^{n}(B_{W^1_{K(\X)}}/yZ,\tr\times yZ)&=
H^n(B_{W^1_{K(\X)}},l_*(\tr\times yZ))\\
&=H^n(B_{W^1_{K(\X)}},\underline{Hom}_{Top}(Z,\mathbb{R}))\\
&=0\\
\end{align*}
Hence (\ref{vanishing-RgR}) holds and
$R^n(j_{\overline{\X},*})\tr=0$ for any $n\geq1$.
\end{proof}

\begin{lemma}\label{lem-use-normal}
Let $\X$ be an irreducible scheme which is flat, separated and of finite type over $Spec(\bz)$. If $\X$ is normal, then the adjunction map $$\mathfrak{f}_{\overline{\X}}^*\,\tr\longrightarrow
j_{\overline{\X}*}\,j_{\overline{\X}}^*\,\mathfrak{f}_{\overline{\X}}^*\,\tr\cong
j_{\overline{\X}*}\,\tr$$ is an isomorphism.
\end{lemma}

\begin{proof}

Firstly, we need to restrict the site $\mathcal{C}_{\overline{\X}}$. The class of connected \'etale $\overline{\X}$-schemes (respectively of connected \'etale $\overline{\Spec(\bz)}$-schemes) is a topologically generating family for the \'etale site of $\overline{\X}$ (respectively of $\overline{\Spec(\bz)}$). It follows easily that the subcategory $\mathcal{C}'_{\overline{\X}}\subset \mathcal{C}_{\overline{\X}}$, consisting in objects $(\mathcal{U}\rightarrow
V\leftarrow\mathcal{Z})$ of $\mathcal{C}_{\overline{\X}}$ such that $\mathcal{U}$ and $V$ are both connected, is a topologically generating family. Then we endow the full subcategory $\mathcal{C}'_{\overline{\X}}$ with the induced topology via the natural fully faithful functor
$$\mathcal{C}'_{\overline{\X}}\hookrightarrow \mathcal{C}_{\overline{\X}}.$$
Then $\mathcal{C}'_{\overline{\X}}$ is a defining site for the topos $\overline{\X}_W$.

The composite map $\mathfrak{f}_{\overline{\X}}\circ j_{\overline{\X}}: B_{W_{K(\X)}}\rightarrow\overline{\X}_W\rightarrow B_{\mathbb{R}}$ is induced by the morphism of topological groups
$W_{K(\X)}\rightarrow\mathbb{R}$. The canonical isomorphism
$j_{\overline{\X}}^*\mathfrak{f}_{\overline{\X}}^*\tr\cong\tr$ induces
$$j_{\overline{\X}*}\,j_{\overline{\X}}^*\,\mathfrak{f}_{\overline{\X}}^*\,\tr\cong
j_{\overline{\X}*}\,\tr.$$
On the one hand $\mathfrak{f}_{\overline{\X}}^*\tr$ is the
sheaf associated to the abelian presheaf
$$
\fonc{\mathfrak{f}_{\overline{\X}}^p\tr}{\mathcal{C}'_{\overline{\X}}}{Ab}{(\mathcal{U}\rightarrow
V\leftarrow\mathcal{Z})}{Hom_{\mathcal{C}'_{\overline{\X}}}((\mathcal{U}\rightarrow
V\leftarrow\mathcal{Z}),(*\rightarrow*\leftarrow\mathfrak{f}^*\tr))}
$$
where $\mathfrak{f}^*\tr$ denotes the object $(\mathbb{R},\mathbb{R},Id)$ of $T_{\overline{\Spec(\bz)}}$ (with trivial action of the Weil groups on $\mathbb{R}$). For any object $(\mathcal{U}\rightarrow V\leftarrow\mathcal{Z})$ of $\mathcal{C}'_{\overline{\X}}$ with $\mathcal{Z}=(Z_0,Z_v,f_v)$, one
has
\begin{align*}
\mathfrak{f}_{\overline{\X}}^p\tr(\mathcal{U}\rightarrow
V\leftarrow\mathcal{Z})
&=Hom_{\mathcal{C}'_{\overline{\X}}}((\mathcal{U}\rightarrow
V\leftarrow\mathcal{Z}),(*\rightarrow*\leftarrow\mathfrak{f}^*\tr))\\
&=Hom_{T_{\overline{\Spec(\bz)}}}(\mathcal{Z},\mathfrak{f}^*\tr)\\
&=Hom_{B_{Top}{W_{\bq}}}(Z_0,\mathbb{R})\\
&=Hom_{Top}(Z_0/W_{\bq},\mathbb{R}).
\end{align*}
One the other hand, the morphism $j_{\X}$ is induced by the continuous functor:
$$
\appl{\mathcal{C}'_{\overline{\X}}}{B_{Top}W_{K(\X)}}{(\mathcal{U}\rightarrow
V\leftarrow\mathcal{Z})}{a^*\mathcal{U}\times_{a^*f^*V}
b^*\mathcal{Z}}
$$
Hence the direct image $j_{\overline{\X}*}$ is given by
$$j_{\overline{\X}*}\tr(\mathcal{U}\rightarrow V\leftarrow
\mathcal{Z})=Hom_{B_{Top}{W_{K(\X)}}}(a^*\mathcal{U}\times_{a^*f^*V}
b^*\mathcal{Z},\mathbb{R}).$$ Here the topological group $\mathbb{R}$ is given with the
trivial action of $W_{K(\X)}$. We set
$$\mathcal{U}_0:=a^*\mathcal{U}\mbox{ and }V_0:=a^*f^*V.$$
Note that $\mathcal{U}_0$ (respectively $V_0$) is given by the finite $G_{K(\X)}$-set (respectively the finite $G_{\bq}$-set) corresponding, via Galois theory, to the \'etale $K(\X)$-scheme
$\mathcal{U}\times_{\X}\Spec K(\X)$ (respectively to the \'etale $\bq$-scheme $V\otimes\bq$).
Here, $\mathcal{U}_0$ (respectively $V_0$) is considered
as a finite set on which $W_{K(\X)}$ acts via
$W_{K(\X)}\rightarrow G_{K(\X)}$ (respectively via $W_{K(\X)}\rightarrow G_{\bq}$). Finally, $W_{K(\X)}$ acts on the space
$Z_0:=b^*\mathcal{Z}$ via $W_{K(\X)}\rightarrow W_\bq$.

Since $W_{K(\X)}$ acts
trivially on $\mathbb{R}$, one has
\begin{align*}
j_{\overline{\X}*}\tr(\mathcal{U}\rightarrow V\leftarrow
\mathcal{Z})
&=Hom_{B_{Top}{W_{K(\X)}}}(a^*\mathcal{U}\times_{a^*f^*V}
b^*\mathcal{Z},\mathbb{R})\\
&=Hom_{Top}((\mathcal{U}_0\times_{V_0}
Z_0)/W_{K(\X)},\mathbb{R}).\\
\end{align*}

\textbf{(i) The map $\mathfrak{f}_{\overline{\X}}^*\,\tr\rightarrow
j_{\overline{\X}*}\,\tr$ is a monomorphism.
}

The morphism $\mathfrak{f}_{\overline{\X}}^*\,\tr\rightarrow
j_{\overline{\X}*}\,\tr$ is given by adjunction. It is induced by the morphism of presheaves on $\mathcal{C}'_{\overline{\X}}$ given by the functorial map
\begin{equation}\label{cool-injection}
Hom_{Top}(Z_0/W_{\bq},\mathbb{R})\longrightarrow Hom_{Top}((\mathcal{U}_0\times_{V_0}
Z_0)/W_{K(\X)},\mathbb{R})
\end{equation}
which is in turn induced by the continuous map
\begin{equation}\label{cool-surjection}
(\mathcal{U}_0\times_{V_0}
Z_0)/W_{K(\X)}\longrightarrow Z_0/W_{\bq}.
\end{equation}
Let $(\mathcal{U}\rightarrow V\leftarrow\mathcal{Z})$ be an object of $\mathcal{C}'_{\overline{\X}}$. Hence $\mathcal{U}$ and $V$ are both connected. Since $\mathcal{U}$ and $V$ are both normal, the schemes $\mathcal{U}\times_{\X}\Spec(K(\X))$ and $V\times_{\overline{\Spec(\bz)}}\Spec(\bq)$ are connected as well. By Galois theory, the Galois groups
$G_{K(\X)}$ and $G_\bq$ act transitively on $\mathcal{U}_0$ and $V_0$ respectively. Hence the Weil groups
$W_{K(\X)}$ and $W_\bq$ act transitively on $\mathcal{U}_0$ and $V_0$ respectively.

We have maps of compactified schemes
$$\mathcal{U}\rightarrow \overline{\X}\rightarrow \overline{\Spec(\bz)}\mbox{ and }\mathcal{U}\rightarrow V\rightarrow \overline{\Spec(\bz)}.$$
We consider the subfield $L(\mathcal{U})$ of $K(\mathcal{U})$ consisting in elements of $K(\mathcal{U})$ that are algebraic over $\mathbb{Q}$, i.e. we set
$$L(\mathcal{U}):=K(\mathcal{U})\cap\overline{\bq}.$$
Note that $\mathcal{U}$ is normal and connected, hence irreducible, so that its function field $K(\mathcal{U})$ is well defined.
We consider the arithmetic curve $\overline{\Spec(\mathcal{O}_{L(\mathcal{U})})}$.
Since $\mathcal{U}$ is normal, we have a canonical map $$\mathcal{U}\longrightarrow\overline{\Spec(\mathcal{O}_{L(\mathcal{U})})}.$$
We denote by $U'$ the (open) image of $\mathcal{U}$ in $\overline{\Spec(\mathcal{O}_{L(\mathcal{U})})}$.
Then we have a factorization $$\mathcal{U}\rightarrow{U}'\rightarrow V\rightarrow \overline{\Spec(\bz)}$$ since $V\rightarrow \overline{\Spec(\bz)}$ is \'etale.

The group $W_{K(\X)}$ acts transitively on $\mathcal{U}_0$, hence the choice of a base point $u_0\in \mathcal{U}_0$ induces an isomorphism of $W_{K(\X)}$-sets :
$$\alpha:\mathcal{U}_0\cong W_{K(\X)}/W_{K(\mathcal{U})}=G_{K(\X)}/G_{K(\mathcal{U})}.$$
We fix such a base point $u_0\in \mathcal{U}_0$. Then one has a 1-1 correspondence
$$\appl{(\mathcal{U}_0\times Z_0)/W_{K(\X)}}{Z_0/W_{K(\mathcal{U})}=Z_0/W_{L(\mathcal{U})}}{(u_0,z)}{z}$$
where the equality $Z_0/W_{K(\mathcal{U})}=Z_0/W_{L(\mathcal{U})}$ follows from the fact that the image of $W_{K(\mathcal{U})}$ in $W_{\bq}$ is precisely $W_{L(\mathcal{U})}$.

We consider now the commutative diagram of topological spaces
\begin{equation*}\begin{CD}
@.(\mathcal{U}_0\times_{V_0} Z_0)/W_{K(\X)} @>(\ref{cool-surjection})>>Z_0/W_\bq\\
@. @ViVV @ AsAA @.\\
@. (\mathcal{U}_0\times Z_0)/W_{K(\X)} @>\cong>\alpha>Z_0/W_{L(\mathcal{U})} @.
\end{CD}\end{equation*}
where $i$ is injective and $s$ is surjective. We denote by $v_0$ the image of $u_0\in \mathcal{U}_0$ in $V_0$ (note that $\mathcal{U}\rightarrow V$ induces a map $\mathcal{U}_0\rightarrow V_0$).

Let $\overline{z}\in Z_0/W_\bq$. There exists $w\in W_{\bq}$ such that $w.z$ goes to $v_0$ under the $W_{\bq}$-equivariant map $Z_0\rightarrow V_0$, since $W_\bq$ acts transitively on $V_0$. Then one has $\overline{w.z}=\overline{z}$, $(u_0,w.z)\in\mathcal{U}_0\times_{V_0} Z_0$, and
$$s\circ\alpha\circ i\overline{(u_0,w.z)}=s\circ\alpha\overline{(u_0,w.z)}=s(\overline{w.z})=\overline{z}$$
where $\overline{*}$ stands for the orbit of some point $*$ under some group action. Using the previous commutative diagram, this shows that the map (\ref{cool-surjection}) is surjective whenever $\mathcal{U}$ and $V$ are both connected. Hence the map (\ref{cool-injection}) is injective for any object $(\mathcal{U}\rightarrow
V\leftarrow\mathcal{Z})$ of $\mathcal{C}'_{\overline{\X}}$. In other words the morphism of presheaves on $\mathcal{C}'_{\overline{\X}}$
$$\mathfrak{f}_{\overline{\X}}^p\tr\longrightarrow
j_{\overline{\X}*}\,\tr$$
is injective. Since the associated sheaf functor is exact, the morphism of sheaves
$$\mathfrak{f}_{\overline{\X}}^*\,\tr\longrightarrow
j_{\overline{\X}*}\,\tr$$ is injective.

\textbf{(ii) The map $\mathfrak{f}_{\overline{\X}}^*\,\tr\rightarrow
j_{\overline{\X}*}\,\tr$ is an epimorphism.}

One has
$$\mathfrak{f}_{\overline{\X}}^*\,\tr(\mathcal{U}\rightarrow *\leftarrow
\mathcal{Z})=Hom_{\overline{\X}_W}(\varepsilon\mathcal{U}\times\varepsilon \mathcal{Z}; \mathfrak{f}_{\overline{\X}}^*\,\tr).$$
Here we denote by $*$ the final object $Et_{\overline{\Spec(\bz)}}$. Moreover, we denote by $\varepsilon\mathcal{U}$ (respectively by $\varepsilon\mathcal{Z}$) the sheaf on the topos $\overline{\X}_W$ associated to the presheaf represented by $(\mathcal{U}\rightarrow *\leftarrow *)$ (respectively by $(*\rightarrow*\leftarrow\mathcal{Z})$), where $*$ stands for the final object of the corresponding site. Finally, the product
$\varepsilon\mathcal{U}\times\varepsilon \mathcal{Z}$ is computed in the topos $\overline{\X}_W$. By adjunction, we have
\begin{align*}
\mathfrak{f}_{\overline{\X}}^*\,\tr(\mathcal{U}\rightarrow *\leftarrow
\mathcal{Z})&=Hom_{\overline{\X}_W}(\varepsilon\mathcal{U}\times\varepsilon \mathcal{Z}; \mathfrak{f}_{\overline{\X}}^*\,\tr)\\
&=Hom_{\overline{\X}_W/\varepsilon\mathcal{U}}(\varepsilon\mathcal{U}\times\varepsilon \mathcal{Z},\varepsilon\mathcal{U}\times\mathfrak{f}_{\overline{\X}}^*\,\tr)
\end{align*}
where we consider the slice topos $\overline{\X}_W/\varepsilon\mathcal{U}$. On the other hand we have
\begin{align*}
\overline{\X}_W/\varepsilon\mathcal{U}&=(\overline{\X}_\et\times_{\overline{\Spec(\bz)}_\et}\overline{\Spec(\bz)}_W)/\varepsilon\mathcal{U}\\
&\cong(\overline{\X}_\et/y\mathcal{U})\times_{\overline{\Spec(\bz)}_\et}\overline{\Spec(\bz)}_W\\
&\cong\mathcal{U}_\et\times_{\overline{\Spec(\bz)}_\et}\overline{\Spec(\bz)}_W\\
&=:\mathcal{U}_W\\
&\cong\mathcal{U}_\et\times_{U'_\et}U'_\et\times_{\overline{\Spec(\bz)}_\et}\overline{\Spec(\bz)}_W\\
&\cong\mathcal{U}_\et\times_{U'_\et}U'_W
\end{align*}
where $U'\subseteq\overline{\Spec(\mathcal{O}_{L(\mathcal{U})})}$ is defined as in the proof of step (i). The last equivalence above is given by Proposition \ref{prop-fiberproduct-makes-sense}. Hence the fiber product site $\mathcal{C}_{\mathcal{U}}$ for $\mathcal{U}_\et\times_{U'_\et}U'_W$ given by the sites $Et_{\mathcal{U}}$, $Et_{U'}$ and $T_{U'}$, is a defining site for the topos $\mathcal{U}_W$. Then the object $\varepsilon\mathcal{U}\times\varepsilon \mathcal{Z}$ of $$\overline{\X}_W/\varepsilon\mathcal{U}=\mathcal{U}_W\cong\mathcal{U}_\et\times_{U'_\et}U'_W$$
is the sheaf associated to the presheaf on $\mathcal{C}_{\mathcal{U}}$ represented by $(*\rightarrow *\leftarrow \mathcal{Z})$
where $\mathcal{Z}=(Z_0,Z_v,f_v)$ is seen as an object of  $T_{U'}$ by restricting the group of operators on $Z_0$ and $Z_v$ for any place $v$ of $\bq$. Moreover, the object $\varepsilon\mathcal{U}\times\mathfrak{f}_{\overline{\X}}^*\,\tr$ of $\overline{\X}_W/\varepsilon\mathcal{U}=\mathcal{U}_W$ is precisely $\mathfrak{f}_{\mathcal{U}}^*\,\tr$. It is the sheaf associated to the presheaf $\mathfrak{f}_{\mathcal{U}}^p\,\tr$ on $\mathcal{C}_{\mathcal{U}}$ represented by $(*\rightarrow *\leftarrow\mathfrak{f}^*_{U'}\tr)$.
Therefore, we have
\begin{align*}
\mathfrak{f}_{\mathcal{U}}^p\,\tr(*\rightarrow *\leftarrow\mathcal{Z})
&= Hom_{\mathcal{C}_{\overline{\X}}}((*\rightarrow
*\leftarrow\mathcal{Z}),(*\rightarrow*\leftarrow\mathfrak{f}^*_{U'}\tr))\\
&=Hom_{T_{U'}}(\mathcal{Z},\mathfrak{f}^*_{U'}\tr)\\
&=Hom_{B_{Top}{W_{L(\mathcal{U})}}}(Z_0,\mathbb{R})\\
&=Hom_{Top}(Z_0/W_{L(\mathcal{U})},\mathbb{R})
\end{align*}
By the universal property of the associated sheaf functor, we obtain a map from
$$\mathfrak{f}_{\mathcal{U}}^p\,\tr(*\rightarrow *\leftarrow\mathcal{Z})=Hom_{Top}(Z_0/W_{L(\mathcal{U})},\mathbb{R})$$ to the set
\begin{align*}
\mathfrak{f}_{\mathcal{U}}^*\,\tr(*\rightarrow *\leftarrow\mathcal{Z})
&= Hom_{\mathcal{U}_W}(\varepsilon\mathcal{U}\times\varepsilon \mathcal{Z},\mathfrak{f}_{\mathcal{U}}^*\,\tr)\\
&=Hom_{\overline{\X}_W/\varepsilon\mathcal{U}}(\varepsilon\mathcal{U}\times\varepsilon \mathcal{Z},\varepsilon\mathcal{U}\times\mathfrak{f}_{\overline{\X}}^*\,\tr)\\
&=Hom_{\overline{\X}_W}(\varepsilon\mathcal{U}\times\varepsilon \mathcal{Z},\mathfrak{f}_{\overline{\X}}^*\,\tr)\\
&=\mathfrak{f}_{\overline{\X}}^*\,\tr(\mathcal{U}\rightarrow *\leftarrow
\mathcal{Z})
\end{align*}
Composing this map
$\mathfrak{f}_{\mathcal{U}}^p\,\tr(*\rightarrow *\leftarrow\mathcal{Z})\rightarrow \mathfrak{f}_{\overline{\X}}^*\,\tr(\mathcal{U}\rightarrow *\leftarrow
\mathcal{Z})$
with
\begin{equation}\label{sympa-surjection}
\mathfrak{f}_{\overline{\X}}^*\,\tr(\mathcal{U}\rightarrow *\leftarrow
\mathcal{Z})\longrightarrow j_{\overline{\X}*}\,\tr(\mathcal{U}\rightarrow *\leftarrow\mathcal{Z}),
\end{equation}
we obtain the natural bijective map from
$$\mathfrak{f}_{\mathcal{U}}^p\,\tr(*\rightarrow *\leftarrow\mathcal{Z})=Hom_{Top}(Z_0/W_{L(\mathcal{U})},\mathbb{R})$$
to the set
$$j_{\overline{\X}*}\,\tr(\mathcal{U}\rightarrow *\leftarrow\mathcal{Z})=Hom_{Top}(Z_0/W_{L(\mathcal{U})},\mathbb{R}).$$
It follows that the map (\ref{sympa-surjection}) is surjective.

It remains to show that the map
\begin{equation}\label{final-surjection}
\mathfrak{f}_{\overline{\X}}^*\,\tr(\mathcal{U}\rightarrow V\leftarrow\mathcal{Z})
\longrightarrow j_{\overline{\X}*}\,\tr(\mathcal{U}\rightarrow V\leftarrow\mathcal{Z})
\end{equation}
is surjective when $V$ is not necessarily the final object. For any object $(\mathcal{U}\rightarrow V\leftarrow \mathcal{Z})$ of $\mathcal{C}'_{\overline{\X}}$, we consider the following commutative diagram
\begin{equation*}\begin{CD}
@.\mathfrak{f}_{\overline{\X}}^*\,\tr(\mathcal{U}\rightarrow *\leftarrow\mathcal{Z})@>(\ref{sympa-surjection})>>j_{\overline{\X}*}\,\tr(\mathcal{U}\rightarrow *\leftarrow\mathcal{Z})\\
@. @VVV @ VpVV @.\\
@. \mathfrak{f}_{\overline{\X}}^*\,\tr(\mathcal{U}\rightarrow V\leftarrow\mathcal{Z}) @>(\ref{final-surjection})>>j_{\overline{\X}*}\,\tr(\mathcal{U}\rightarrow V\leftarrow\mathcal{Z}) @.
\end{CD}\end{equation*}
We have proven above that the map (\ref{sympa-surjection}) is surjective. The vertical arrow $p$ is the natural map from
$$j_{\overline{\X}*}\,\tr(\mathcal{U}\rightarrow *\leftarrow\mathcal{Z})= Hom_{B_{Top}W_{K(\X)}}(\mathcal{U}_0\times Z_0,\mathbb{R})$$
to the set
$$j_{\overline{\X}*}\,\tr(\mathcal{U}\rightarrow V\leftarrow\mathcal{Z})= Hom_{B_{Top}W_{K(\X)}}(\mathcal{U}_0\times_{V_0} Z_0,\mathbb{R}),$$
which is surjective as well. Indeed, $\mathcal{U}_0\times_{V_0} Z_0$ is an open and closed $W_{K(\X)}$-equivariant subspace of
$\mathcal{U}_0\times Z_0$, hence any equivariant continuous map $\mathcal{U}_0\times_{V_0} Z_0\rightarrow\mathbb{R}$ extends to an equivariant continuous map $\mathcal{U}_0\times Z_0\rightarrow\mathbb{R}$. It follows immediately from the previous commutative diagram that the map
(\ref{final-surjection}) is surjective, for any object $(\mathcal{U}\rightarrow V\leftarrow \mathcal{Z})$ of $\mathcal{C}'_{\overline{\X}}$. Therefore the morphism of sheaves $\mathfrak{f}_{\overline{\X}}^*\,\tr\rightarrow j_{\overline{\X}*}\,\tr$ is surjective.
\end{proof}

\begin{theorem}\label{thm-global-cohomology}
Let $\X$ be an irreducible scheme wich is flat, separated and of finite type over $Spec(\bz)$.
If $\X$ is normal then the morphism
$$\mathfrak{f}^*_{\overline{\X}}:H^n(B_{\mathbb{R}},\tr)\longrightarrow H^n(\overline{\X}_W,\tr)$$
is an isomorphism for any $n\geq0$.
\end{theorem}
\begin{proof}
The Leray spectral sequence
$$H^p(\overline{\X}_W,R^qj_{\overline{\X}*}\tr)\Rightarrow
H^{p+q}(B_{W_{K(\X)}},\tr)$$ degenerates by
Proposition \ref{prop-acyclic-inclusion}. This shows that the
canonical morphism
$$H^n(\overline{\X}_W,\tr)\cong
H^n(\overline{\X}_W,j_{\overline{\X}*}\tr)\longrightarrow
H^{n}(B_{W_{K(\X)}},\tr)$$ is an isomorphism,
where the first identification is given by Lemma
\ref{lem-use-normal}. These cohomology groups can be computed using
the spectral sequence associated with the extension
$$1\rightarrow W_{K(\X)}^1\rightarrow W_{K(\X)}\rightarrow\mathbb{R}\rightarrow1.$$
Indeed, localizing along $E\mathbb{R}$, we obtain a pull-back square
\begin{equation}\begin{CD}
@.B_{W^1_{K(\X)}} @>q'>>\mathcal{T}\cong B_{\mathbb{R}}/E\mathbb{R}\\
@. @Vp'VV @VpVV @.\\
@. B_{W_{K(\X)}} @>q>>B_{\mathbb{R}} @. {}
\end{CD}\end{equation}
This gives an isomorphism
$$p^*R^n(q_*)\cong R^n(q'_*)p'^*$$
for any $n\geq0$. The argument of the proof of Proposition
\ref{prop-acyclic-inclusion} shows that $R^n(q'_*)\tr=0$ for any
$n\geq1$. Hence
$R^n(q_*)\tr=0$ for any $n\geq1$.
Moreover $q$ is connected, i.e. $q^*$ is fully faithful, hence we have
$$q_*\tr=q_*q^*\tr=\tr.$$
Therefore, the Leray spectral sequence given by the
morphism $q$
$$H^i(B_{\mathbb{R}},R^j(q_*)\tr)\Rightarrow H^{i+j}(B_{W_{K(\X)}},\tr)$$
degenerates and yields
$$H^n(B_{\mathbb{R}},\tr)=H^n(B_{\mathbb{R}},q_*\tr)\cong H^n(B_{W_{K(\X)}},\tr)$$
for any $n\geq0$. The result follows.
\end{proof}

\section{Compact support Cohomology of $\X_W$ with $\tr$-coefficients}\label{comp-support}
Throughout this section, the arithmetic scheme $\X$ is supposed to be irreducible, normal, flat, and proper over $\Spec(\bz)$.

\subsection{The morphism $\gamma_{\overline{\X}}:\overline{\mathcal{X}}_W\rightarrow\overline{\mathcal{X}}_{et}$.}

Recall the notion of \'etale $\overline{\mathcal{X}}$-scheme defined in section \ref{sect-AVetaletopos}. An \'etale $\overline{\mathcal{X}}$-scheme is in particular a topological space so that it makes sense to speak of connected \'etale $\overline{\mathcal{X}}$-schemes. Theorem \ref{thm-global-cohomology} yields the following result.

\begin{corollary}\label{cor-cohomology-of-U}
For any connected \'etale $\overline{\mathcal{X}}$-scheme
$\overline{\mathcal{U}}$, the morphism
$$\mathfrak{f}^*_{\overline{\mathcal{U}}}:H^n(B_{\mathbb{R}},\tr)\longrightarrow H^n(\overline{\mathcal{U}}_W,\tr)$$
is an isomorphism for any $n\geq0$. In particular, one has
$H^n(\overline{\mathcal{U}}_W,\tr)=\mathbb{R}$
for $n=0,1$ and
$H^n(\overline{\mathcal{U}}_W,\tr)=0$ for
$n\geq2$.
\end{corollary}

\begin{proof} This is clear from the fact that an \'etale $\overline{\mathcal{X}}$-scheme $\overline{\mathcal{U}}=(\mathcal{U},D)$ is connected if and only if the scheme $\mathcal{U}$ is connected, and Theorem \ref{thm-global-cohomology} applies to $\mathcal{U}$.
\end{proof}

Recall that
$\gamma_{\overline{\X}}:\overline{\mathcal{X}}_W\rightarrow\overline{\mathcal{X}}_{et}$
is the projection induced by Definition \ref{xwdef}.
\begin{prop}\label{prop-global-base-change}
The sheaf $R^n\gamma_{\overline{\X}*}(\tr)$ is
the constant \'etale sheaf on $\overline{\X}$ associated to the
discrete abelian group $\mathbb{R}$ for $n=0,1$ and
$R^n\gamma_{\overline{\X}*}(\tr)=0$ for $n\geq2$.
\end{prop}
\begin{proof}
For any $n\geq0$, $R^n\gamma_{\overline{\X}*}(\tr)$ is the sheaf associated to the presheaf
$$
\appl{Et_{\overline{\X}}}{Ab}{\overline{\mathcal{U}}}
{H^n(\overline{\X}_W/\gamma_{\overline{\X}}^*\overline{\mathcal{U}},\tr)}
$$
Hence the result follows immediately from Corollary \ref{cor-cohomology-of-U}, since $\overline{\X}_W/\gamma_{\overline{\X}}^*\overline{\mathcal{U}}\cong\overline{\mathcal{U}}_W$.
\end{proof}

\subsection{The morphism $\gamma_\infty:\X_{\infty,W}\to \X_{\infty}$}

\subsubsection{} If $T$ is a locally compact space (or any space), one can define the big topos $TOP(T)$ of $T$ as the category of sheaves on the site $(Top/T,\mathcal{J}_{op})$ where $\mathcal{J}_{op}$ is the open cover topology. It is well known that the natural morphism $TOP(T)\rightarrow Sh(T)$ is a cohomological equivalence. The following lemma gives a slight generalization of this result.
\begin{lemma}\label{lemma-section}
Let $T$ be an object of $Top$. Let $\mathcal{T}/_T$ be the big topos
of $T$ and let $Sh(T)$ be its small topos. For any topos
$\mathcal{S}$, the canonical morphism
$$f:\mathcal{T}/_T\times\mathcal{S}\longrightarrow Sh(T)\times\mathcal{S}$$
has a section $s$ such that $s^*\cong f_*$.
\end{lemma}
\begin{proof}
We first observe that one has a canonical equivalence
$$TOP(T):=\widetilde{(Top/T,\mathcal{J}_{op})}\cong\mathcal{T}/_T,$$
where $TOP(T)$ is the big topos of the topological space $T$. In what follows, we shall identify $\mathcal{T}/_T$ with $TOP(T)$.
The morphism
$f':\mathcal{T}/_T\rightarrow Sh(T)$ has a canonical section $s':Sh(T)\rightarrow\mathcal{T}/_T$, hence the map
$$f=(f',Id_{\mathcal{S}}):\mathcal{T}/_T\times\mathcal{S}\longrightarrow Sh(T)\times\mathcal{S}$$
has a section
$$s:=(s',Id_{\mathcal{S}}):Sh(T)\times\mathcal{S}\longrightarrow\mathcal{T}/_T\times\mathcal{S}.$$
Moreover, we have $s'^*=f'_*$ hence a sequence of three adjoint functors
$$f'^*,\,\,\,f'_*=s'^*,\,\,\,s'_*.$$
The functor $f'_*=s'^*$ is called \emph{restriction} and is denoted by $Res$. The functor $f'^*$ is called \emph{prologement} and is denoted by $Prol$.

The category $Op(*)$ of open sets of the one point space is a defining site for the final topos $\underline{Set}$. The site $(\mathcal{S},\mathcal{J}_{can})$ and $(\mathcal{T}/_T,\mathcal{J}_{can})$ can be seen as sites for the topoi $\mathcal{S}$ and  $\mathcal{T}/_T$ respectively, where $\mathcal{J}_{can}$ denotes the canonical topology. Then the morphisms $f'$ and $s'$ are induced by the left exact continuous functors $f'^*$ and $s'^*$ respectively. A site for $\mathcal{T}/_T\times\mathcal{S}$ (respectively for $Sh(T)\times\mathcal{S}$) is given by the category $\mathcal{C}$ of objects of the form
$(\mathcal{F}\rightarrow *\leftarrow S)$ (respectively by the category $C$ of objects $(F\rightarrow *\leftarrow S)$). Here $S$ is an object of $\mathcal{S}$, $F$ is an \'etal\'e space on $T$, $\mathcal{F}$ is a big sheaf on $T$ and $*$ is the set with one element. The categories $C$ and $\mathcal{C}$ both have an initial object $(\emptyset\rightarrow\emptyset\leftarrow\emptyset)$. The morphism of topoi $f$ is induced by the morphism of sites
$$\fonc{f^{-1}}{C}{\mathcal{C}}{(F\rightarrow *\leftarrow S)}{(Prol(F)\rightarrow *\leftarrow S)}
$$
and the morphism $s$ is induced by the morphism of sites
$$\fonc{s^{-1}}{\mathcal{C}}{C}{(\mathcal{F}\rightarrow *\leftarrow S)}{(Res(\mathcal{F})\rightarrow *\leftarrow S)}
$$
Let $\mathcal{L}$ be a sheaf of $\mathcal{T}/_T\times\mathcal{S}$. Then one has
$$f_*\mathcal{L}(F\rightarrow *\leftarrow S)=\mathcal{L}(Prol(F)\rightarrow *\leftarrow S)$$
for any object $(F\rightarrow *\leftarrow S)$ of $C$. On the other hand, $s^*\mathcal{L}$ is the sheaf associated with the presheaf
$$
\fonc{s^p\mathcal{L}}{C}{\underline{Set}}{(F\rightarrow *\leftarrow S)}{\underrightarrow{lim}\,\mathcal{L}(\mathcal{F}\rightarrow *\leftarrow S')}
$$
where the direct limit is taken over the category of arrows $$(F\rightarrow *\leftarrow S)\longrightarrow (Res(\mathcal{F})\rightarrow *\leftarrow S').$$
But this category has an initial object given by $(Prol(F)\rightarrow *\leftarrow S)$, since $Prol$ is left adjoint to $Res$. We obtain
$$s^{-1}\mathcal{L}(F\rightarrow *\leftarrow S)=\mathcal{L}(Prol(F)\rightarrow *\leftarrow S).$$
Hence $s^{-1}\mathcal{L}$ is already a sheaf isomorphic to $f_*\mathcal{L}$. This identification is functorial hence gives an isomorphism
of functors
$$f_*\cong s^*.$$
\end{proof}

\begin{corollary}\label{cor-section}
Let $\mathcal{A}$ be an abelian object of
$\mathcal{T}/_T\times\mathcal{S}$ and let $\mathcal{A}'$ be an abelian object of $Sh(T)\times\mathcal{S}$. We have the following canonical isomorphisms:
$$H^n(\mathcal{T}/_T\times\mathcal{S},\mathcal{A})\cong H^n(Sh(T)\times\mathcal{S},f_*\mathcal{A})$$
$$H^n(Sh(T)\times\mathcal{S},\mathcal{A}')\cong H^n(\mathcal{T}/_T\times\mathcal{S},f^*\mathcal{A}').$$
\end{corollary}
\begin{proof}
By Lemma \ref{lemma-section}, the Leray spectral sequence associated with the morphism $f:\mathcal{T}/_T\times\mathcal{S}\rightarrow Sh(T)\times\mathcal{S}$ degenerates since $f_*\cong s^*$ is exact. This yields the first isomorphism
$$H^n(\mathcal{T}/_T\times\mathcal{S},\mathcal{A})\cong H^n(Sh(T)\times\mathcal{S},f_*\mathcal{A})=H^n(Sh(T)\times\mathcal{S},s^*\mathcal{A}).$$
Applying this identification to the sheaf $f^*\mathcal{A}'$, we obtain
$$H^n(\mathcal{T}/_T\times\mathcal{S},f^*\mathcal{A}')\cong H^n(Sh(T)\times\mathcal{S},\mathcal{A}')$$
Indeed, we have $f\circ s\cong Id$ hence
$$f_*f^*\mathcal{A}'\cong s^*f^*\mathcal{A}'\cong(f\circ s)^*\mathcal{A}'\cong \mathcal{A}'.$$
\end{proof}

\begin{corollary}\label{cohomology-basechange-overT}
Let $\mathcal{S}$ be any topos. We denote by $p_1:\mathcal{T}\times\mathcal{S}\rightarrow\mathcal{T}$ and $p_2:\mathcal{T}\times\mathcal{S}\rightarrow\mathcal{S}$ the projections.
For any abelian object $\mathcal{A}'$ of $\mathcal{T}\times\mathcal{S}$, one has
$$H^n(\mathcal{T}\times\mathcal{S},\mathcal{A}')\cong H^n(\mathcal{S},p_{2*}\mathcal{A}').$$
For any abelian object $\mathcal{A}$ of $\mathcal{T}$, one has
$$H^n(\mathcal{T}\times\mathcal{S},p_1^*\mathcal{A})\cong H^n(\mathcal{S},\mathcal{A}(*)).$$
\end{corollary}
\begin{proof}
The topos $\mathcal{T}$ is the big topos of the one point space $\{*\}$ while $Sh(*)=\underline{Set}$ is the final topos. The map $$f:\mathcal{T}\times\mathcal{S}\rightarrow Sh(*)\times\mathcal{S}=\underline{Set}\times\mathcal{S}= \mathcal{S}$$
is the second projection $p_2$. Hence one has
\begin{equation}\label{la}
H^n(\mathcal{T}\times\mathcal{S},\mathcal{A}')\cong H^n(\mathcal{S},p_{2*}\mathcal{A}')
\end{equation}
as it follows from Corollary \ref{cor-section}.

There is a pull-back square :
\begin{equation}\begin{CD}
@.\mathcal{T}\times\mathcal{S} @>p_2>>\mathcal{S}\\
@. @Vp_1VV @Ve_{\mathcal{S}}VV @.\\
@.\mathcal{T}@>e_{\mathcal{T}}>>\underline{Set}@.
\end{CD}\end{equation}
The functor $e_{\mathcal{T}}*$ has a right adjoint, so that $e_{\mathcal{T}}*$ commutes with arbitrary inductive limits and in particular with filtered inductive limits. Hence the morphism $e_{\mathcal{T}}$ is tidy (see \cite{johnstone02} C.3.4.2). It follows that the Beck-Chevalley natural transformation
$$e^*_{\mathcal{S}}\circ e_{\mathcal{T}*}\cong p_{2*}\circ p_1^*.$$
is an isomorphism (see \cite{johnstone02} C.3.4.10). But the sheaf
$$p_{2*}p_1^*\mathcal{A}=e^*_{\mathcal{S}}e_{\mathcal{T}*}\mathcal{A}$$ is the constant sheaf on $\mathcal{S}$ associated with the abelian group $\mathcal{A}(*)$, since $e_{\mathcal{T}*}$ is the global section functor and $e^*_{\mathcal{S}}$ is the constant sheaf functor. Applying (\ref{la}) to the sheaf $p_1^*\mathcal{A}$, we obtain
$$
H^n(\mathcal{T}\times\mathcal{S},p_1^*\mathcal{A})\cong H^n(\mathcal{S},p_{2*}p_1^*\mathcal{A})\cong
H^n(\mathcal{S},e^*_{\mathcal{S}}e_{\mathcal{T}*}\mathcal{A})=H^n(\mathcal{S},\mathcal{A}(*))
$$
for any $n\geq0$.
\end{proof}

\begin{lemma}\label{lem-key}
Let $U$ be a contractible topological space and let
$$q:\mathcal{T}\times Sh(U)\longrightarrow\mathcal{T}$$
be the first projection. Then one has
$$R^n(q_*)q^*\tr=0\mbox{ for $n\geq1$.}$$
\end{lemma}
\begin{proof}
The sheaf $R^n(q_*)q^*\tr$ is the sheaf
associated to the presheaf
$$\fonc{P^n(q_*)q^*\tr}{Top}{Ab}{K}{H^n(\mathcal{T}\times
Sh(U),q^*yK,q^*\tr)}
$$
Recall that $Top$ denotes the category of locally compact topological spaces. The category $Top^c$
of compact spaces is a topologically generating family
of the site $(Top,\mathcal{J}_{ls})$. It is therefore enough to show
\begin{equation}\label{vanish-Rnq}
H^n(\mathcal{T}\times
Sh(U),q^*yK,q^*\tr):=H^n((\mathcal{T}\times
Sh(U))/_{q^*yK},q^*\tr\times q^*yK)=0
\end{equation}
for any compact space $K$ and any $n\geq1$. We have immediately
$$(\mathcal{T}\times
Sh(U))/_{q^*yK}=\mathcal{T}/{yK}\times Sh(U).
$$
We denote by ${q}_K:\mathcal{T}/{yK}\times Sh(U)\rightarrow\mathcal{T}/{yK}\rightarrow\mathcal{T}$ the morphism obtained by projection and localization. Equivalently $q_K$ is the composition
$$\mathcal{T}/{yK}\times Sh(U)\cong (\mathcal{T}\times
Sh(U))/_{q^*yK}\longrightarrow (\mathcal{T}\times
Sh(U))\longrightarrow\mathcal{T}.$$ We consider also the map
$$s:Sh(K)\times Sh(U)\longrightarrow\mathcal{T}/yK\times
Sh(U)$$ defined in Lemma \ref{lemma-section}. Then the following
identifications
\begin{align*}
H^n(\mathcal{T}\times Sh(U),q^*yK,\tr)&\cong
H^n(\mathcal{T}/yK\times
Sh(U),q_K^*\tr)\\
&\cong
H^n(Sh(K)\times Sh(U),s^*q_K^*\tr)\\
&\cong
H^n(Sh(K\times U),\tilde{s}^*q_K^*\tr)
\end{align*}
are induced by the following composite morphism of topoi
$$\tilde{s}:Sh(K\times U)\longrightarrow Sh(K)\times Sh(U)\longrightarrow\mathcal{T}/yK\times
Sh(U).$$ Indeed, the first map $Sh(K\times U)\rightarrow Sh(K)\times
Sh(U)$ is an equivalence since $K$ is compact, and the second map induces an isomorphism on
cohomology by Corollary \ref{cor-section}. The commutative
diagram
\begin{equation}\begin{CD}
@.Sh(K\times U) @>>>Sh(K)\times Sh(U)@>s>>\mathcal{T}/yK\times
Sh(U)\\
@. @VVp_1V @VVV @VVq_KV @.\\
@.Sh(K) @>>>\mathcal{T}/_K @>>>\mathcal{T} @. {}
\end{CD}\end{equation}
shows that the sheaf $\tilde{s}^*q_K^*\tr$ on the
product space $K\times U$ is the inverse image of the sheaf
$\mathcal{C}^0(K,\mathbb{R})$, of continuous real functions on $K$,
along the continuous projection $p_1:K\times U\rightarrow K$. In
other words, one has
$$\tilde{s}^*q_K^*\tr=p_1^*\mathcal{C}^0(K,\mathbb{R}).$$
Consider the proper map
$$p_2:K\times U\longrightarrow U.$$
By proper base change, the stalk of the sheaf $R^n(p_{2*})p_1^*\mathcal{C}^0(K,\mathbb{R})$ on $U$ at some point
$u\in U$ is given by
$$(R^n(p_{2*})p_1^*\mathcal{C}^0(K,\mathbb{R}))_u=
H^n(p_2^{-1}(u),p_1^*\mathcal{C}^0(K,\mathbb{R})\mid_{p_2^{-1}(u)})
=H^n(K,\mathcal{C}^0(K,\mathbb{R})).$$ This group is trivial for any $n\geq1$. Indeed
$K$ is compact, in particular paracompact, hence
$\mathcal{C}^0(K,\mathbb{R})$ is fine on $K$. Thus we have
$$R^n(p_{2*})p_1^*\mathcal{C}^0(K,\mathbb{R})=0\mbox{ for any }n\geq1.$$
Applying again proper base change to the proper map $K\rightarrow
*$, we see that $p_{2*}p_1^*\mathcal{C}^0(K,\mathbb{R})$ is the constant sheaf on
$U$ associated with the discrete abelian group
$$C^0(K,\mathbb{R}):=H^0(K,\mathcal{C}^0(K,\mathbb{R})).$$
The Leray spectral sequence associated with the continuous map
$K\times U\rightarrow U$ therefore yields
$$H^n(Sh(K\times U),p_1^*\mathcal{C}^0(K,\mathbb{R}))
\cong
H^n(U,p_{2*}p_1^*\mathcal{C}^0(K,\mathbb{R}))=H^n(U,C^0(K,\mathbb{R}))$$
for any $n\geq0$. But $U$ is contractible hence
$H^n(U,C^0(K,\mathbb{R}))=0$ for $n\geq1$, since sheaf cohomology
with constant coefficients of locally contractible spaces coincides
with singular cohomology, which is in turn homotopy invariant. We
obtain
$$H^n(K\times U,p_1^*\mathcal{C}^0(K,\mathbb{R}))
=H^n(U,C^0(K,\mathbb{R}))=0$$ for any $n\geq1$. The result follows since we have
\begin{align*}
H^n(\mathcal{T}\times Sh(U),q^*yK,\tr)&\cong H^n(Sh(K\times U),\tilde{s}^*q_K^*\tr)\\
&\cong  H^n(Sh(K\times U),p_1^*\mathcal{C}^0(K,\mathbb{R}))\\
&=0
\end{align*}
for any compact space $K$ and any $n\geq1$.
\end{proof}

\subsubsection{}We still denote by $\X$ an irreducible normal scheme which is flat and proper over $\Spec(\bz)$. Recall that
$\X_{\infty}$ is the topological space $\X^{an}/G_{\mathbb{R}}$, and that the Weil-\'etale topos of $\X_\infty$ is defined as follows (see Definition \ref{xinftydef}):
\[ \X_{\infty,W}:=B_{\mathbb{R}}\times Sh(\X_{\infty})\]
\begin{prop}\label{prop-basechange-infty}
Consider the projection morphism
$$\gamma_{\infty}:\X_{\infty,W}= B_{\mathbb{R}}\times Sh(\X_{\infty})\longrightarrow Sh(\X_{\infty}).$$
If $\br$ denotes the constant sheaf on $\X_{\infty}$ associated to the discrete abelian
group $\mathbb{R}$ we have
\[ R^n\gamma_{\infty*}(\tr)\cong\begin{cases} \br & n=0,1\\ 0 & n\geq 2\end{cases}\]
and
\[ R^n\gamma_{\infty*}(\bz)\cong\begin{cases} \bz & n=0\\ 0 & n\geq 1.\end{cases}\]
\end{prop}

\begin{proof}
The sheaf $R^n\gamma_{\infty*}(\tr)$ is the
sheaf on the topological space $\X_{\infty}$ associated to the presheaf
$$
\fonc{P^n(\gamma_{\infty*})(\tr)}{Op(\X_{\infty})}{Ab}{U}{H^n(B_{\mathbb{R}}\times
Sh(\X_{\infty}),U,\tr)}
$$
where $Op(\X_{\infty})$ is the category of open sets of
$\X_{\infty}$. One has
$$H^n(B_{\mathbb{R}}\times
Sh(\X_{\infty}),U,\tr):=H^n(B_{\mathbb{R}}\times
Sh(\X_{\infty})/U,\tr)=H^n(B_{\mathbb{R}}\times
Sh(U),\tr).$$ The family of contractible open subsets $U\subset\X_{\infty}$ forms a topologically generating family of the site $(Op(\X_{\infty}),\mathcal{J}_{op})$, since $\X_{\infty}$ is locally contractible.
It is therefore enough to compute the groups $H^n(B_{\mathbb{R}}\times
Sh(U),\tr)$ for $U$ contractible. For any
contractible open subset $U\subset\X_{\infty}$, we consider the following pull-back square :
\begin{equation}\begin{CD}
@.\mathcal{T}\times Sh(U) @>q>>\mathcal{T}\\
@. @Vl'VV @VlVV @.\\
@.B_{\mathbb{R}}\times Sh(U)@>p>>B_{\mathbb{R}} @.
\end{CD}\end{equation}
Here the vertical arrows $l$ and $l'$ are both localisation maps (recall that
$B_{\mathbb{R}}/E\mathbb{R}\cong\mathcal{T}$), while $p$ and and $q$ are the projections. This yields a canonical isomorphism
$$l^*(Rp_*)\cong (Rq_*)l'^*.$$
By Lemma \ref{lem-key}, we obtain
$$l^*(R^np_*)\tr\cong(R^nq_*)l'^*\tr=(R^nq_*)q^*\tr=0$$
for any $n\geq1$. It follows immediately that
$(R^np_*)\tr=0$ for $n\geq 1$, since
$l^*:B_{\mathbb{R}}\rightarrow\mathcal{T}$ is the forgetful functor
(forget the $y\mathbb{R}$-action).

The contractible topological space $U$ is connected and locally
connected, hence so is the morphism of topoi
$Sh(U)\rightarrow\underline{Set}$. Since connected and locally
connected morphisms are stable under base change (see \cite{johnstone02} C.3.3.15), the first
projection $p:B_{\mathbb{R}}\times Sh(U)\rightarrow B_{\mathbb{R}}$
is also connected and locally connected. In particular, $p^*$ is fully faithful hence we have
$$p_*\tr:= p_*p^*\tr=\tr$$
The Leray spectral sequence
associated to the morphism $p$ therefore yields
\begin{equation}\label{ident-coh-infty-contractible-BR}
H^n(B_{\mathbb{R}}\times Sh(U),p^*\tr)\cong H^n(B_{\mathbb{R}},p_*p^*\tr)
=H^n(B_{\mathbb{R}},\tr).
\end{equation}
But one has $H^n(B_{\mathbb{R}},\tr)=\mathbb{R}$ for $n=0,1$ and $H^n(B_{\mathbb{R}},\tr)=0$ for $n\geq 2$. Hence the sheaf $R^n\gamma_{\infty*}(\tr)$ is the constant sheaf on $\X_{\infty}$ associated to the discrete abelian
group $\mathbb{R}$ for $n=0,1$ and $R^n(\gamma_{\infty*})\tr=0$ for $n\geq 2$.

To compute $R^n\gamma_{\infty*}(\bz)$ recall that for any group object $\G$ in a topos $\E$ and any abelian $\G$-object $\A$ there is a spectral sequence
\[   H^p(H^q(\E/\G^\bullet,\A))\Rightarrow H^{p+q}(B_\G,\A).\]
Applying this to $\G=\br$ in $\E=\T\times Sh(U)$ we note that the classifying topos of $\G$ is just $B_\br\times Sh(U)$ by \cite{diac75}. Hence for $\A=\bz$ we obtain a spectral sequence
\[   H^p(H^q(\T/\br^\bullet\times Sh(U),\bz))\cong H^p(H^q(Sh(\br^\bullet\times U),\bz))\Rightarrow H^{p+q}(B_\br\times Sh(U),\bz)\]
where we have again used Corollary \ref{cor-section} and the fact that the spaces $\br^q$ are locally compact. Now if $U$ is contractible so is $\br^q\times U$ and $H^q(Sh(\br^\bullet\times U),\bz)=\bz$ (resp. $0$) for $q=0$ (resp. $q>0$). The spectral sequence degenerates to an isomorphism
\[ H^p(B_\br\times Sh(U),\bz)\cong H^p(C(\bz))=\begin{cases} \bz & p=0\\ 0 & p>0 \end{cases}\]
where $C(\bz)$ is the complex associated to the constant simplicial abelian group $\bz$ which is quasi-isomorphic to $\bz[0]$. The sheaf
$R^n\gamma_{\infty*}(\bz)$ is associated to the presheaf $U\mapsto H^p(B_\br\times Sh(U),\bz)$ and hence takes the values in the statement of Proposition \ref{prop-basechange-infty}.
\end{proof}

By Proposition \ref{prop-basechange-infty} the Leray spectral sequence for $\gamma_\infty$ induces a long exact sequence
\[ \cdots\to H^i(\X_\infty,\br)\to H^i(\X_{\infty,W},\tr)\to H^{i-1}(\X_\infty,\br)\to \cdots\]
which decomposes into a collection of canonical isomorphisms
\begin{equation}
H^i(\X_{\infty,W},\tr)\cong H^i(\X_\infty,\br)\oplus H^{i-1}(\X_\infty,\br)
\label{xinfty-split}\end{equation}
since $\gamma_\infty$ is canonically split by the morphism of topoi $\sigma:Sh(\X_{\infty})\to B_{\mathbb{R}}\times Sh(\X_{\infty})$ which is the product with $Sh(\X_{\infty})$ of the canonical splitting $\underline{Set}\to\T\to B_\br$ of the canonical projection $B_\br\to\T\to\underline{Set}$. Note here that $\sigma^*$ applied to the adjunction map $\br=\gamma_\infty^*\gamma_{\infty,*}\tr\to\tr$ is an isomorphism $\br=\sigma^*\br\cong\sigma^*\tr\cong\br$.

\subsection{The fundamental class.}\label{funclass}
The map $\mathfrak{f}_{\overline{\X}}:\overline{\X}_W\rightarrow B_{\br}$ induces
an isomorphism
$$\mathfrak{f}^*_{\overline{\X}}:Hom_{c}(\br,\br)=H^1(B_{\br},\tr)\rightarrow H^1(\overline{\X}_W,\tr).$$
\begin{definition}\label{thetadef}
The \emph{fundamental class} is defined as follows:
$$\theta:=\mathfrak{f}^*_{\overline{\X}}(Id_{\br})\in H^1(\overline{\X}_W,\tr).$$
\end{definition}
We consider the sheaf $\tr$ as a ring object on the topos $\overline{\X}_W$. For any $\tr$-module $M$ on $\overline{\X}_W$, one has (see \cite{sga4} V.3.5)
$$Ext^n_{\tr}(\overline{\X}_W,\tr,M)=Ext^n_{\mathbb{Z}}(\overline{\X}_W,\mathbb{Z},M)=H^n(\overline{\X}_W,M).$$
Hence the Yoneda product
$$Ext^1_{\tr}(\overline{\X}_W,\tr,\tr)\times Ext^n_{\tr}(\overline{\X}_W,\tr,M)\longrightarrow Ext^{n+1}_{\tr}(\overline{\X}_W,\tr,M)$$
gives a morphism
$$H^1(\overline{\X}_W,\tr)\times H^n(\overline{\X}_W,M)\longrightarrow H^{n+1}(\overline{\X}_W,M).$$
Thus the fundamental class $\theta\in H^1(\overline{\X}_W,\tr)$ defines a $\mathbb{R}$-linear map of
$\mathbb{R}$-vector spaces
\begin{equation}\label{cupproduct}
\cup\theta:H^n(\overline{\X}_W,M)\longrightarrow H^{n+1}(\overline{\X}_W,M).
\end{equation}
Furthermore, the \'etale sheaf $R^n\gamma_{\overline{\X},*}(M)$ is the sheaf associated with the presheaf
$$
\fonc{P^n\gamma_{\overline{\X},*}(M)}{Et_{\overline{\X}}}{Ab}{\overline{\mathcal{U}}}{H^n(\overline{\mathcal{U}}_W,M)}
$$
For any $\overline{\mathcal{U}}$ \'etale over $\overline{\X}$ we define $\theta_{\overline{\mathcal{U}}}$ to be the pull-back of $\theta$ in $H^1(\overline{\mathcal{U}}_W,\tr)$. Then cup product with the fundamental class $\theta_{\overline{\mathcal{U}}}$ gives a map $H^n(\overline{\mathcal{U}}_W,M)\rightarrow H^{n+1}(\overline{\mathcal{U}}_W,M)$, which is functorial in $\overline{\mathcal{U}}$. In other words, we have a morphism
of presheaves $P^n\gamma_{\overline{\X},*}(M)\rightarrow P^{n+1}\gamma_{\overline{\X},*}(M)$. Applying the associated sheaf functor, we obtain a morphism of sheaves
\begin{equation}\label{cupproduct-on-sheaves}
{\cup\theta}:R^n\gamma_{\overline{\X},*}(M)\longrightarrow R^{n+1}\gamma_{\overline{\X},*}(M)
\end{equation}
More precisely, the map (\ref{cupproduct}) is induced by a morphism of complexes
\begin{equation}\label{cupproduct-cplexes}
\cup\theta: R\Gamma_{\overline{\X}_W}(M)\longrightarrow R\Gamma_{\overline{\X}_W}(M)[1].
\end{equation}
Consider now the complex of \'etale sheaves $R\gamma_{\overline{\X},*}(M)$. For any \'etale $\overline{\X}$-scheme $\overline{\mathcal{U}}$, the complex of abelian groups $R\gamma_{\overline{\X},*}(M)(\overline{\mathcal{U}})$ is quasi-isomorphic to $R\Gamma_{\overline{\mathcal{U}}_W}(M)$. Hence cup product with the canonical classes $\theta_{\overline{\mathcal{U}}}$ yields a morphism of complexes of sheaves
\begin{equation}\label{cupproduct-on-complexofsheaves}
{\cup\theta}:R\gamma_{\overline{\X},*}(M)\longrightarrow R\gamma_{\overline{\X},*}(M)[1]
\end{equation}
The morphisms of complexes (\ref{cupproduct-cplexes}) and (\ref{cupproduct-on-complexofsheaves}) above are well defined in the corresponding derived category. Moreover, the morphisms (\ref{cupproduct}), (\ref{cupproduct-on-sheaves}), (\ref{cupproduct-cplexes}), and (\ref{cupproduct-on-complexofsheaves}) are functorial in $M$.

Finally, the morphism (\ref{cupproduct-on-complexofsheaves}) is compatible with (\ref{cupproduct}) in the following sense.
Under the canonical isomorphisms $H^n(\overline{\X}_W,M)=\mathbb{H}^n(\overline{\X}_{et},R\gamma_{\overline{\X},*}(M))$ and $H^{n+1}(\overline{\X}_W,M)=\mathbb{H}^n(\overline{\X}_{et},R\gamma_{\overline{\X},*}(M)[1])$,
the morphism induced by (\ref{cupproduct-on-complexofsheaves}) on hypercohomology groups
\begin{equation}\label{cupproduct-hypercohomology}
\mathbb{H}^n(\overline{\X}_{et},R\gamma_{\overline{\X},*}(M))\longrightarrow \mathbb{H}^n(\overline{\X}_{et},R\gamma_{\overline{\X},*}(M)[1])
\end{equation}
coincide with the morphism (\ref{cupproduct}).

Consider now the open-closed decomposition
$$\varphi:\X_{et}\longrightarrow\overline{\X}_{et}\longleftarrow Sh(\X_{\infty}):u_{\infty}$$
given by Corollary \ref{closed/open-decomp-etale}. The morphism $\gamma:\overline{\X}_W\rightarrow\overline{\X}_{et}$
gives pull-back squares
\begin{equation*}\begin{CD}
@.\X_W@>\gamma_{\X}>>\X_\et \\
@. @V{\phi}VV @V{\varphi}VV @.\\
@.\overline{\X}_W@>\gamma_{\overline{\X}}>>\overline{\X}_\et @.
\end{CD}\end{equation*}
and
\begin{equation*}\begin{CD}
@.\X_{\infty,W}@>\gamma_{\infty}>>Sh(\X_{\infty}) \\
@. @Vi_{\infty}VV @Vu_{\infty}VV@.\\
@.\overline{\X}_W@>\gamma_{\overline{\X}}>>\overline{\X}_\et @.
\end{CD}\end{equation*}
The second square is indeed a pull-back, as can be seen from the following commutative diagram:
\begin{equation*}\begin{CD}
@.\X_{\infty,W} @>\gamma_{\infty}>>Sh(\X_{\infty})@>>>Sh(\infty)=\underline{Set}\\
@. @Vi_{\infty}VV @Vu_{\infty}VV @VVV @.\\
@.\overline{\X}_W @>\gamma_{\overline{\X}}>>\overline{\X}_{et} @>>>\overline{\Spec(\bz)}_{et} @. {}
\end{CD}\end{equation*}
The right hand side square and the total square are both pull-backs by Corollary \ref{pull-ba} and Proposition \ref{prop-pullback-archW} respectively. It follows that the left hand side square is a pull-back as well.

\begin{theorem}\label{thm-basechange-cpctsupp}
There is an isomorphism
$R^n\gamma_{\overline{\X}*}(\phi_!\tr)\cong\varphi_!\tr$ for $n=0,1$,
and $R^n\gamma_{\overline{\X}*}(\phi_!\tr)=0$ for $n\geq2$. Under these identifications, the morphism $$\cup\theta:R^0\gamma_{\overline{\X}*}(\phi_!\tr)\longrightarrow R^1\gamma_{\overline{\X}*}(\phi_!\tr)$$ given by cup product with the fundamental class, is the identity of the sheaf $\varphi_!\tr$.
\end{theorem}
\begin{proof}
We have an exact sequence of abelian sheaves on $\overline{\X}_W$:
$$0\rightarrow\phi_!\tr\rightarrow\tr\rightarrow i_{\infty*}\tr\rightarrow 0$$
Applying the functor $R\gamma_{\overline{\X}*}$, we obtain an exact sequence of \'etale sheaves
$$0\rightarrow\gamma_{\overline{\X}*}\phi_!\tr\rightarrow \gamma_{\overline{\X}*}\tr\rightarrow \gamma_{\overline{\X}*}i_{\infty*}\tr\rightarrow R^1\gamma_{\overline{\X}*}(\phi_!\tr)\rightarrow R^1\gamma_{\overline{\X}*}(\tr)\rightarrow R^1\gamma_{\overline{\X}*}(i_{\infty*}\tr)\rightarrow...$$
But we have canonical isomorphisms
\begin{equation}\label{iso-given-by-clsedsubtopos}
R^n\gamma_{\overline{\X}*}(i_{\infty*}\tr)\cong R^n(\gamma_{\overline{\X}*}i_{\infty*})\tr\cong
R^n(u_{\infty*}\gamma_{\infty*})\tr\cong u_{\infty*}R^n\gamma_{\infty*}(\tr)
\end{equation}
for any $n\geq0$, since the direct image of a closed embedding of topoi is exact. Therefore, by Proposition \ref{prop-global-base-change} and Proposition \ref{prop-basechange-infty}, we obtain an exact sequence
$$0\rightarrow\gamma_{\overline{\X}*}\phi_!\tr\rightarrow\mathbb{R}\rightarrow u_{\infty*}\mathbb{R}\rightarrow R^1\gamma_{\overline{\X}*}(\phi_!\tr)\rightarrow\mathbb{R}\rightarrow u_{\infty*}\mathbb{R}\rightarrow R^2\gamma_{\overline{\X}*}(\phi_!\tr)\rightarrow 0$$
and $R^n\gamma_{\overline{\X}*}(\phi_!\tr)=0$ for $n\geq3$. The map $\mathbb{R}\rightarrow u_{\infty*}\mathbb{R}$ is surjective since $u_{\infty}$ is a closed embedding. Hence we have an exact sequence
$$0\rightarrow R^n\gamma_{\overline{\X}*}(\phi_!\tr)\rightarrow\mathbb{R}\rightarrow u_{\infty*}\mathbb{R}\rightarrow 0$$
for $n=0,1$ and $R^n\gamma_{\overline{\X}*}(\phi_!\tr)=0$ for $n\geq2$. The first claim of the theorem follows.

For any connected \'etale $\overline{\X}$-scheme $\overline{\mathcal{U}}$, we have a commutative square of $\mathbb{R}$-vector spaces
\begin{equation*}\begin{CD}
@.H^0(\overline{\mathcal{U}}_W,\tr)@>\cup\theta_{\overline{\mathcal{U}}}>>H^1(\overline{\mathcal{U}}_W,\tr) \\
@. @AAA @AAA @.\\
@.H^0(B_{\mathbb{R}},\tr)=\mathbb{R}@>\cup{Id_{\mathbb{R}}}>>H^1(B_{\mathbb{R}},\tr)\cong\mathbb{R} @.
\end{CD}\end{equation*}
where the vertical maps are isomorphisms by Corollary \ref{cor-cohomology-of-U}. The $\mathbb{R}$-linear map
\begin{equation}\label{cup-product-BR}
\cup{Id_{\mathbb{R}}}:H^0(B_{\mathbb{R}},\tr)=\mathbb{R}\longrightarrow H^1(B_{\mathbb{R}},\tr)=Hom_{cont}(\mathbb{R},\mathbb{R})
\end{equation}
sends $1\in\mathbb{R}$ to $Id_{\mathbb{R}}$. Under the identification $$H^1(B_{\mathbb{R}},\tr)=Hom_{cont}(\mathbb{R},\mathbb{R})\cong\mathbb{R}$$
which maps $f:\mathbb{R}\rightarrow\mathbb{R}$ to $f(1)$, the morphism (\ref{cup-product-BR}) is the identity of $\mathbb{R}$.
Hence the morphism
$$\cup\theta_{\overline{\mathcal{U}}}:H^0(\overline{\mathcal{U}}_W,\tr)=\mathbb{R}\longrightarrow H^1(\overline{\mathcal{U}}_W,\tr)\cong\mathbb{R}$$
is just the identity, for any connected \'etale $\overline{\X}$-scheme $\overline{\mathcal{U}}$. It follows that the morphism of sheaves defined in (\ref{cupproduct-on-sheaves})
\begin{equation}\label{cuptheta-identity-on-R}
\cup\theta:R^0\gamma_{\overline{\X}*}(\tr)=\tr\longrightarrow R^1\gamma_{\overline{\X}*}(\tr)\cong\tr
\end{equation}
is the identity of the sheaf $\tr$.

The same argument is valid for the sheaf $i_{\infty*}\tr$. The composite morphism
$$p:\X_{\infty,W}=B_{\mathbb{R}}\times Sh(\X_{\infty})\longrightarrow\overline{\X}\longrightarrow B_{\mathbb{R}}$$
is the first projection. We consider the fundamental class
$$\theta_{\infty}:=p^*(Id_{\br})=i_{\infty}^*(\theta)\in H^1(\X_{\infty,W},\tr).$$
Then the morphism $$\cup\theta:R^0\gamma_{\overline{\X}*}(i_{\infty*}\tr)\longrightarrow R^1\gamma_{\overline{\X}*}(i_{\infty*}\tr)$$
coincides, via the canonical isomorphism (\ref{iso-given-by-clsedsubtopos}), with the morphism $u_{\infty*}R^0\gamma_{\infty*}(\tr)\rightarrow u_{\infty*}R^1\gamma_{\infty*}(\tr)$ induced by
$$\cup\theta_{\infty}:R^0\gamma_{\infty*}(\tr)\longrightarrow R^1\gamma_{\infty*}(\tr).$$
But for any contractible open subset $U\subset\X_{\infty}$, one has a commutative square
\begin{equation*}\begin{CD}
@.H^0(\X_{\infty,W},U,\tr)@>\cup\theta_{\infty}>>H^1(\X_{\infty,W},U,\tr) \\
@. @AAA @AAA @.\\
@.H^0(B_{\mathbb{R}},\tr)=\mathbb{R}@>\cup{Id_{\mathbb{R}}}>>H^1(B_{\mathbb{R}},\tr)\cong\mathbb{R} @.
\end{CD}\end{equation*}
where all the maps are isomorphisms, as it follows from (\ref{ident-coh-infty-contractible-BR}). Hence the map
$$\cup\theta_{\infty}:\tr=\gamma_{\infty*}(\tr)\longrightarrow R^1\gamma_{\infty*}(\tr)\cong\tr$$
is the identity, and so is the morphism
\begin{equation}\label{cuptheta-identity-on-inftyR}
\cup\theta:R^0\gamma_{\overline{\X}*}(i_{\infty*}\tr)=u_{\infty*}\tr\longrightarrow R^1\gamma_{\overline{\X}*}(i_{\infty*}\tr)\cong u_{\infty*}\tr.
\end{equation}
The morphism (\ref{cupproduct-on-sheaves}) is functorial hence $\cup\theta$ gives a morphism of exact sequences
from
$$0\rightarrow\phi_!\tr\rightarrow\tr\rightarrow i_{\infty*}\tr\rightarrow 0$$
to
$$0\rightarrow R^1\gamma_{\overline{\X}*}(\phi_!\tr)\rightarrow R^1\gamma_{\overline{\X}*}(\tr)\rightarrow R^1\gamma_{\overline{\X}*}(i_{\infty*}\tr)\rightarrow0$$
But the morphisms (\ref{cuptheta-identity-on-R}) and (\ref{cuptheta-identity-on-inftyR}) are both given by the identity map, hence so is the morphism
$$\cup\theta:R^0\gamma_*(\phi_!\tr)=\varphi_!\tr\longrightarrow R^1\gamma_*(\phi_!\tr)\cong\varphi_!\tr.$$
\end{proof}

\begin{definition}
For any abelian sheaf $\mathcal{A}$ on $\X_W$, the compact support cohomology groups $H^i_c(\X_W,\tr)$ are defined as follows:
$$H^i_c(\X_W,\mathcal{A}):=H^i(\overline{\X}_W,\phi_!\mathcal{A})$$
\end{definition}

\begin{theorem} Assume that $\X$ is irreducible, normal, flat and proper over $\Spec(\bz)$. The compact support cohomology groups $H^i_c(\X_W,\tr)$
are finite dimensional vector spaces over $\br$, vanish for almost
all $i$ and satisfy
\[ \sum_{i\in\bz}(-1)^i\dim_\br H^i_c(\X_W,\tr)=0.\]
Moreover, the complex of $\br$-vector spaces
\[
\cdots\xrightarrow{\cup\theta}H^i_c(\X_W,\tr)\xrightarrow{\cup\theta}H^{i+1}_c(\X_W,\tr)\xrightarrow{\cup\theta}\cdots\]
is acyclic.
\label{ac-theo}\end{theorem}

\begin{proof}
Consider the Leray spectral sequence
$$H^p(\overline{\X}_{et},R^q\gamma_{\overline{\X}*}(\phi_!\tr))\Longrightarrow H^{p+q}(\overline{\X}_{W},\phi_!\tr)$$
given by the morphism $\gamma_{\overline{\X}}$. This spectral sequence
yields
$$H^{0}(\overline{\X}_{W},\phi_!\tr)=H^{0}(\overline{\X}_{et},\varphi_!\mathbb{R})=0$$
and a long exact sequence
\begin{equation*}\begin{CD}
@.0 @>>>H^{1}(\overline{\X}_{et},R^0\gamma_{\overline{\X}*}(\phi_!\mathbb{R})) @>>>H^{1}(\overline{\X}_{W},\phi_!\tr)@>>>H^{0}(\overline{\X}_{et},R^1\gamma_{\overline{\X}*}(\phi_!\mathbb{R}))\\
@. @. @. @VV{\cup\theta}V @. @. \\
@.... @>>>H^{2}(\overline{\X}_{et},R^0\gamma_{\overline{\X}*}(\phi_!\mathbb{R})) @>>>H^{2}(\overline{\X}_{W},\phi_!\tr) @>>>H^{1}(\overline{\X}_{et},R^1\gamma_{\overline{\X}*}(\phi_!\mathbb{R}))  @. {}\\
@. @. @. @VV{\cup\theta}V @. @. \\
@.... @>>>H^{3}(\overline{\X}_{et},R^0\gamma_{\overline{\X}*}(\phi_!\mathbb{R})) @>>>H^{3}(\overline{\X}_{W},\phi_!\tr) @>>>H^{2}(\overline{\X}_{et},R^1\gamma_{\overline{\X}*}(\phi_!\mathbb{R})) @. {}\\
@. @. @. @. @. @. \\
@.... @>>>H^{4}(\overline{\X}_{et},R^0(\gamma_{\overline{\X}*})\phi_!\mathbb{R})@>>>...
\end{CD}\end{equation*}
Here the vertical maps $\cup\theta$ are given by cup product with the fundamental class. More precisely, the morphism (\ref{cupproduct-on-complexofsheaves})
$$R\gamma_{\overline{\X},*}(\phi_!\mathbb{R})\longrightarrow R\gamma_{\overline{\X},*}(\phi_!\mathbb{R})[1].$$
induces a morphism of spectral sequences. This morphism of spectral sequences induces in turn these vertical maps $\cup\theta$. It follows that the composite map
\begin{equation}\label{give-section}
H^{i}(\overline{\X}_{et},R^0\gamma_{\overline{\X}*}(\phi_!\mathbb{R}))\rightarrow
H^{i}(\overline{\X}_{W},\phi_!\tr)\xrightarrow{\cup\theta} H^{i+1}(\overline{\X}_{W},\phi_!\tr)\rightarrow H^{i}(\overline{\X}_{et},R^1\gamma_{\overline{\X}*}(\phi_!\mathbb{R}))
\end{equation}
is induced by the isomorphism of sheaves
$$R^0\gamma_{\overline{\X}*}(\phi_!\tr)=\varphi_!\mathbb{R}\xrightarrow{\cup\theta}R^1\gamma_{\overline{\X}*}(\phi_!\tr)
\cong\varphi_!\mathbb{R}.$$
Hence the map (\ref{give-section}) is an isomorphism for any $i\geq0$, by Theorem \ref{thm-basechange-cpctsupp}. This yields a section to the map
$$H^{i+1}(\overline{\X}_{W},\phi_!\tr)\longrightarrow H^{i}(\overline{\X}_{et},R^1\gamma_{\overline{\X}*}(\phi_!\mathbb{R})).$$
It follows that the long exact sequence above
decomposes into a collection of canonical isomorphisms
\begin{align}
H^{i}(\overline{\X}_{W},\phi_!\tr)&\cong H^{i}(\overline{\X}_{et},R^0\gamma_{\overline{\X}*}(\phi_!\tr))\oplus H^{i-1}(\overline{\X}_{et},R^1\gamma_{\overline{\X}*}(\phi_!\tr))\\
 &\cong H^{i}(\overline{\X}_{et},\varphi_!\mathbb{R})\oplus H^{i-1}(\overline{\X}_{et},\varphi_!\mathbb{R})\\
\label{cohomology-direct-sum}& \cong H_c^{i}(\X_{et},\mathbb{R})\oplus H_c^{i-1}(\X_{et},\mathbb{R})
\end{align}
for any $i\geq1$. By Proposition \ref{prop-cpctspp-etale-coh}, the $\br$-vector space $H_c^{i}(\X_{et},\mathbb{R})$ is finite dimensional and zero for $i$ large. Hence we have
$$\dim_\br H^i_c(\X_W,\tr)=\dim_\br H^{i}_c(\X_{et},\br)+\dim_\br H^{i-1}_c(\X_{et},\br)$$
and $$\sum_{i\in\bz}(-1)^i\dim_\br H^i_c(\X_W,\tr)=0.$$
Under the identification (\ref{cohomology-direct-sum}), the morphism given by cup product with the fundamental class
$$H^i_c(\X_W,\tr)\xrightarrow{\cup\theta}H^{i+1}_c(\X_W,\tr)$$
is obtained by composing the projection with the inclusion as follows:
\begin{equation}\label{cup-prod-decompose}
H^{i}_c(\X_W,\tr)\twoheadrightarrow H^{i}_c(\X_{et},\mathbb{R})\hookrightarrow H^{i+1}_c(\X_W,\tr).
\end{equation}
It follows immediately from (\ref{cohomology-direct-sum}) and (\ref{cup-prod-decompose}) that the complex of $\br$-vector spaces
\[
\cdots\xrightarrow{\cup\theta}H^i_c(\X_W,\tr)\xrightarrow{\cup\theta}H^{i+1}_c(\X_W,\tr)\xrightarrow{\cup\theta}\cdots\]
is acyclic.
\end{proof}

\begin{remark}\label{rem-cohomology-direct-sum}
For any $i\geq1$, there is a canonical isomorphism of $\br$-vector spaces
$$H_c^{i}(\X_{W},\tr)\cong H_c^{i}(\X_{et},\mathbb{R})\oplus  H_c^{i-1}(\X_{et},\mathbb{R})$$
\end{remark}

\begin{prop} Assume that $\X$ is irreducible, normal, flat and proper over $\Spec(\bz)$. Then one has
\begin{align*}
\sum_{i\in\bz}(-1)^ii\dim_\br H^i_c(\X_W,\tr)&=\sum_{i\in\bz}(-1)^{i+1}\dim_\br H^i_c(\X_\et,\br)\\
&= -1+\sum_{i\in\bz}(-1)^i\dim_\br H^i(\X_{\infty},\br)\\
&= -1+\sum_{i\in\bz}(-1)^i\dim_\br H^i(\X^{an},\br)^+\\
\end{align*}
\label{hc}\end{prop}

\begin{proof}
The first equality (respectively the second) follows from Remark \ref{rem-cohomology-direct-sum}
(respectively from Proposition \ref{prop-cpctspp-etale-coh}). To prove the third, we consider the morphism of topoi
$$(\pi^*,\pi_*^{G_{\br}}):Sh(G_{\br},\X^{an})\rightarrow
Sh(\X_{\infty})$$
given by the quotient map $\pi:\X^{an}\rightarrow \X^{an}/G_{\br}$, where $Sh(G_{\br},\X^{an})$ is the topos of $G_{\br}$-equivariant
sheaves on the space $\X^{an}$. The constant sheaf $\br$ on $\X^{an}$ is endowed with its $G_{\br}$-equivariant structure. For any $n\geq1$, the stalk of $R^n(\pi_*^{G_{\br}})\mathbb{R}$ at some fixed point $x\in\X(\mathbb{R})\subset\X^{\infty}$ is
the abelian group $H^n(G_{\br},\br)$, which is zero since $\br$ is uniquely divisible. This gives
$$R^n(\pi_*^{G_{\br}})\mathbb{R}=0\mbox{ for $n\geq1$}$$
and a canonical isomorphism
$$H^n(\X_{\infty},\br)\cong H^n(Sh(G_{\br},\X^{an}),\br)$$
for any $n\geq0$. But the spectral sequence
$$H^p(G_{\br},H^q(\X^{an},\br))\Longrightarrow H^{p+q}(Sh(G_{\br},\X^{an}),\br)$$
degenerates and gives an isomorphism
$$ H^{n}(Sh(G_{\br},\X^{an}),\br)\cong H^0(G_{\br},H^n(\X^{an},\br))=:H^n(\X^{an},\br)^+$$
for any $n$. The result follows.
\end{proof}

\section{Relationship to the Zeta-function}\label{zeta}

\subsection{Motivic L-functions} We first recall the expected properties of motivic
L-functions \cite{serre69}. For any smooth proper scheme $X/\bq$ of pure
dimension $d$ and $0\leq i\leq 2d$ one defines the $L$-function
\[  L(h^i(X),s)=\prod_p L_p(h^i(X),s)\]
as an Euler product over all primes $p$ where
\[L_p(h^i(X),s)= P_p(h^i(X),p^{-s})^{-1}\]
and
\[ P_p(h^i(X),T)= \mydet_{\bq_l}\bigl(1-\mathrm{Frob}_p^{-1}\cdot T\vert
H^i(X_{\bar{\bq},et},\bq_l)^{I_p}\bigr)\] is a polynomial (conjecturally) with
rational coefficients independent of the prime
$l\neq p$. By \cite{weilii} this product converges for
$\Re(s)>\frac{i}{2}+1$. Set
\[\Gamma_\br(s)=\pi^{-s/2}\Gamma(\frac{s}{2});\quad \Gamma_{\BC}(s)=2(2\pi)^{-s}\Gamma(s)\]
and
\[L_\infty(h^i(X),s)=\prod_{p<q}\Gamma_{\BC}(s-p)^{h^{p,q}}\cdot\prod_{p=\frac{i}{2}}
\Gamma_{\br}(s-p)^{h^{p,+}}\Gamma_{\br}(s-p+1)^{h^{p,-}}\] where
$H^i(X(\BC),\BC)\cong\bigoplus_{p+q=i}H^{p,q}$ is the Hodge
decomposition,
\[h^{p,q}=\dim_{\BC}H^{p,q};\quad h^{p,\pm}=\dim_{\BC}(H^{p,p})^{F_\infty=\pm
(-1)^p}\] and $F_\infty$ is the map induced by complex conjugation
on the manifold $X(\BC)$. Here the product over $p=\frac{i}{2}$ is
understood to be empty for odd $i$. The completed $L$-function
\[ \Lambda(h^i(X),s)=L_\infty(h^i(X),s)L(h^i(X),s) \]
is expected to meromorphically continue to all $s$ and satisfy a
functional equation \begin{equation}
\Lambda(h^i(X),s)=\epsilon(h^i(X),s)\Lambda(h^{2d-i}(X),d+1-s).\label{fe}\end{equation}
Here $\epsilon(h^i(X),s)$ is the product of a constant and an
exponential function in $s$, in particular nowhere vanishing.

\begin{lemma} Assuming meromorphic continuation and the functional equation we have
\[
\ord_{s=0}L(h^i(X),s)=\begin{cases}-t+\dim_{\BC}H^0(X(\BC),\BC)^{F_\infty=1}
& i=0\\ \dim_{\BC}H^i(X(\BC),\BC)^{F_\infty=1} & i>0.\end{cases}\]
where $t$ is the number of connected components of the scheme $X$.
\label{ordat0}\end{lemma}

\begin{proof} For $i>0$ the point $d+1>\frac{2d-i}{2}+1$ lies in the
region of absolute convergence of $L(h^{2d-i}(X),s)$ so that
$L(h^{2d-i}(X),d+1)\neq 0$. The Gamma-function has no zeros and has
simple poles precisely at the non-positive integers. For $p+q=2d-i$
and $p<q$ we have $p<d-\frac{i}{2}$, hence $\Gamma_\BC(d+1-p)\neq
0$. For $p=d-\frac{i}{2}$ we likewise have $\Gamma_\br(d+1-p)\neq 0$
and $\Gamma_\br(d+1+1-p)\neq 0$. Hence
$L_\infty(h^{2d-i}(X),d+1)\neq 0$ and the functional equation shows
$\Lambda(h^i(X),0)\neq 0$, i.e.
\begin{align}\ord_{s=0}L(h^i(X),s)=&-\ord_{s=0}L_\infty(h^i(X),s)\label{compi}\\
=&\sum_{p<q}h^{p,q}+\sum_{p=\frac{i}{2}}h^{p,\pm}=\dim_{\BC}H^i(X(\BC),\BC)^{F_\infty=1}\notag\end{align}
where this last identity follows from $F_\infty(H^{p,q})=H^{q,p}$
and the sign $\pm$ in $h^{p,\pm}$ is the one for which
$\pm(-1)^p=1$. Indeed, $\Gamma_\br(s-p)$ (resp. $\Gamma_\br(s-p+1)$)
has a simple pole at $s=0$ precisely for even (resp. odd) $p$.

For $i=0$ the function
$$L(h^0(X),s)=\zeta_{K_1}(s)\cdots\zeta_{K_t}(s)$$
is a product of Dedekind Zeta-functions where $H^0(X,\co_X)=K_1\times\cdots\times K_t$ is the ring of global regular functions on $X$ and the $K_i$ are number fields. It is classical that $\ord_{s=1}\zeta_{K_j}(s)=-1$ and therefore
$$\ord_{s=0}\Lambda(h^0(X),s)=\ord_{s=1}\Lambda(h^0(X),s)=\sum_{j=1}^t\ord_{s=1}\zeta_{K_j}(s)=-t.$$
Hence (\ref{compi}) holds for $i=0$ with $-t$ added to the right hand side.

\end{proof}

\subsection{Zeta-functions} For any separated scheme $\X$ of finite type over
$\Spec(\bz)$ one defines a Zeta-function
\[ \zeta(\X,s):=\prod_{x\in X^{cl}}\frac{1}{1-N(x)^{-s}}=\prod_p\zeta(\X_{\mathbb F_p},s)\]
as an Euler product over all closed points. By Grothendieck's
formula \cite{miletale}[Thm. 13.1]
\[\zeta(\X_{\mathbb F_p},s)=\prod_{i=0}^{2\dim(\X_{\mathbb F_p})}
\mydet_{\bq_l}\bigl(1-\mathrm{Frob}_p^{-1}\cdot p^{-s}\vert
H^i_c(\X_{\bar{\mathbb F}_p,{et}},\bq_l)\bigr)^{(-1)^{i+1}}.\] If
$\X_\bq\to\Spec(\bq)$ is smooth and proper of relative dimension
$d$, there will be an open subscheme $U\subseteq\Spec(\bz)$ on which
$\X_U\to U$ is smooth and proper. By smooth and proper base change
we have for $p\in U$
$$H^i_c(\X_{\bar{\mathbb F}_p,{et}},\bq_l)\cong H^i(\X_{\bar{\mathbb F}_p,{et}},\bq_l)\cong
H^i(\X_{\bar{\bq},et},\bq_l)\cong
H^i(\X_{\bar{\bq},et},\bq_l)^{I_p}$$ and therefore \begin{equation}
\zeta(\X,s)=\prod_{p\notin
U}E_p(s)\prod_{i=0}^{2d}L(h^i(\X_\bq),s)^{(-1)^i}\label{zetamot}\end{equation}
where
\[E_p(s)=\prod_{i=0}^{\infty}\left(\frac{\mydet_{\bq_l}\bigl(1-\mathrm{Frob}_p^{-1}\cdot
p^{-s}\vert
H^i(\X_{\bar{\bq},et},\bq_l)^{I_p}\bigr)}{\mydet_{\bq_l}\bigl(1-\mathrm{Frob}_p^{-1}\cdot
p^{-s}\vert H^i_c(\X_{\bar{\mathbb
F}_p,{et}},\bq_l)\bigr)}\right)^{(-1)^i} \] is a rational function
in $p^{-s}$.

\begin{theorem} Let $\X$ be a regular scheme, proper and
flat over $\Spec(\bz)$. Assume that the L-functions
$L(h^i(\X_\bq),s)$ can be meromorphically continued and satisfy the
functional equation (\ref{fe}). Then
\[ \ord_{s=0}\zeta(\X,s)=\sum_{i\in\bz}(-1)^i\cdot i\cdot\dim_\br
H^i_c(\X_W,\tr).\]
\label{b-theo}\end{theorem}

\begin{proof} Note that regularity of $\X$ implies that $\X_\bq\to\Spec(\bq)$ is smooth. By Lemma \ref{ordat0} and Proposition \ref{hc} we have
\begin{align*}\ord_{s=0}\prod_{i\in\bz}L(h^i(\X_\bq),s)^{(-1)^i}=
&-t+\sum_{i\in\bz}(-1)^i\dim_{\BC}H^i(\X_\bq(\BC),\BC)^{F_\infty=1}\\
=&-t+\sum_{i\in\bz}(-1)^i\dim_{\br}H^i(\X^{an},\br)^{F_\infty=1}\\
=&\sum_{i\in\bz}(-1)^i\cdot i\cdot\dim_\br H^i_c(\X_W,\tr)
\end{align*}
and in view of (\ref{zetamot}) it remains to show that
$\ord_{s=0}E_p(s)=0$ for all $p$ (or just $p\notin U$). This follows
from the fact that the $\mathrm{Frob}_p^{-1}$ eigenvalue $1$ (of
weight $0$) has the same multiplicity on $H^i_c(\X_{\bar{\mathbb
F}_p,{et}},\bq_l)=H^i(\X_{\bar{\mathbb F}_p,{et}},\bq_l)$ and on
$H^i(\X_{\bar{\bq},et},\bq_l)^{I_p}$ by part b) of Theorem
\ref{main} in the next section.
\end{proof}

\begin{corollary} Let $F$ be a totally real number field and $\X$ a proper, regular model of a Shimura curve
over $F$, or of $E\times E\times \cdots \times E$ where $E$ is an
elliptic curve over $F$. Then
\[ \ord_{s=0}\zeta(\X,s)=\sum_{i\in\bz}(-1)^i\cdot i\cdot\dim_\br
H^i_c(\X_W,\tr).\]
\end{corollary}

\begin{proof} For any Shimura curve $X$, by the now classical results of Eichler,
Shimura, Deligne, Carayol and others, $L(h^1(X),s)$ is a product of
$L$-functions associated to weight $2$ cusp forms for a suitable
arithmetic subgroup of $\text{PSL}_2(\br)$ associated to $X$, hence
satisfies (\ref{fe}). It is moreover well known that any curve
always has a proper regular model.

By the Kuenneth formula we have
$$h^i(E^d)\cong
\bigoplus_{\substack{i_0+i_1+i_2=d\\i_1+2i_2=i}} h^0(E)^{\otimes
i_0}\otimes h^1(E)^{\otimes i_1}\otimes h^2(E)^{\otimes
i_2}\cong\bigoplus_{\substack{i_0+i_1+i_2=d\\i_1+2i_2=i}}
h^1(E)^{\otimes i_1}(-i_2)
$$
and each tensor power $h^1(E)^{\otimes i_1}$ is a direct sum of Tate
twists of symmetric powers $\text{Sym}^k h^1(E)$. But for elliptic
curves $E$ over totally real fields $F$ the meromorphic continuation
and functional equation of $L(\text{Sym}^k h^1(E)/F,s)$ follows from
recent deep results of Harris, Taylor, Shin et al (see
\cite{blght}[Cor. 8.8]). We remark that a proper regular model $\X$
of $E^d$ certainly exists if $E$ has semistable reduction at all
primes since then the product singularities of $\E^d$, where $\E$ is
a proper regular model of $E$, can be resolved \cite{scholl90}.
\end{proof}

\begin{theorem} Let $\X$ be a smooth proper variety over a
finite field. Then a)-f) in the introduction hold for $\X$.
\label{charp-theo}\end{theorem}

\begin{proof} This was proved for $\X_W^{sm}$ in
\cite{geisser04}[Thm. 9.1] since one clearly has
$$H^i(\X^{sm}_W,\bz)\otimes_\bz\br\cong H^i(\X^{sm}_W,\br).$$ But in
view of Corollary \ref{cohomology-basechange-overT} (see also Corollary \ref{big-small} and the remark after it) we have
\[ H^i(\X_W,\bz)\cong H^i(\X_W^{sm},\bz);\quad H^i(\X_W,\tr)\cong H^i(\X_W^{sm},\br)\]
when $\X_W$ is defined by Definition \ref{xwdef}. Note here that our fundamental class $\theta$ defined in Definition \ref{thetadef} is different from the class $e\in H^1(\X_W,\tr)$ used in \cite{geisser04}. The class $e$ lies in the image of $H^1(\X_W,\bz)$ and is the pullback of the identity map in
$$H^1(\Spec(\bof_p)_W,\bz)=\Hom_{\bz}(W_{\bof_p},\bz)\cong\Hom_{\bz}(\bz,\bz).$$
Since the natural map $W_{\bof_p}\to\br$ sends the Frobenius to $\log(p)$, the elemens $\theta$ and $e$ differ by a factor of $\log(p)$. This is consistent with the fact that
\begin{equation}\zeta^*(\X,0)=\log(p)^rZ^*(\X,1)\end{equation}
where $Z(\X,T)\in\bq(T)$ is the rational function so that $\zeta(\X,s)=Z(\X,p^{-s})$ and $Z(\X,T)=(1-T)^rZ^*(\X,1)$ with $r\in\bz$ and $Z^*(\X,1)\neq 0,\infty$.
\end{proof}

\subsection{Remarks}\label{rems} We finish this section with some remarks to put
our results in perspective.

\subsubsection{Cohomology with $\bz$-coefficients.} If
$\X\to\Spec(\bz)$ is a (proper, flat, regular) arithmetic scheme
with a section then $R\Gamma(\overline{\Spec(\bz)}_W,\bz)$ is a
direct summand of $R\Gamma(\overline{\X}_W,\bz)$. Hence by
\cite{flach06-1} $H^4(\X_W,\bz)$ will not be a finitely generated
abelian group and d) does not hold. Even if one could find a
definition of $\overline{\Spec(\bz)}_W$ with the expected
$\bz$-cohomology the definition of $\X_W$ as a fibre product
(Definition \ref{xwdef}) will not be the right one. Heuristically
this is because one should view the fibre product of topoi as a
"homotopy pullback", and the "homotopy fibre" of $\gamma: \X_W\to
\X_{et}$ is not independent of $\X$, unlike in the situation over
finite fields. Indeed, viewing $R\gamma_*\bz$ as the cohomology of
the fibre, Geisser has shown \cite{geisser04} that this complex has
cohomology $\bz$, $\bq$, $0$ in degrees $0$, $1$, $\geq 2$,
respectively, for any $\X$ over $\Spec(\bof_p)$. So for any $\X$
over $\Spec(\bof_p)$ one can view the fibre as the pro-homotopy type
of a solenoid.

For $\X=\overline{\Spec(\co_F)}$ where $F$ is a number field, one
expects $R\gamma_*\bz$ to be concentrated in degrees $0$ and $2$
(see \cite{morin09}[Sec.9]). On the other hand, if $$\X=\mathbb
P^1_{\Spec(\co_F)}$$ has the correct $\bz$-cohomology, compatible
with the computations of $\tilde{\br}$-cohomology in this paper,
then $H^4(\bar{\X}_W,\bz)$ must be a finitely generated group of
rank $r_2$, the rank of $K_3(\co_F)$ (see j) in section \ref{moreaxioms} below). This can only happen if $R^i\gamma_*\bz$ is nonzero
for $i=3$ or $i=4$, the most likely scenario being that
$R^4\gamma_*\bz$ is nonzero with global sections
$H^0(\bar{\X}_{et},R^4\gamma_*\bz)\cong \Hom_\bz(K_3(\co_F),\bq)$.
Again, this is only a heuristic argument since we have not
rigorously defined the homotopy fibre, let alone established any
relation between its $\bz$-cohomology and $R\gamma_*\bz$.

\subsubsection{Weil-groups of finitely generated fields} The definition of the Weil-\'etale topos as a fibre
product is closely related to the idea, briefly mentioned by
Lichtenbaum in the introduction of \cite{li04}, of defining the
Weil-\'etale topos via Weil-groups for all scheme points $x\in X$,
and then gluing into a global topos in the spirit of \cite{li04}.
This is because the Weil-group of a field $k(x)$ of finite transcendence degree over
its prime subfield $F$ would be defined as the fibre product
$G_{k(x)}\times_{G_F}W_F$ and the classifying topos of this group is
the fibre product of the classifying topoi of the factors by Corollary \ref{cor-fiberproduct-classtopoi-topgrps}.
The remarks of the previous section would then apply to such a definition as well.

\subsubsection{Properties a)-f) for $\overline{\X}$} If $\X$ is regular, proper and flat over $\Spec(\bz)$ with generic fibre $X$ of dimension $d$ it follows easily from our results that properties a)-c) hold for $\overline{\X}$ where of course
\[ R\Gamma_c(\overline{\X}_W,\tr)=R\Gamma(\overline{\X}_W,\tr)\]
and
\[\zeta(\overline{\X},s)=\zeta(\X,s)\prod_{i=0}^{2d}L_\infty(h^i(X),s)^{(-1)^i}.\]
Property d) must also hold for any reasonable definition of $R\Gamma(\overline{\X}_W,\bz)$ as will become clear from our discussion in section \ref{moreaxioms} below. This discussion will also show, however, that properties e) and f) will definitely not hold for any definition of
$R\Gamma(\overline{\X}_W,\bz)$. This is consistent with the fact that there are no special value conjectures for the completed L-functions $\Lambda(h^i(X),s)$ in the literature.

\subsubsection{Non-regular/non-proper schemes.} For varieties over finite fields
which are not smooth and proper the work of Geisser \cite{geisser05}
shows that one has to replace the \'etale topology by the
eh-topology (which allows abstract blow-ups as coverings) in order
to define groups $H^i_c(X_{Wh},\bz)$ and $H^i_c(X_{Wh},\br)$ which
are independent of a choice of compactification of $X$ and which
satisfy a)-f) in the introduction (where the index $W$ is replaced
by $Wh$). For arithmetic schemes over $\Spec(\bz)$ a similar
modification will be necessary, and one also has to assume some
strong form of resolution of singularities for arithmetic schemes.
We have refrained from trying to incorporate the idea of the
eh-topology in this paper since our results (based on the fibre
product definition of $\X_W$) are only very partial in any case.

\subsection{Relation to the Tamagawa number conjecture}\label{tama} In this section we establish the compatibility of the conjectural properties of Weil-\'etale cohomology, as outlined in the introduction and augmented with some further assumptions below, with the Tamagawa number conjecture of Bloch and Kato.

\subsubsection{Statement of the Tamagawa number conjecture} Let $\X$ be a proper, flat, regular $\bz$-scheme with
generic fibre $X$ of dimension $d$. The original Tamagawa number
conjecture of Bloch and Kato \cite{bk88} concerned the leading
Taylor coefficient of $L(h^i(X),s)$ at integers
$s\geq\frac{i+1}{2}$ . This was then generalized by Fontaine and
Perrin-Riou \cite{fpr91} to a conjecture about the vanishing order
and leading coefficient at any integer $s$. In this paper we are
only concerned with $s=0$.

One defines "integral motivic cohomology" groups
$H^{p}_M(X_{/\bz},\bq(q))$ for example, as
\[H^{p}_M(X_{/\bz},\bq(q)):=\text{im}\bigl(K_{2q-p}(\X)_\bq^{(q)}\to
K_{2q-p}(X)_\bq^{(q)}\bigr),\] with $K_j(X)_\bq^{(q)}$ the $q$-th
Adams eigenspace of the algebraic K-groups $K_j(X)\otimes_\bz\bq$.
Denote by $W^*=\Hom_\bq(W,\bq)$ the dual $\bq$-space and set $W_\br:=W\otimes_\bq\br$.

\begin{conjecture} (Vanishing order) The space $H^{2d-i+1}_M(X_{/\bz},\bq(d+1))$ is finite dimensional and
\begin{align*} \ord_{s=0}L(h^i(X),s)=&\dim_\bq
H^1_f(h^i(X)^*(1))^*-\dim_\bq H^0_f(h^i(X)^*(1))^*\\ =&\dim_\bq
H^1_f(h^{2d-i}(X)(d+1))^*-\dim_\bq H^0_f(h^{2d-i}(X)(d+1))^*
\\=&\dim_\bq H^{2d-i+1}_M(X_{/\bz},\bq(d+1))^*
\end{align*}
\label{van}\end{conjecture}

Let $H^p_{\mathcal D}(X_{/\br},\br(q))$ denote (real) Deligne
cohomology and let
\begin{align}\rho^i_\infty:H^{2d-i+1}_M(X_{/\bz},\bq(d+1))_\br\to &H^{2d-i+1}_{\mathcal
D}(X_{/\br},\br(d+1))\label{beili}\end{align} be the Beilinson
regulator.

\begin{conjecture} (Beilinson) The map $\rho^i_\infty$ is an isomorphism for $i\geq 1$ and there is an
exact sequence \begin{equation}0\to
H^{2d+1}_M(X_{/\bz},\bq(d+1))_\br\xrightarrow{\rho^0_\infty} H^{2d+1}_{\mathcal
D}(X_{/\br},\br(d+1))\to CH^0(X)^*_\br\to 0\label{beil0}\end{equation}
for $i=0$. \label{beil}\end{conjecture}

We remark that Deligne cohomology satisfies a duality
\begin{equation}H^{2d-i+1}_{\mathcal
D}(X_{/\br},\br(d+1))^*\cong H^i_{\mathcal
D}(X_{/\br},\br)=H^i(X(\bc),\br)^+\label{dual}\end{equation} for $i\geq 0$ and deduce the well known fact that the
vanishing order of $L(h^i(X),s)$ predicted by Conjectures \ref{van} and \ref{beil}
is in accordance with Lemma \ref{ordat0}. Another consequence of conjecture \ref{beil} is
\begin{equation} H^{2d-i+1}_M(X_{/\bz},\bq(d+1))=0 \label{beil-soule}\end{equation}
for $i\geq 2d+1$, a particular case of the Beilinson-Soule conjecture.

Define the fundamental line
\[\Delta_f(h^i(X))=\mydet_\bq^{-1}(H^i(X(\bc),\bq)^+)\otimes_\bq
\mydet_\bq H^{2d-i+1}_M(X_{/\bz},\bq(d+1))^*\] for $i>0$ and
\[\Delta_f(h^0(X))=\mydet_\bq CH^0(X)_\bq\otimes_\bq\mydet_\bq^{-1}(H^0(X(\bc),\bq)^+)\otimes_\bq
\mydet_\bq H^{2d+1}_M(X_{/\bz},\bq(d+1))^*\] for $i=0$. There is an
isomorphism
\[ \vartheta^i_\infty:\br\cong \Delta_f(h^i(X))_\br\]
induced by (\ref{dual}) and the dual of (\ref{beili}) (resp. (\ref{beil0}))  for
$i>0$ (resp. $i=0$).

Now fix a prime number $l$ and let $U\subseteq\Spec(\bz)$ an open
subscheme on which $l$ is invertible. For any smooth $l$-adic sheaf
$V$ on $U$ and prime $p\neq l$ define a complex concentrated in
degrees $0$ and $1$
\[R\Gamma_f(\bq_p,V)=R\Gamma(\bof_p,i_p^*j_{p,*}V)=V^{I_p}\xrightarrow{1-\Frob_p^{-1}}V^{I_p}\]
where $I_p$ is the inertia subgroup at $p$ and
$i_p:\Spec(\bof_p)\to\Spec(\bz)$ and $j_p:U\to\Spec(\bz)$ are the
natural immersions. For $p=l$ define
\[R\Gamma_f(\bq_p,V)=D_{cris}(V)\xrightarrow{(1-\phi,\iota)}D_{cris}(V)\oplus D_{dR}(V)/F^0D_{dR}(V)\]
where $D_{cris}$ and $D_{dR}$ are Fontaine's functors \cite{fpr91}.
In both cases there is a map of complexes
\[ R\Gamma_f(\bq_p,V)\to R\Gamma(\bq_p,V)\]
and one defines $R\Gamma_{/f}(\bq_p,V)$ as the mapping cone. The next Lemma shows that the complex $R\Gamma_f(\bq_p,V)$ has a uniform description for $p=l$ and $p\neq l$ in the case that interests us.

\begin{lemma} Let $V$ be finite dimensional $\bq_p$-vector space with a continuous $G_p:=\Gal(\bar{\bq}_p/\bq_p)$-action and such that $D_{dR}(V)/F^0D_{dR}(V)=0$. Then there is a commutative diagram in the derived category of $\bq_p$-vector spaces
\[\begin{CD} R\Gamma_f(\bq_p,V) @>>> R\Gamma(\bq_p,V)\\
@V\kappa VV \Vert@.\\
R\Gamma(\bof_p,V^{I_p}) @>>> R\Gamma(\bq_p,V)
\end{CD}\]
where $\kappa$ is a quasi-isomorphism.
\label{fil0}\end{lemma}

\begin{proof} For a profinite group $G$ and continuous $G$-module $M$ we denote by $C^*(G,M)$ the standard complex of continuous cochains. There is an exact sequence of continuous $G_p$-modules
\[ 0\to V\to B^0(V)\xrightarrow{d^0}B^1(V) \to 0\]
where $B^0(V)=B_{cris}\otimes_{\bq_p}V$ (with diagonal $G_p$-action), $B^1(V)=B_{cris}\otimes_{\bq_p}V\oplus (B_{dR}/F^0B_{dR})\otimes_{\bq_p}V$ and $d^0(x)=((1-\phi)(x),\iota(x))$ where $\iota$ is induced by the canonical inclusion $B_{cris}\to B_{dR}$ (see \cite{fpr91} for more on Fontaine's rings $B_{cris}$ and $B_{dR}$). Viewing this sequence as a quasi-isomorphism between $V$ and a two term complex we obtain a quasi-isomorphism
\[ R\Gamma(\bq_p,V)=R\Gamma(G_p,V)=C^*(G_p,V)\cong \Tot\left(C^*(G_p,B^0(V))\xrightarrow{d^{0*}}C^*(G_p,B^1(V))\right) \]
where $\Tot$ denotes the simple complex associated to a double complex. By definition $R\Gamma_f(\bq_p,V)$ is the subcomplex
\[ D_{cris}(V)=H^0(G_p,B^0(V))\xrightarrow{d^0}H^0(G_p,B^1(V))=D_{cris}(V)\oplus D_{dR}(V)/F^0D_{dR}(V) \]
of this double complex. For any continuous  $G_p$-module $M$ there is moreover a quasi-isomorphism
\[R\Gamma(G_p,M)\cong R\Gamma(\bof_p,R\Gamma(I_p,M))\cong \Tot\left(C^*(I_p,M)\xrightarrow{1-\Frob_p^{-1}}C^*(I_p,M)\right)\]
where $\Frob_p\in G_p$ is any lift of the Frobenius automorphism in $G_p/I_p$, acting  simultaneously on $I_p$ (by conjugation) and on $M$. The complex $R\Gamma(\bof_p,H^0(I_p,M))$ is the subcomplex
\[ H^0(I_p,M)\xrightarrow{1-\Frob_p^{-1}}H^0(I_p,M)\] of this double complex. Combining these two constructions, we deduce that $R\Gamma(\bq_p,V)$ is canonically isomorphic to the total complex of the triple complex
\[\begin{CD} C^*(I_p,B^0(V))@>d^{0*}>>C^*(I_p,B^1(V))\\
@V 1-\Frob_p^{-1} VV @V 1-\Frob_p^{-1} VV\\
C^*(I_p,B^0(V))@>d^{0*}>>C^*(I_p,B^1(V))
\end{CD}\]
and $R\Gamma(\bof_p,H^0(I_p,V))$ is canonically isomorphic to the total complex of the double subcomplex
\[\begin{CD} H^0(I_p,B^0(V))@>d^{0}>>H^0(I_p,B^1(V))\\
@V 1-\Frob_p^{-1} VV @V 1-\Frob_p^{-1} VV\\
H^0(I_p,B^0(V))@>d^{0}>>H^0(I_p,B^1(V)).
\end{CD}\]
Now if $D_{dR}(V)/F^0D_{dR}(V)=0$ this double complex is naturally quasi-isomorphic to $R\Gamma_f(\bq_p,V)$ via the first vertical map $\kappa$ in the following diagram
\[\begin{CD} D_{cris}(V)@> 1-\phi >> D_{cris}(V)\\
@V\kappa^0 VV @V\kappa^1 VV\\
H^0(I_p,B^0(V))@>1-\phi >> H^0(I_p,B^0(V))\\
@V 1-\Frob_p^{-1} VV @V 1-\Frob_p^{-1} VV\\
H^0(I_p,B^0(V))@>1-\phi >> H^0(I_p,B^0(V)).
\end{CD}\]
Indeed, the vertical sequences in this diagram are short exact sequences. The space $D_{cris}(V)=H^0(G_p,B^0(V))$ is clearly the kernel of $1-\Frob_p^{-1}$ on $H^0(I_p,B^0(V))$, and $1-\Frob_p^{-1}$ is surjective. This is because there is an isomorphism of $\Frob_p$-modules $H^0(I_p,B^0(V))=D_{cris}(V)\otimes_{\bq_p}\bqpur\cong(\bqpur)^d$ where $d=\dim_{\bq_p}D_{cris}(V)$ and $\bqpur$ is the $p$-adic completion of the maximal unramified extension of $\bq_p$. It is well known that $1-\Frob_p$ is surjective $\bqpur$. This concludes the proof of the Lemma.
\end{proof}

Next one defines a global complex $R\Gamma_f(\bq,V)$ as the mapping fibre of
\[R\Gamma(U_{\et},V)\to\bigoplus_{p\notin U}R\Gamma_{/f}(\bq_p,V).\]
Then there is an exact triangle in the derived category of
$\bq_l$-vector spaces \begin{equation} R\Gamma_c(U_{\et},V)\to
R\Gamma_f(\bq,V)\to\bigoplus_{p\notin U}R\Gamma_f(\bq_p,V)\label{tri1}\end{equation} where the primes $p\notin U$ include $p=\infty$ with the convention $R\Gamma_f(\br,V)=R\Gamma(\br,V)$. One can further show that Artin-Verdier duality
induces a duality
\[ H^i_f(\bq,V)\cong H^{3-i}_f(\bq,V^*(1))^*.  \]

The index "$f$" stands for "finite" which in this context is synonymous for "unramified" or "coming from an integral model". The following proposition justifies this interpretation of the complex $R\Gamma_f$ in the case of interest in this paper.

\begin{prop} Let $\pi:\X\to\Spec(\bz)$ be a regular, proper, flat $\bz$-scheme and
$\bar{\X}_\et$ its Artin-Verdier \'etale topos. Let
$U\subseteq\Spec(\bz)$ be an open subscheme so that $\pi_U:\X_U\to U$
is proper and smooth, let $l$ be a prime number invertible on
$U$ and set $\X_p=\X\otimes_\bz\bof_p$. For brevity we write $\X_{\infty,\et}$ for $Sh(\X_\infty)$ (see Prop. \ref{closed/open-decomp-etale}). Assume Conjecture \ref{w0} in the next section. Then there is an isomorphism of exact triangles in the derived
category of $\bq_l$-vector spaces
\[\minCDarrowwidth1em\begin{CD}
R\Gamma_c(\X_{U,\et},\bq_l) @>>> R\Gamma(\bar{\X}_{\et},\bq_l) @>>>
\bigoplus\limits_{p\notin U}R\Gamma(\X_{p,\et},\bq_l) @>>> {}\\
@VVV @VVV @VVV @.\\
\bigoplus\limits_{i=0}^{2d}R\Gamma_c(U_\et,V^i_l)[-i] @>>>
\bigoplus\limits_{i=0}^{2d}R\Gamma_f(\bq,V^i_l)[-i]
@>>>\bigoplus\limits_{p\notin U}\bigoplus\limits_{i=0}^{2d}R\Gamma_f(\bq_p,V^i_l)[-i]@>>>
\end{CD}\]
where $V^i_l:=H^i(X_{\bar{\bq},\et},\bq_l)$ and the bottom exact
triangle is a sum over triangles (\ref{tri1}).
\label{reform}\end{prop}

\begin{proof} For all $p$ and $l$ (including $p=\infty$ with a suitable interpretation of the terms) we shall first show that there is a commutative diagram
\begin{equation}\minCDarrowwidth1em\begin{CD}
R\Gamma(\X_{p,\et},\bq_l) @>>> R\Gamma(\X_{\bq_p,\et},\bq_l) @>>> R\Gamma_{\X_p}(\X_{\bz_p,\et},\bq_l)[1] @>>> {}\\
@V\alpha VV @V\beta VV @VVV @. \\
\bigoplus\limits_{i=0}^{2d}R\Gamma_f(\bq_p,V^i_l)[-i] @>>>
\bigoplus\limits_{i=0}^{2d}R\Gamma(\bq_p,V^i_l)[-i]@>>>\bigoplus\limits_{i=0}^{2d}R\Gamma_{/f}(\bq_p,V^i_l)[-i]@>>>
\end{CD}\label{dia-pl}\end{equation}
where the rows are exact and the vertical maps are quasi-isomorphism.  This then induces a commutative diagram where the vertical maps are quasi-isomorphism
\begin{equation}\minCDarrowwidth1em\begin{CD}
R\Gamma(\X_{U,\et},\bq_l)@>>>\bigoplus\limits_{p\notin U} R\Gamma(\X_{\bq_p,\et},\bq_l) @>>> \bigoplus\limits_{p\notin U}R\Gamma_{\X_p}(\X_{\bz_p,\et},\bq_l)[1]\\
@VVV @VVV @VVV \\\bigoplus\limits_{i=0}^{2d}R\Gamma(U_\et,V^i_l)[-i]
@>>>
\bigoplus\limits_{p\notin U}\bigoplus\limits_{i=0}^{2d}R\Gamma(\bq_p,V^i_l)[-i]@>>>\bigoplus\limits_{p\notin U}\bigoplus\limits_{i=0}^{2d}R\Gamma_{/f}(\bq_p,V^i_l)[-i].
\end{CD}\label{dia2}\end{equation}
Indeed, the first commutative square is induced by the commutative diagram
\[\begin{CD}
\X_U @<<< \X_{\bq_p}\\
@VVV @VVV\\
U @<<< \Spec(\bq_p)
\end{CD}\]
and a decomposition $R\pi_{U,*}\bq_l\cong\bigoplus_{i=0}^{2d}V_l^i[-i]$ in the derived category of $l$-adic sheaves on $U$, and the second is a sum over $p\notin U$ of the right hand square in (\ref{dia-pl}).
Taking mapping fibres of the composite horizontal maps in (\ref{dia2}) we obtain an isomorphism of exact triangles
\[\minCDarrowwidth1em\begin{CD}
R\Gamma(\bar{\X}_{\et},\bq_l) @>>> R\Gamma(\X_{U,\et},\bq_l) @>>>
\bigoplus\limits_{p\notin U}R\Gamma_{\X_p}(\X_{\bz_p,\et},\bq_l)[1] @>>> {}\\
@VVV @VVV @VVV @.\\
\bigoplus\limits_{i=0}^{2d}R\Gamma_f(\bq,V^i_l)[-i] @>>>
\bigoplus\limits_{i=0}^{2d}R\Gamma(U_\et,V^i_l)[-i]
@>>>\bigoplus\limits_{p\notin U}\bigoplus\limits_{i=0}^{2d}R\Gamma_{/f}(\bq_p,V^i_l)[-i]@>>>
\end{CD}\]
where we use excision to identify the first fibre with $R\Gamma(\bar{\X}_{\et},\bq_l)$. The octahedral axiom then gives the isomorphism of exact triangles in Proposition \ref{reform}, using the fact that the mapping fibre of the top left (resp. bottom left) horizontal map in (\ref{dia2}) is
$R\Gamma_c(\X_{U,\et},\bq_l)$ (resp. $\bigoplus_{i=0}^{2d}R\Gamma_c(U_\et,V^i_l)[-i]$).

Concerning (\ref{dia-pl}), for $p=\infty$ we declare  $R\Gamma(\X_{p,\et},\bq_l) = R\Gamma(\X_{\bq_p,\et},\bq_l)$ and $R\Gamma_{\X_p}(\X_{\bz_p,\et},\bq_l)=0$. This agrees with the convention $R\Gamma_f(\br,-)=R\Gamma(\br,-)$ introduced above.
For $p\neq\infty$ the top exact triangle is simply a localization triangle in \'etale cohomology since we have
$R\Gamma(\X_{p,\et},\bq_l)\cong R\Gamma(\X_{\bz_p,\et},\bq_l)$ by proper base change. It suffices to construct quasi-isomorphisms $\alpha $ and $\beta$ so that the left hand square in (\ref{dia-pl}) commutes. For brevity we now omit the index $\et$ when referring to (continuous $l$-adic) \'etale cohomology.

The quasi-isomorphism $\beta$ is induced by the Leray spectral sequence for $\pi_{\bq_p}$ and a decomposition
\begin{equation} R\pi_{\bq_p,*}\bq_l\cong \bigoplus\limits_{i=0}^{2d}V_l^i[-i]\label{decomp}\end{equation}
in the derived category of $l$-adic sheaves on $\Spec(\bq_p)$. The existence of $\alpha$ follows if the composite map
\[ H^i(\X_p,\bq_l)\to H^i(\X_{\bq_p},\bq_l)\xrightarrow{H^i(\beta)}H^0(\bq_p,V^i_l)\oplus H^1(\bq_p,V^{i-1}_l)\oplus H^2(\bq_p,V^{i-2}_l)\]
induces an isomorphism
\[ H^i(\X_p,\bq_l)\cong H^0_f(\bq_p,V^i_l)\oplus H^1_f(\bq_p,V^{i-1}_l).\]
We shall show this only referring to the filtration $F^*$ on $H^i(\X_{\bq_p},\bq_l)$ induced by the Leray spectral sequence for $\pi_{\bq_p}$, not any particular decomposition (\ref{decomp}).
The Hochschild-Serre spectral sequence for the covering $\X_{\hat{\bz}_p^{ur}}\to\X_{\bz_p}$, whose group we identify with $\Gal(\bar{\bof}_p/\bof_p)$, induces a commutative diagram with exact rows
\begin{equation}\minCDarrowwidth1em\begin{CD} 0 @>>> H^1(\bof_p, H^{i-1}(\X_{\bar{\bof}_p},\bq_l)) @>>>  H^i(\X_p,\bq_l) @>>> H^0(\bof_p,H^i(\X_{\bar{\bof}_p},\bq_l)) @>>> 0\\
@. @VVV @VVV @VVV @.\\
0 @>>> H^1(\bof_p, H^{i-1}(\X_{\bqpur},\bq_l)) @>>>  H^i(\X_{\bq_p},\bq_l) @>>> H^0(\bof_p,H^i(\X_{\bqpur},\bq_l)) @>>> 0\\
@. @VVV @VVV @VV\gamma V @.\\
{} @. H^1(\bof_p, H^0(I_p,V_l^{i-1})) @. H^0(\bq_p,V_l^i) @= H^0(\bof_p,H^0(I_p,V_l^i)).@.{}
\end{CD}\label{dia23}\end{equation}
The left and right composite vertical maps are isomorphisms by Theorem \ref{main} b) for $l\neq p$ (resp. Conjecture \ref{w0} for $l=p$) and the fact that $$R\Gamma(\bof_p,V)\cong R\Gamma(\bof_p,W_0V)$$ for any $l$-adic sheaf $V$ on $\Spec(\bof_p)$ where $W_0V\subseteq V$ is the generalized Frobenius eigenspace for eigenvalues which are roots of unity (or just for the eigenvalue $1$). Note also that
\[ H^k_f(\bq_p,V_l^i)=H^k(\bof_p,H^0(I_p,V_l^i))\]
for $k=0,1$ and all $l$ and $i$ by Lemma \ref{fil0} since
$$D_{dR}(V_p^i)\cong H^i_{dR}(\X_{\bq_p}/\bq_p)=F^0H^i_{dR}(\X_{\bq_p}/\bq_p)\cong F^0D_{dR}(V_p^i).$$
The kernel of the map $\gamma$ in (\ref{dia23}) is $H^0(\bof_p,H^1(I_p,V_l^{i-1}))$, hence there is a commutative diagram with exact rows
\[\minCDarrowwidth1em\begin{CD}
0 @>>> H^1(\bof_p, H^{i-1}(\X_{\bqpur},\bq_l)) @>>> F^1H^i(\X_{\bq_p},\bq_l) @>>> H^0(\bof_p,H^1(I_p,V_l^{i-1})) @>>> 0\\
@. @VVV @VVV \Vert@. @.\\
0 @>>> H^1(\bof_p, H^0(I_p,V_l^{i-1})) @>>> H^1(\bq_p,V_l^{i-1}) @>>> H^0(\bof_p,H^1(I_p,V_l^{i-1}))@>>> 0
\end{CD}\]
which implies that the left vertical isomorphism in (\ref{dia23}) fits into a commutative diagram with the natural map $F^1H^i(\X_{\bq_p},\bq_l)\to H^1(\bq_p,V_l^{i-1})$. This finishes the proof of the existence of $\alpha$ and of Proposition \ref{reform}.

We remark that for $l\neq p$ we have $W_0H^1(I_p,V_l^i)=W_0((V_l^i)_{I_p}(-1))=0$ and hence isomorphisms
\[ W_0H^i(\X_{\bar{\bof}_p},\bq_l)\cong W_0H^i(\X_{\bqpur},\bq_l)\cong W_0H^0(I_p,V_l^i)\]
which implies that the top left and right, and therefore the top middle vertical maps in (\ref{dia23}) are isomorphisms. We conclude that \[ R\Gamma(\X_p,\bq_l)\cong R\Gamma(\X_{\bq_p},\bq_l)\]
for $l\neq p$ like for $p=\infty$.
\end{proof}

We continue with the statement of the Tamagawa number conjecture.
One might view the following conjecture as an $l$-adic analogue of
Beilinson's conjecture, or as a generalization of Tate's conjecture.
\begin{conjecture} (Bloch-Kato) There are isomorphisms
\[ \rho_l^i:H^2_f(\bq,V_l^i)\cong H^{2d-i+1}_M(X_{/\bz},\bq(d+1))_{\bq_l}^*  \]
and $H^1_f(\bq,V_l^i)=0$ for any $i$. \label{bk}\end{conjecture}

One can show easily that $H^0_f(\bq,V_l^0)\cong Ch^0(X)_{\bq_l}$,
$H^0_f(\bq,V_l^i)=0$ for $i>0$ and $H^3_f(\bq,V_l^i)=0$ so that
Conjecture \ref{bk} computes the entire cohomology of
$R\Gamma_f(\bq,V_l^i)$. Together with Artin's comparison isomorphism
\[V^i_l=H^i(X_{\bar{\bq},\et},\bq_l)\cong H^i(X(\bc),\bq)_{\bq_l}\] as
well as the isomorphisms
\[ \iota_p:\mydet_{\bq_l}R\Gamma_f(\bq_p,V)\cong\bq_l  \]
induced by the identity map on $(V_l^i)^{I_p}$ and
$D_{cris}(V_l^i)$, Conjecture \ref{bk} induces an isomorphism
\[ \vartheta^i_l:\Delta_f(h^i(X))_{\bq_l}\cong \mydet_{\bq_l}R\Gamma_f(\bq,V_l^i)\otimes\mydet_{\bq_l}R\Gamma(\br,V_l^i)\cong\mydet_{\bq_l}R\Gamma_c(U_{\et},V^i_l).\]

\begin{conjecture} ($l$-part of the Tamagawa number conjecture) There is an identity of
free rank one $\bz_l$-submodules of
$\mydet_{\bq_l}R\Gamma_c(U_{\et},V^i_l)$
\[
\bz_l\cdot\vartheta^i_l\circ\vartheta^i_\infty(L^*(h^i(X),0)^{-1})=\mydet_{\bz_l}R\Gamma_c(U_{\et},T^i_l)\]
for any Galois stable $\bz_l$-lattice $T^i_l\subseteq V^i_l$.
\label{tam}\end{conjecture}

This conjecture is independent of the choice of the lattice $T_l^i$
since
\begin{equation}\prod_{i\in\bz}|H^i_c(U_{\et},M)|^{(-1)^i}=1\label{finite}\end{equation} for
any finite locally constant sheaf $M$ whose cardinality is
invertible on $U$. The following conjecture allows a reformulation
of the Tamagawa number conjecture in terms of the L-function
\[  L_U(h^i(X),s)=\prod_{p\in U} L_p(h^i(X),s)\]
associated to the smooth $l$-adic sheaf $V_l^i$ over $U$. Recall
that a two term complex $C=\bigl(W\xrightarrow{\lambda}W\bigr)$ is
called {\em semisimple at $0$} if the composite map
\[ H^0(C)=\mathrm{ker}(\lambda)\subseteq
W\to\mathrm{coker}(\lambda)=H^1(C)\] is an isomorphism. This is
always the case, for example, if the complex $C$ is acyclic.

\begin{conjecture} (Frobenius-Semisimplicity at the eigenvalue 1)
For any prime number $p$ the complex $R\Gamma_f(\bq_p,V^i_l)$ is
semisimple at zero. \label{ss}\end{conjecture}

Under this conjecture one has a {\em second} isomorphism
\[ \tilde{\iota}_p:\mydet_{\bq_l}R\Gamma_f(\bq_p,V^i_l)\cong\bq_l  \]
which satisfies
\[\iota_p=P^*_p(h^i(X),1)^{-1}\tilde{\iota}_p=L_p^*(h^i(X),0)\log(p)^{r_{i,p}}\tilde{\iota}_p\]
where $r_{i,p}=\ord_{T=1}P_p(h^i(X),T)=-\ord_{s=0}L_p(h^i(X),s)$ (see \cite{bufl98}[Lemma 2]). If $R\Gamma_f(\bq_p,V^i_l)$ is acyclic then
$\tilde{\iota}_p$ is the canonical trivialization of the determinant
of an acyclic complex. Using this second isomorphism the Tamagawa
number conjecture becomes
\begin{equation}\bz_l\cdot\tilde{\vartheta}^i_l\circ\tilde{\vartheta}^i_\infty(L^*_U(h^i(X),0)^{-1})=\mydet_{\bz_l}R\Gamma_c(U_{\et},T^i_l)\label{tam1}\end{equation}
where
\[\tilde{\vartheta}^i_l=\prod_{p\notin U} P^*_p(h^i(X),1)\vartheta_l^i,\quad \tilde{\vartheta}^i_\infty=\prod_{p\notin U}\log(p)^{r_{i,p}}\vartheta_\infty^i.\]

\subsubsection{Further assumptions on Weil-\'etale cohomology}\label{moreaxioms} In order to establish the compatibility of the
conjectural picture a)-f) outlined in the introduction with the
Tamagawa number conjecture, we need to augment it with a number of
further assumptions. Even though a)-f) only refer to cohomology groups we now assume that these groups do indeed arise from a topos $\X_W$ - different from the one defined in Definition \ref{xwdef} - and that compact support cohomology is defined via an embedding into a proper scheme followed by an Artin-Verdier type compactification $\bar{\X}_W$ (and is independent of a choice of compactification).

\begin{itemize}
\item[g)] For an open subscheme $U$ of an arithmetic scheme $\X$
with closed complement $Z$ there is an exact triangle in the derived
category of abelian groups
\[ R\Gamma_c(U_W,\bz)\to R\Gamma_c(\X_W,\bz)\to R\Gamma_c(Z_W,\bz)\to.\]
\item[h)] There is a morphism of topoi $\gamma:\X_W\to\X_\et$ for any
arithmetic scheme $\X$ (or the Artin-Verdier compactification of
such a scheme). Moreover, for any constructible sheaf $\F$ on
$\X_\et$ the adjunction $\F\to R\gamma_*\gamma^*\F$ is an
isomorphism.
\end{itemize}

If $\X$ has finite characteristic then g) and h) hold if one
understands the index $W$ as denoting the Weil-eh cohomology of
Geisser (see \cite{geisser05}[Thm. 5.2b), Thm. 3.6] for h) and
\cite{geisser05}[Def. 5.4, eq. (4)] for g)). The following property is a natural extension of property g) to the Artin-Verdier compactification.
\begin{itemize}
\item[i)] If $\X$ is regular, proper, flat over $\Spec(\bz)$ then
there is an exact triangle in the derived category of abelian groups
\[ R\Gamma_c(\X_W,\bz)\to R\Gamma(\bar{\X}_W,\bz)\to R\Gamma(\X_{\infty,W},\bz)\to\]
and there is an exact triangle
\[ R\Gamma_c(\X_W,\tr)\to R\Gamma(\bar{\X}_W,\tr)\to R\Gamma(\X_{\infty,W},\tr)\to\]
in the derived category of $\br$-vector spaces, where $\X_{\infty,W}$ was defined in Definition \ref{xinftydef}.
\end{itemize}
Note that
\[R\Gamma(\X_{\infty,W},\bz)\cong R\Gamma(\X_\infty,\gamma_{\infty*}(\bz))\cong R\Gamma(\X_\infty,\bz)\]
by Proposition \ref{prop-basechange-infty} and this last complex is isomorphic to the singular complex of the (locally contractible) compact space $\X_\infty$ and is therefore a perfect complex of abelian groups. Since the complex $R\Gamma_c(\X_W,\bz)$ is perfect by d) the triangle in i) then implies that $R\Gamma(\bar{\X}_W,\bz)$ is also a perfect complex of abelian groups. Note also that, unlike in the situation g), the triangle for $\tr$-coefficients is not the scalar extension of the triangle for $\bz$-coefficients since neither $R\Gamma(\X_{\infty,W},\bz)$ nor $R\Gamma(\bar{\X}_W,\bz)$ satisfies property e). One rather has a commutative diagram of long exact sequences
\[\minCDarrowwidth1em\begin{CD}
  @>>> H^{i+1}(\X_{\infty,W},\tr) @>>> H^{i+2}_c(\X_W,\tr) @>>> 0 @>>>
  H^{i+2}(\X_{\infty,W},\tr)\\
  @. @AA\alpha_\infty A @AA\alpha A @AAA @AAA\\
  @>>> H^{i+1}(\X_{\infty,W},\bz) @>>> H^{i+2}_c(\X_W,\bz) @>>> H^{i+2}(\bar{\X}_W,\bz) @>>>
  H^{i+2}(\X_{\infty,W},\bz)
\end{CD}\]
where only $\alpha_\br$ is an isomorphism by e). Here we assume $i\geq 0$ so that $H^{i+2}(\bar{\X}_W,\tr)=0$ by Theorem \ref{thm-global-cohomology}. There is a direct sum decomposition
\[H^{i+1}(\X_{\infty,W},\tr)\cong
H^{i+1}(\X_{\infty},\br)\oplus H^i(\X_\infty,\br)\] by (\ref{xinfty-split}) and an isomorphism
\[ H^{i+1}(\X_{\infty,W},\bz)_\br\cong H^{i+1}(\X_\infty,\bz)_\br \cong H^{i+1}(\X_{\infty},\br).\]
One therefore obtains a map for $i\geq 0$ \[ r^i_\infty: H^i(\X_\infty,\br)\to
H^{i+2}(\bar{\X}_W,\bz)_\br\] which is an isomorphism for $i>0$.

Proposition \ref{reform} and assumption h) yield an isomorphism for $i\geq 0$
\[ r_l^i:H^2_f(\bq,V_l^i)\cong H^{i+2}(\bar{\X}_\et,\bq_l)\cong
H^{i+2}(\bar{\X}_W,\bq_l)\cong
H^{i+2}(\bar{\X}_W,\bz)_{\bq_l}.\]The following is
the key requirement on a definition of a Weil-\'etale topos.
\begin{itemize}
\item[j)] If $\X$ is regular, proper, flat over $\Spec(\bz)$ with generic fibre $X$ of dimension $d$ then there are isomorphisms
\[ \lambda^i:H^{i+2}(\overline{\X}_W,\bz)_\bq\cong H^{2d-i+1}_M(X_{/\bz},\bq(d+1))^*\]
for $i\geq 0$ such that $\lambda_\br^i\circ
r_\infty^i=(\rho^i_\infty)^*$ and $\lambda_{\bq_l}^i\circ
r_l^i=\rho_l^i$.
\end{itemize}

This is true for $d=i=0$ with Lichtenbaum's current definition where
\[ H^{2}(\overline{\Spec(\co_F)}_W,\bz)_\bq\cong H^{1}_M(\Spec(F)_{/\bz},\bq(1))^*=\Hom_\bz(\co_F^\times,\bq).\]
Note that j) together with (\ref{beil-soule}) and h) also implies
\[ H^{i+2}(\overline{\X}_W,\bz)=0 \]
for $i>2d+1$, which is not satisfied by the current definition of $\overline{\Spec(\co_F)}_W$.

\begin{prop} Suppose there is a definition of Weil-\'etale cohomology
groups for arithmetic schemes satisfying a)-j) except perhaps f) for schemes of characteristic $0$. Let $X$ be a proper, smooth variety over $\bq$ of
dimension $d$ which has a proper, regular model over $\Spec(\bz)$
such that Conjectures \ref{van},\ref{beil},\ref{bk},\ref{ss},\ref{w0} are
satisfied. Assume $L(h^i(X),s)$ has a meromorphic continuation to $s=0$ for all $i$. Then the Tamagawa number conjecture (Conjecture
\ref{tam}) for the motive
\[h(X)=\bigoplus_{i=0}^{2d}h^i(X)[-i]\]
is equivalent to statement f) for any arithmetic scheme $\X$ with
generic fibre $X$.
\label{tamcompare}\end{prop}

\begin{proof} If $\X$ is any arithmetic scheme with generic fibre $X$ then there
exists an open subscheme $U\subseteq\Spec(\bz)$ so that $\pi:\X_U\to
U$ is proper and smooth. Let $Z$ be the closed complement of $U$.
Then by g) we have an isomorphism
\[ \mydet_\bz R\Gamma_c(\X_W,\bz)\cong \mydet_\bz R\Gamma_c(\X_{U,W},\bz)\otimes_\bz \mydet_\bz
R\Gamma_c(X_{Z,W},\bz)  \] as well as factorizations
\[\zeta(\X,s)=\zeta(\X_U,s)\zeta(\X_Z,s);\quad \zeta^*(\X,0)=\zeta^*(\X_U,0)\zeta^*(\X_Z,0).\]
Since we assume f) for $X_Z=\coprod_{p\in Z}\X_p$ , statement f) for
$\X$ is equivalent to statement f) for $\X_U$. We now assume that
$\X$ is the proper regular model of $X$. The exact triangle
\[ R\Gamma_c(\X_{U,W},\bz)\to R\Gamma(\bar{\X}_W,\bz)\to \bigoplus_{p\in Z\cup\{\infty\}}R\Gamma(\X_{p,W},\bz) \]
together with assumption j) induces an isomorphism
\begin{align*}
\vartheta_W:\mydet_\bq R\Gamma_c(\X_{U,W},\bz)_\bq\cong & \mydet_\bq
R\Gamma(\bar{\X}_W,\bz)_\bq\otimes \bigotimes_{p\in
Z\cup\{\infty\}}\mydet_\bq^{-1} R\Gamma(\X_{p,W},\bz)_\bq\\
\cong & \bigotimes_{i=0}^{2d}\Delta_f(h^i(X))^{(-1)^i}.
\end{align*}
By assumption j) there is a commutative diagram of isomorphisms
\[\begin{CD} \br @>\gamma >> \mydet_\br
R\Gamma_c(\X_{U,W},\bz)_\br\\
\Vert@. @VV \vartheta_{W,\br} V\\
\br @>\otimes_i(\tilde{\vartheta}_\infty^i)^{(-1)^i} >>
\bigotimes_{i=0}^{2d}\Delta_f(h^i(X))_\br^{(-1)^i}
\end{CD}\]
where $\gamma$ is induced by c). The power of $\log(p)$ in $\tilde{\vartheta}$ appears for the same reason as in the proof of Theorem \ref{charp-theo}. Similarly, j) implies that for any prime $l\in Z$ we have a
commutative diagram of isomorphisms \begin{equation}\begin{CD}
\mydet_{\bq_l} R\Gamma_c(\X_{U,W},\bz)_{\bq_l} @>>>
\mydet_{\bq_l} R\Gamma_c(\X_{U,\et},\bq_l) \\
@VV \vartheta_{W,\bq_l} V @VVV\\
\bigotimes_{i=0}^{2d}\Delta_f(h^i(X))_{\bq_l}^{(-1)^i}@>\otimes_i(\tilde{\vartheta}_l^i)^{(-1)^i}>>
\bigotimes_{i=0}^{2d}\mydet_{\bq_l}^{(-1)^i}R\Gamma_c(U_\et,V_l^i)
\end{CD}\label{dia3}\end{equation}
where the top isomorphism is induced by an isomorphism
\[R\Gamma_c(\X_{U,W},\bz)\otimes_\bz\bz_l\cong R\Gamma_c(\X_{U,\et},\bz_l) \]
coming from assumption h) and the right vertical isomorphism is
induced by the isomorphism
\[R\Gamma_c(\X_{U,\et},\bz_l)\cong R\Gamma_c(U_\et,R\pi_*\bz_l)\]
and
\[\mydet_{\bz_l}R\Gamma_c(U_\et,R\pi_*\bz_l)\cong\bigotimes_{i=0}^{2d}\mydet_{\bz_l}^{(-1)^i}
R\Gamma_c(U_\et,L_l^i)=\bigotimes_{i=0}^{2d}\mydet_{\bz_l}^{(-1)^i}
R\Gamma_c(U_\et,T_l^i)\] where $L_l^i:=R^i\pi_*\bz_l$ and
$T_l^i\subseteq V_l^i$ is the torsion free part of $L_l^i$. Note
that we have an exact sequence of locally constant $\bz_l$-sheaves
on $U$
\[ 0\to L^i_{l,tor}\to L_l^i \to T_l^i\to 0\]
and an identity $\mydet_{\bz_l}R\Gamma_c(U_\et,T_l^i)=
\mydet_{\bz_l}R\Gamma_c(U_\et,L_l^i)$ of invertible
$\bz_l$-submodules of $\mydet_{\bq_l}R\Gamma_c(U_\et,V_l^i)$ by
(\ref{finite}).

As discussed above statement f) for $\X_U$ is equivalent to
statement f) for $\X_{U'}$ for $U'\subset U$, hence we can always
assume that a given prime $l$ is not in $U$. If we know statement f)
for $\X_U$ then the image under $\gamma$ of
$$\zeta^*(\X_U,0)=\prod_{i=0}^{2d}L^*_U(h^i(X),0)^{(-1)^i}$$ generates
the natural invertible $\bz_l$-submodule
\[R\Gamma_c(\X_{U,W},\bz)\otimes_\bz\bz_l\cong \bigotimes_{i=0}^{2d}\mydet_{\bz_l}^{(-1)^i}R\Gamma_c(U_\et,T_l^i)\]
in (\ref{dia3}) (see the discussion in the previous paragraph). Hence we obtain the Tamagawa number
conjecture in the form (\ref{tam1}) for $h(X)$. Conversely, knowing
the Tamagawa number conjecture for $h(X)$, we obtain the $l$-primary
part of statement f) for $\X_{U\setminus\{l\}}$ which is equivalent
to the $l$-primary part of statement f) for $\X_U$. Varying $l$ we
obtain f) for $\X_U$. Here by $l$-primary parts, we mean that for
any perfect complex of abelian groups $C$, such as
$R\Gamma_c(\X_{U,W},\bz)$, an element $b\in\mydet_\bz(C)\otimes\bq$
is a generator of $\mydet_\bz(C)$ if and only if the image of $b$ in
$\mydet_\bz(C)\otimes\bq_l$ is a generator of
$\mydet_\bz(C)\otimes\bz_l$ for all primes $l$.

\end{proof}

\section{On the local theorem of invariant cycles}\label{loc-inv-cycles} Let
$R$ be a complete discrete valuation ring with quotient field $K$ and
finite residue field $k$ of characteristic $p$. Set $S=\Spec(R)$, $\eta=\Spec(K)$,
$s=\Spec(k)$. Let $\bar{S}=(\bar{S},\sbar,\etabar)$ be the
normalization of $S$ in a separable closure $\bar{K}$ of $K$ and
denote by $I\subseteq G:=\Gal(\bar{K}/K)$ the inertia subgroup.

\subsection{$l$-adic cohomology for $p\neq l$} In this section $l$ is a prime different from $p$.
The following lemma might be well known as a consequence of de
Jong's theorem on alterations \cite{dj96}, and also of Deligne's
work \cite{weilii} in case $\mathrm{char}(K)=p$. We
shall only need it for $X_\eta\to\Spec(K)$ proper and smooth.

\begin{lemma} Let $X_\eta\to\Spec(K)$ be separated and of finite
type. Then the $G$-representation $H^i(X_\etabar,\bq_l)$ has a
(unique) $G$-invariant weight filtration $$\dots\subseteq
W_jH^i(X_\etabar,\bq_l)\subseteq
W_{j+1}H^i(X_\etabar,\bq_l)\subseteq\cdots$$ in the sense of
\cite{weilii}[Prop.-Def. 1.7.5], i.e. if $F\in G$ is any lift of a
geometric Frobenius element in $\Gal(\bar{k}/k)$ then the
eigenvalues of $F$ on $\gr^W_jH^i(X_\etabar,\bq_l)$ are Weil numbers
of weight $j\in\bz$ with respect to $|k|$. The same is true for the
$G$-representation $H_c^i(X_\etabar,\bq_l)$. One has
$W_{-1}H^i(X_\etabar,\bq_l)=W_{-1}H_c^i(X_\etabar,\bq_l)=0$.
\label{filt}\end{lemma}

\begin{proof} By \cite{weilii}[Prop.-Def. 1.7.5] it suffices to show
that all eigenvalues $\alpha$ of $F$ on $H^i(X_\etabar,\bq_l)$ are
Weil numbers of some weight $j=j(\alpha)\in\bz$. In doing so, one
may pass to an open subgroup $G'\subseteq G$, i.e. replace $X_\eta$
by its base change to a finite extension $K'/K$, since an algebraic
number $\alpha$ is a Weil number with respect to $|k|$ if and only
if $\alpha^{[k':k]}$ is a Weil number with respect to $|k'|$. One
can now argue exactly as in the proof of \cite{berthelot}[Prop.
6.3.2] to which we refer for more details. If $X_\eta$ is the
generic fibre of a proper, strictly semistable scheme, then the
vanishing cycle spectral sequence computed by Rapoport and Zink
\cite{rz}[Satz 2.10] \begin{equation} E_1^{-r,i+r}=\bigoplus_{q\geq
0,r+q\geq 0}H^{i-r-2q}(Y^{(r+2q)},\bq_l)(-r-q)\Rightarrow
H^i(X_\etabar,\bq_l)\label{ss0}\end{equation} together with the Weil
conjectures for the smooth proper schemes $Y^{(i)}$ give the
statement (and moreover the weight filtration on
$H^i(X_\etabar,\bq_l)$ is the filtration induced by the spectral
sequence). If $X_\eta$ is only smooth and proper then by de Jong's
theorem \cite{berthelot}[Thm. 1.4.1] there is a generically finite,
flat $X'_{\eta'}\to X_\eta$ where $X'_{\eta'}$ is strictly
semistable. Hence $H^i(X_\etabar,\bq_l)$ is a direct summand of the
$G'$-representation $H^i(X'_{\etabar},\bq_l)$ for which the
statement holds. Then one can use induction on the dimension
together with the long exact localization sequence to prove the
statement for $H_c^i(X_\etabar,\bq_l)$ for any separated $X_\eta$ of
finite type. Another application of de Jong's theorem is necessary
here to assure that a regular open subscheme $U\subseteq X_\eta$ has
a finite cover $U'\to U$ which is open in a proper regular
$K$-scheme. For $X_\eta$ smooth over $K$, Poincare duality then
implies the statement for $H^i(X_\etabar,\bq_l)$ and for general $X$
one uses a hypercovering argument. In this proof, starting with
(\ref{ss0}), all occurring $F$-eigenvalues have non-negative weight,
i.e. we have
$W_{-1}H^i(X_\etabar,\bq_l)=W_{-1}H_c^i(X_\etabar,\bq_l)=0$.
\end{proof}

Let $f:X\to S$ be a proper, flat, generically smooth morphism of
relative dimension $d$. For $0\leq i\leq 2d$ one defines the
specialization morphism
\begin{equation} \spe: H^i(X_\sbar,\bq_l)\to H^i(X_\etabar,\bq_l)^{I}\label{sp}\end{equation}
as the composite \begin{equation}H^i(X_\sbar,\bq_l)\cong
H^i(X',\bq_l)\to H^i(X'_\eta,\bq_l)\to
H^i(X_\etabar,\bq_l)\label{seq1}\end{equation} where $X'$ is the
base change of $X$ to a strict Henselization of $S$ at $\sbar$ and
the first isomorphism is proper base change. The map $\spe$ is
$G$-equivariant and respects the weight filtration.
\medskip

\begin{theorem} If $X$ is {\em regular} then the following hold.
\begin{itemize}
\item[a)] The map
\[H^i(X_\sbar,\bq_l)=W_iH^i(X_\sbar,\bq_l)\to W_iH^i(X_\etabar,\bq_l)^{I}\]
induced by $\spe$ is surjective for all $i$.
\item[b)] The map
\begin{equation} W_1H^i(X_\sbar,\bq_l)\to
W_1H^i(X_\etabar,\bq_l)^{I}\label{wsp}\end{equation} induced by
$\spe$ is an isomorphism for all $i$, and the zero map for $i>d$.
\item[c)] The map $\spe$ is an isomorphism for $i=0,1$.
\item[d)] If
$W_iH^i(X_\etabar,\bq_l)^{I}=H^i(X_\etabar,\bq_l)^{I}$
for all $i$ then the map
\begin{equation} W_{i-1}H^i(X_\sbar,\bq_l)\to
W_{i-1}H^i(X_\etabar,\bq_l)^{I}\notag\end{equation} induced by
$\spe$ is an isomorphism for all $i$.
\end{itemize}
\label{main}\end{theorem}

\begin{remarks} a) By part a) of the theorem, the assumption of part d) is equivalent to the surjectivity of the map $\spe$ for all $i$, a statement which is called the local theorem on invariant cycles. It is known to hold if $R$ is the local ring of a smooth curve over
$k$ by \cite{weilii}[Lemma 3.6.2], see also \cite{weilii}[Thm. 3.6.1], but it is only conjectured in mixed characteristic. Unconditionally, we were only able to prove the weak statement in b) rather than the full conclusion of d).
Part c) is probably well known and follows, for
example, from b) and results of \cite{sga7}[Expos\'e IX] on Neron
models which assure that
$W_1H^1(X_\etabar,\bq_l)^{I}=H^1(X_\etabar,\bq_l)^{I}$.

b) It is easy to construct examples where $\spe$ is not injective
for $i\geq 2$. For example if $X$ is the blowup of a proper smooth
relative curve over $S$ in a closed point, then $H^2(X_\sbar,\bq_l)$
will have an extra summand $\bq_l(-1)$ corresponding to the
exceptional divisor which gives a new irreducible component of
$X_\sbar$.

c) If $X$ arises by base change from a regular, proper, flat scheme
$\X\to\Spec(\bz)$ part b) of the theorem implies
\[ \ord\limits_{s}\zeta(\X,s)=\ord\limits_{s}\prod_{i=0}^{2d}L(h^i(\X_\bq),s)^{(-1)^i}\] for integers $n\leq 0$ (and for
$n=\frac{1}{2}$) where $\zeta(\X,s)$ is the Zeta-function
of the arithmetic scheme $\X$ and $L(h^i(\X_\bq),s)$ is the $L$-function of the motive $h^i(\X_\bq)$
defined by Serre \cite{serre69}. Indeed the former (resp. latter) is
an Euler product of characteristic polynomials of Frobenius on
$H^i(\X\otimes\bar{\mathbb F}_p,\bq_l)$ (resp.
$H^i(\X\otimes\bar{\bq}_p,\bq_l)^{I_p}$). These are equal
for almost all primes and at the finitely many (bad reduction)
primes where they might differ, part b) assures that the vanishing
order at $s\leq 0$ of both factors, which equals $(-1)^{i+1}$
times the multiplicity of the Frobenius-eigenvalue $p^n$ (of weight 2n),
is the same.

d) Regularity of $X$ is a key assumption in the theorem. The map
$\spe$ will be an isomorphism for $i=0$ if $X$ is only normal but
for $i=1$ normality is not even sufficient for surjectivity of
$\spe$ on $W_0$, as the following example of de Jeu \cite{dejeu00}
shows. If $E$ is an elliptic curve over $\bq$ given by a projective
Weierstrass equation
$$Y^2Z=X^3+AXZ^2+BZ^3$$ with $A,B\in\bq$ then for any
$u\in\bq^\times$ the curve
$$Y^2Z=X^3+u^4AXZ^2+u^6BZ^3$$
is isomorphic to $E$, and if $u^4A,u^6B\in \bz$ this equation
defines a normal scheme $\E$, proper and flat over $\Spec(\bz)$,
inside $\mathbb P^2_{\bz}$. Indeed, the affine coordinate ring of
the complement of the zero section $(X:Y:Z)=(0:1:0)$ is
$R=\bz[x,y]/(y^2-x^3-u^4Ax-u^6B)$ and hence a complete intersection.
So $R$ is normal if and only if all local rings $A_\p$ for primes
$\p$ of height $\leq 1$ are regular. If $\p$ maps to the generic
point of $\Spec(\bz)$ this is clear because $E$ is a smooth curve
over $\bq$. If $\p$ maps to $(p)$ for some prime number $p$, then
$\p=R\cdot p$ since $R\cdot p$ is already a prime ideal as the
equation $y^2-x^3-u^4Ax-u^6B$ remains irreducible modulo $p$. Hence
$\p$ is principal and $A_\p$ is a DVR. The generic point of the zero
section maps to the generic point of $\Spec(\bz)$, hence $\E$ is
normal.

If we now pick $u$ in addition to be a multiple of some prime $p$
where $E$ has split multiplicative reduction, then $\E_{\sbar}$ is a
cuspidal cubic curve and therefore
\[ 0=H^1(\E_{\sbar},\bq_l)\to H^1(\E_{\etabar},\bq_l)^I=W_0H^1(\E_{\etabar},\bq_l)^I\cong\bq_l\]
is not surjective.

However, the condition that $X$ is locally factorial (all local
rings are UFDs) lies between normality and regularity and is
sufficient to ensure that our proof of c) given below goes through.
Regularity is only used for the isomorphism $\Pic(X)\cong\Cl(X)$ and
for \cite{ray}[Thm. 6.4.1] via normality.

\end{remarks}

\begin{proof}

Since the statement of Theorem \ref{main} only depends on the base
change of $f$ to the strict Henselization of $S$ at $\sbar$ we may
assume that $S$ is strictly Henselian. Note that regularity is
preserved by this base change by \cite{miletale}[I,3.17 c)].

For a) we follow Deligne's proof of \cite{weilii}[Thm. 3.6.1],
replacing duality for the essentially smooth morphism $X\to\Spec(k)$
by duality for the morphism $f$ combined with purity for the regular
schemes $X$, proved by Thomason and Gabber (see \cite{fuji96}), and
$S$, proved by Grothendieck in \cite{sga5}[I,Thm. 5.1]. The same
arguments as in loc. cit. lead to the commutative diagram with exact
rows and columns
\begin{equation}\minCDarrowwidth1em\begin{CD}
{} @. {} @. H^{i+1}_{X_\sbar}(X,\bq_l) @.{} @. {}\\
@. @. @AAA @. @.\\
0 @>>> H^{i-1}(X_\etabar,\bq_l)_I(-1) @>>> H^i(X_\eta,\bq_l) @>>>
H^i(X_\etabar,\bq_l)^I @>>> 0\\
@. @. @AAA @AA{\spe}A @.\\
{} @. {} @. H^i(X,\bq_l) @>\sim>> H^i(X_\sbar,\bq_l) @. {}\\
@. @. @AAA @. @.\\
{} @. {} @. H^i_{X_\sbar}(X,\bq_l) @.{} @. {}
\end{CD}\label{dia1}\end{equation}
and after application of the exact functor $W_i$ to a diagram
\[\begin{CD}
{} @. {} @. W_iH^{i+1}_{X_\sbar}(X,\bq_l) @.{} @. {}\\
@. @. @AAA @. @.\\
{} @. {} @. W_iH^i(X_\eta,\bq_l) @>>>
W_iH^i(X_\etabar,\bq_l)^I @>>> 0\\
@. @. @AAA @AA{\spe}A @.\\
{} @. {} @. W_iH^i(X,\bq_l) @>\sim>> W_iH^i(X_\sbar,\bq_l) @. {}
\end{CD}\]
so that it remains to show that
\begin{equation} W_iH^{i+1}_{X_\sbar}(X,\bq_l)=0\label{van09}\end{equation}
for all $i$.
The vertical long exact sequence in (\ref{dia1}) arises by applying
the (exact) global section functor $\Gamma(S,-)$ to the exact
triangle
\[ Rf_*R\underline{\Hom}_X(i_*\bq_l,\bq_l)\to Rf_*\bq_l\to
Rf_*Rj_*\bq_l\] where $i:X_\sbar\to X$ and $j:X_\eta\to X$ are the
inclusions. By purity for $X$ \cite{fuji96}[\S8] we have $\bq_l\cong
Rf^!\bq_l(-d)[-2d]$ and the (sheafified) adjunction between $Rf^!$
and $Rf_!$ gives
\begin{align*}Rf_*R\underline{\Hom}_X(i_*\bq_l,\bq_l)&\cong
R\underline{\Hom}_S(Rf_!i_*\bq_l,\bq_l(-d))[-2d]\\
&\cong R\underline{\Hom}_S(i_{s,*}Rf_{s,*}\bq_l,\bq_l(-d))[-2d]\\
&\cong i_{s,*}R\underline{\Hom}_s(Rf_{s,*}\bq_l,Ri_s^!\bq_l(-d))[-2d]\\
&\cong
i_{s,*}R\underline{\Hom}_s(Rf_{s,*}\bq_l,\bq_l)(-d-1)[-2d-2]\\
\end{align*}
where $i_s:s\to S$ is the closed immersion and $f_s:X_s\to s$ the
base change of $f$. Here we have also used $Rf_*=Rf_!$ ($f$ proper)
as well as the sheafified adjunction between $i_{s,!}=i_{s,*}$ and
$i_s^!$, and purity for $S$. This last complex has cohomology in
degree $i+1$ given by
\[ \Hom_{\bq_l}(H^{2d+2-i-1}(X_\sbar,\bq_l),\bq_l)(-d-1)\]
which has weights greater or equal to $2(d+1)-(2d+2-i-1)=i+1$ since
$W_kH^k(X_\sbar,\bq_l)=H^k(X_\sbar,\bq_l)$ by \cite{weilii}[Cor.
3.3.8]. This finishes the proof of a).
\medskip

Concerning b), we apply the exact functor $W_1$ to the diagram (\ref{dia1}) and obtain
a commutative diagram
\[\begin{CD}
{} @. {} @. W_1H^i(X_\eta,\bq_l) @>\beta >>
W_1H^i(X_\etabar,\bq_l)^I @. {}\\
@. @. @AA\alpha A @AA{\spe}A @.\\
{} @. {} @. W_1H^i(X,\bq_l) @>\sim>> W_1H^i(X_\sbar,\bq_l). @. {}
\end{CD}\]
For $i\geq 2$ the map $\alpha$ is an isomorphism since
\[  W_1H^j_{X_\sbar}(X,\bq_l)\subseteq W_{j-1}H^j_{X_\sbar}(X,\bq_l)=0 \]
for $j=i,i+1$ by (\ref{van09}). For $i=0,1$ we already have
\[H^i_{X_\sbar}(X,\bq_l)\cong \Hom_{\bq_l}(H^{2d+2-i}(X_\sbar,\bq_l),\bq_l)(-d-1)=0\]
before applying $W_1$ and the map $\alpha$ is also an isomorphism. For any $i$ the map $\beta$ is an isomorphism since
\[ W_1\bigl(H^{i-1}(X_\etabar,\bq_l)_I(-1)\bigr)=W_{-1}\bigl(H^{i-1}(X_\etabar,\bq_l)_I\bigr)(-1)=0\]
by Lemma \ref{filt}. Hence the map induced by $\spe$ on $W_1$ is also an isomorphism. For $i>d$ both sides of (\ref{wsp}) vanish. Indeed, the weights of
$H^i(X_{\sbar},\bq_l)$ are greater or equal to $2(i-d)\geq 2$ by
\cite{weilii}[Cor. 3.3.4] and the same is true for
$H^i(X_\etabar,\bq_l)$ as follows from Poincare duality and the fact
that the weights on $H^i(X_\etabar,\bq_l)$ are $\leq 2i$ for $i<d$.
This in turn can be read off from the spectral sequence (\ref{ss0})
in the strictly semistable case and follows in general from de
Jong's theorem. Hence
\[W_1H^i(X_\sbar,\bq_l)=W_1H^i(X_\etabar,\bq_l)^{I}=0\]
for $i>d$ and we have finished the proof of b).

Concerning d), we apply the exact functor $W_{i-1}$ to the diagram (\ref{dia1}) and obtain a commutative diagram
\[\begin{CD}
{} @. {} @. W_{i-1}H^i(X_\eta,\bq_l) @>\beta >>
W_{i-1}H^i(X_\etabar,\bq_l)^I @. {}\\
@. @. @AA\alpha A @AA{\spe}A @.\\
{} @. {} @. W_{i-1}H^i(X,\bq_l) @>\sim>> W_{i-1}H^i(X_\sbar,\bq_l) @. {}
\end{CD}\]
where $\alpha$ is an isomorphism for the same reason as in the proof of b) and $\beta$ is an isomorphism since
\[ W_{i-1}\bigl(H^{i-1}(X_\etabar,\bq_l)_I(-1)\bigr)=W_{i-3}\bigl(H^{i-1}(X_\etabar,\bq_l)_I\bigr)(-1)\]
is dual to
\[ H^{2d-i+1}(X_\etabar,\bq_l)^I(d+1)/W_{2d-i+2}\bigl(H^{2d-i+1}(X_\etabar,\bq_l)^I\bigr)(d+1)\]
which vanishes by the assumption in d).

Concerning c), the case $i=0$ follows from b) since
$W_1H^0(X_\sbar,\bq_l)=H^0(X_\sbar,\bq_l)$ and
$W_1H^0(X_\etabar,\bq_l)^I=H^0(X_\etabar,\bq_l)^I$. The case $i=1$
can be deduced from b) and \cite{sga7}[Expos\'e IX] or from results
of Raynaud on the Picard functor \cite{ray}. We give the details of
this last argument because the method, essentially using motivic cohomology, might be of some interest. The short exact sequence $0\to\mu_{l^\nu}\to
\bg_m\xrightarrow{l^\nu}\bg_m\to 0$ of sheaves on $X_{et}$ induces
an isomorphism \begin{equation} R^1f_*\mu_{l^\nu}\cong
(R^1f_*\bg_m)_{l^\nu}\label{kummer}\end{equation} of sheaves on
$S_\et$ since $(f_*\bg_m)/l^\nu=0$. Indeed, the stalks
$H^0(Y,\co_Y^\times)=\prod_i R_i^\times$ and
$H^0(Y_\etabar,\co_{Y_\etabar}^\times)=\prod_i
(L_i\otimes_K\bar{K})^\times$ of $f_*\bg_m$ are $l$-divisible since
$S$ is strictly Henselian. Here $$X\to Y=\coprod_i\Spec(R_i)\to S$$ is the Stein
factorization and $L_i$ is the fraction field of $R_i$. The Leray
spectral sequence for $f$ gives an exact sequence
\[ 0\to H^1(S,f_*\bg_m)\to H^1(X,\bg_m)\to
H^0(S,R^1f_*\bg_m)\to H^2(S,f_*\bg_m)\] and $H^i(S,f_*\bg_m)=0$ for
$i=1,2$. Indeed, $H^1(S,f_*\bg_m)=\Pic(Y)=0$ (resp.
$H^2(S,f_*\bg_m)=\mathrm{Br}(Y)=0$) since $Y$ is the disjoint union
of spectra of local (resp. strictly Henselian local) rings. Hence
\begin{equation} \Pic(X)\cong H^1(X,\bg_m)\cong
H^0(S,R^1f_*\bg_m).\label{eq1}\end{equation} A similar argument for
$f_\eta$ shows
\begin{equation} \Pic(X_\eta)\cong H^1(X_\eta,\bg_m)\cong
H^0(\eta,R^1f_*\bg_m).\label{eq2}\end{equation}
 We have a commutative diagram
\[\minCDarrowwidth1em\begin{CD} H^1(X_\sbar,\mu_{l^\nu}) @>\sim >> H^0(S,R^1f_*\mu_{l^\nu})
@>\sim>> H^0(S,R^1f_*\bg_m)_{l^\nu} @>\sim>> \Pic(X)_{l^\nu}\\ @V\mathrm{sp} VV @VVV @VVV @VVV\\
H^1(X_\etabar,\mu_{l^\nu})^I @>\sim >> H^0(\eta,R^1f_*\mu_{l^\nu})
@>\sim>> H^0(\eta,R^1f_*\bg_m)_{l^\nu} @>\sim>> \Pic(X_\eta)_{l^\nu}
\end{CD}\]
where the isomorphisms in the top row are given by proper base
change, (\ref{kummer}) and (\ref{eq1}) and in the bottom row by an
elementary stalk computation, (\ref{kummer}) and (\ref{eq2}).
Passing to the inverse limit over $\nu$ we are reduced to studying
the map \begin{equation} \varprojlim_\nu
\Pic(X)_{l^\nu}=:T_l\Pic(X)\to T_l\Pic(X_\eta):=\varprojlim_\nu
\Pic(X_\eta)_{l^\nu}\label{eq3}\end{equation} and the proof of c)
for $i=1$ is then finished by the following Lemma.
\end{proof}

\begin{lemma} The map (\ref{eq3}) is injective with finite cokernel.
\end{lemma}

\begin{proof} Since $X$ is regular and $X_\eta$ is an
open subscheme the map
\[\Pic(X)=\Cl(X)\to
\Cl(X_\eta)=\Pic(X_\eta)\] is surjective and its kernel $K$ is the
subgroup of $\Cl(X)$ generated by divisors supported in the closed
subscheme $X_\sbar\subset X$, hence is a finitely generated abelian
group \cite{hartshorne}[II.6]. By the snake lemma we obtain an exact
sequence
\begin{equation} 0=T_lK\to T_l\Pic(X)\to T_l\Pic(X_\eta)\to
\hat{K}\xrightarrow{\rho}\hat{\Pic}(X)\label{eq5}\end{equation}
where $\hat{A}=\varprojlim_\nu A/l^\nu$ denotes the $l$-completion
of an abelian group $A$.

Let $\Pic^0(X)\subseteq \Pic(X)$ be the subgroup defined in
\cite{ray}[3.2 d)], i.e. the kernel of the map \begin{equation}
\Pic(X)=P(S)\to (P/P^0)(\sbar)\times (P/P^0)(\etabar)
\label{eq4}\end{equation} where $P=\Pic_{X/S}$ is the relative
Picard functor of $f$ \cite{ray}[1.2] and $P^0$ is the connected
component of $P$ restricted to schemes over $\sbar$ (resp. $\eta$).
Note that over a field $P$ is represented by a group scheme, locally
of finite type, hence has a well defined connected component. By
\cite{ray}[Thm. 3.2.1] the target group in (\ref{eq4}) - the product
of the Neron-Severi groups of the geometric fibres - is finitely
generated, hence so is $\Pic(X)/\Pic^0(X)$.

By \cite{ray}[Thm. 6.4.1] - and this is the key fact in the proof-
the group $K\cap\Pic^0(X)$ is finite. In the notation of loc. cit.
we have $K=E(S)$ by Prop. 6.1.3 and $\Pic^0(X)=P^0(S)\subseteq
P^\tau(S)$. Hence the kernel of $K\to\Pic(X)/\Pic^0(X)$ is finite
and since both groups are finitely generated, so is the kernel on
their $l$-completions. But this means that the map $\rho$ in
(\ref{eq5}) has finite kernel which proves the Lemma.\end{proof}

\subsection{$p$-adic cohomology} In this section we assume that $K$ has characteristic $0$ and for simplicity also that $k=\mathbb F_p$. For $l=p$ one still has the specialisation map
\begin{equation} \spe: H^i(X_\sbar,\bq_p)\to H^i(X_\etabar,\bq_p)^{I}\label{sp-p}\end{equation}
since proper base change holds for arbitrary torsion sheaves. However, it is well known that $p$-adic \'etale cohomology of varieties in characteristic $p$ only captures the slope $0$ part of the full $p$-adic cohomology, which is Berthelot's rigid cohomology $H^i_{rig}(X_s/k)$ (for proper $X_s$ this follows from \cite[Thm. 1.1]{bbe07} and \cite[Prop. 3.28, Lemma 5.6]{illusie79}). Here the slope $0$ part $V^{\slope 0}$ of a finite dimensional $\bq_p$-vector space $V$ with an endomorphism $\phi$ is the maximal subspace on which the eigenvalues of $\phi$ are $p$-adic units. One knows that the eigenvalues of $\phi$ on $H^i_{rig}(X_s/k)$ are Weil numbers, and a proof similar to that of Lemma \ref{filt} shows that the same is true for $D_{pst}(H^i(X_\etabar,\bq_p))$, and hence for $$D_{cris}(H^i(X_\etabar,\bq_p))=D_{st}(H^i(X_\etabar,\bq_p))^{N=0}=D_{pst}(H^i(X_\etabar,\bq_p))^{I,N=0}.$$ Therefore one deduces weight filtrations on both spaces.

In analogy with the $l$-adic situation one might make the following conjecture.

\begin{conjecture} Let $X\to S$ be proper, flat and generically smooth. Then there is a $\phi$-equivariant specialization map
\[ H^i_{rig}(X_s/k) \xrightarrow{\spe'} D_{cris}(H^i(X_\etabar,\bq_p))\]
and a commutative diagram of $\Gal(\bar{k}/k)$-modules
\[\begin{CD} H^i(X_\sbar,\bq_p) @>\spe >> H^i(X_\etabar,\bq_p)^{I}\\
@V\lambda_s VV @V\lambda_\eta VV \\
H^i_{rig}(X_s/k)\otimes_{\bq_p}\bqpur @>\spe'\otimes 1 >> D_{cris}(H^i(X_\etabar,\bq_p))\otimes_{\bq_p}\bqpur
\end{CD}\]
where $\bqpur$ is the $p$-adic completion of the maximal unramified extension of $\bq_p$. Moreover, the vertical maps induce isomorphisms
\[\lambda_s: H^i(X_\sbar,\bq_p)\cong (H^i_{rig}(X_s/k)\otimes_{\bq_p}\bqpur)^{\phi\otimes\phi=1}\cong H^i_{rig}(X_s/k)^{\slope 0}\]
and
\[\lambda_\eta:H^i(X_\etabar,\bq_p)^{I}\cong (D_{cris}(H^i(X_\etabar,\bq_p))\otimes_{\bq_p}\bqpur)^{\phi\otimes\phi=1}\cong D_{cris}(H^i(X_\etabar,\bq_p))^{\slope 0}.\]
\label{specon}\end{conjecture}
Note here that for any $\phi$-module $D$ the $\Gal(\bar{k}/k)$-module $(D\otimes_{\bq_p}\bqpur)^{\phi\otimes\phi=1}$ can also be viewed as a $\phi$-module (via the action of $\phi\otimes 1$) and as such is non-canonically isomorphic to $D^{\slope 0}$. Moreover the action of $\Frob_p^{-1}\in
\Gal(\bar{k}/k)$ coincides with that of $1\otimes\phi^{-1}=\phi\otimes 1=\phi$.

The $p$-adic analogue of Theorem \ref{main} (replacing a) by the conjectural local theorem on invariant cycles) would be the following conjecture.

\begin{conjecture}Assume that $X$ is moreover regular. Then the following hold. \begin{itemize}
\item[a)] The map $\spe'$ is surjective.
\item[b)] The map
\[W_{i-1}H^i_{rig}(X_s/k) \xrightarrow{\spe'}W_{i-1}D_{cris}(H^i(X_\etabar,\bq_p))\]
induced by $\spe'$ is an isomorphism.
\item[c)] The map $\spe'$ is an isomorphism for $i=0,1$.
\end{itemize}
\end{conjecture}

Combining both conjectures we deduce the following statement for $p$-adic \'etale cohomology.

\begin{conjecture} If $X$ is regular then the map
\[W_{i-1}H^i(X_\sbar,\bq_p)\xrightarrow{\spe}W_{i-1}H^i(X_\etabar,\bq_p)^{I}\]
induced by $\spe$ is an isomorphism.
\end{conjecture}

Here we deduce the weight filtrations on $H^i(X_\sbar,\bq_p)$ and $H^i(X_\etabar,\bq_p)^{I}$ from Conjecture \ref{specon} via the injectivity of the maps $\lambda_s$ and $\lambda_\eta$. For the applications in this paper we only need this isomorphism on $W_0$ (or in fact on the still smaller generalized eigenspace for the eigenvalue $1$). For reference we record this statement separately.

\begin{conjecture} If $X$ is regular then the map
\[W_0H^i(X_\sbar,\bq_p)\xrightarrow{\spe}W_0H^i(X_\etabar,\bq_p)^{I}\]
induced by $\spe$ is an isomorphism, where $W_0$ is the sum of generalized $\phi$-eigenspaces for eigenvalues which are roots of unity.
\label{w0}\end{conjecture}

Again, if Conjecture \ref{specon} holds the maps $\lambda_s$ and $\lambda_\eta$ are injectve and it suffices to establish an isomorphism $$W_0H^i_{rig}(X_s/k)\cong W_0D_{cris}(H^i(X_\etabar,\bq_p)).$$ We do not know how to establish Conjecture \ref{specon} or Conjecture \ref{w0} in general, since it seems difficult to make use of the regularity assumption. In case $X$ has semistable reduction, however, it seems plausible that one can avoid any reference to rigid cohomology and establish a commutative diagram of $\Gal(\bar{k}/k)$-modules
\[\begin{CD} H^i(X_\sbar,\bq_p) @>\spe >> H^i(X_\etabar,\bq_p)^{I}\\
@V\tilde{\lambda}_s VV @V\lambda_\eta VV \\
(H^i_{HK}(X_s/k)^{N=0})\otimes_{\bq_p}\bqpur @>c'\otimes 1 >> D_{cris}(H^i(X_\etabar,\bq_p))\otimes_{\bq_p}\bqpur
\end{CD}\]
where $H^i_{HK}(X_s/k)$ is Hyodo-Kato cohomology. Contrary to what the notation suggests this cohomology theory not only depends on $X_s/k$ but on the scheme $X/S$. Building on work of Fontaine-Messing, Bloch-Kato, Hyodo-Kato, and Kato-Messing, Tsuji \cite{tsuji99} proved that there is an isomorphism of $(\phi,N)$-modules $$H^i_{HK}(X_s/k)\xrightarrow{c}D_{st}(H^i(X_\etabar,\bq_p))$$ and hence an isomorphism of $\phi$-modules
$$H^i_{HK}(X_s/k)^{N=0}\xrightarrow{c'}D_{st}(H^i(X_\etabar,\bq_p))^{N=0}=D_{cris}(H^i(X_\etabar,\bq_p)).$$
In addition to the commutative diagram it would then be enough to show that $\tilde{\lambda}_s$ and $\lambda_\eta$ are injective. We refrain from giving more details since in this paper Conjecture \ref{w0} is only used in the proof of Proposition \ref{tamcompare} (via Proposition \ref{reform}) which already needs to assume a host of other, much deeper conjectures that we are unable to prove.


\Addresses

\begin{bibdiv}
\begin{biblist}
\bib{av}{article}{
  author={Artin, M.},
  author={Verdier, J.-L.},
  title={Seminar on \'Etale Cohomology of Number Fields},
  status={Lecture notes prepared in connection with a seminar held at the Woods Hole Summer Institute on Algebraic Geometry, July 6-31, 1964},
}

\bib{berthelot}{article}{
  author={Berthelot, P.},
  title={Alt\'erations de vari\'et\'es alg\'ebriques},
  booktitle={S\'eminaire Bourbaki, exp. 815 (1995/96), Ast\'erisque \bf {241}},
  date={1997},
  pages={273\ndash 311},
}

\bib{bbe07}{article}{
  author={Berthelot, P.},
  author={Bloch, S.},
  author={Esnault, H.},
  title={On Witt vector cohomology for singular varieties},
  journal={Compositio Math.},
  volume={143},
  date={2007},
  pages={363\ndash 392},
}

\bib{bk88}{article}{
  author={Bloch, S.},
  author={Kato, K.},
  title={L-functions and Tamagawa numbers of motives},
  booktitle={In: The Grothendieck Festschrift I, Progress in Math. \bf {86}},
  publisher={Birkh\"auser},
  place={Boston},
  date={1990},
  pages={333\ndash 400},
}

\bib{blght}{article}{
  author={Barnet-Lamb, T.},
  author={Geraghty, D.},
  author={Harris, M.},
  author={Taylor, R.},
  title={A family of Calabi-Yau varieties and potential automorphy II},
  status={preprint available at http://www.math.harvard.edu/~rtaylor/},
}

\bib{bourbaki}{book}{
  author={Bourbaki, N.},
  title={General topology},
  publisher={Springer},
  date={1989},
}

\bib{bufl98}{article}{
  author={Burns, D.},
  author={Flach, M.},
  title={On Galois structure invariants associated to Tate motives},
  journal={Amer. J. Math.},
  volume={120},
  date={1998},
  pages={1343\ndash 1397},
}

\bib{dejeu00}{article}{
  author={de Jeu, R.},
  title={Appendix to the paper of Scholl: a counterexample to a conjecture of Beilinson},
  book={title={The arithmetic and geometry of algebraic cycles (Banff, AB, 1998)}, series={NATO Sci. Ser. C Math. Phys. Sci. 548}, publisher={Kluwer Acad. Publ.}, date={2000}},
  pages={491\ndash 493},
}

\bib{weilii}{article}{
  author={Deligne, P.},
  title={La Conjecture de Weil. II},
  date={1980},
  journal={Publ. Math. IHES},
  volume={52},
  pages={137\ndash 252},
}

\bib{diac75}{article}{
  author={Diaconescu, R.},
  title={Change of base for toposes with generators},
  journal={Jour. Pure and Appl. Math.},
  volume={6},
  date={1975},
  pages={191\ndash 218},
}

\bib{dj96}{article}{
  author={de Jong, J. A.},
  title={Smoothness, semi-stability and alterations},
  journal={Publ. Math. IHES},
  volume={83},
  date={1996},
  pages={51\ndash 93},
}

\bib{fgiknv05}{book}{
  author={Fantechi, B.},
  author={Göttsche, L.},
  author={Illusie, L.},
  author={Kleiman, S.},
  author={Nitsure, N.},
  author={Vistoli, A.},
  title={Fundamental algebraic geometry. Grothendieck's FGA explained},
  series={Mathematical Surveys and Monographs},
  volume={123},
  publisher={Amer. Math. Soc.},
  place={Providence},
  date={2005},
}

\bib{flach06-1}{article}{
  author={Flach, M.},
  title={Iwasawa theory and motivic L-functions},
  journal={Pure and Appl. Math. Quarterly, Jean Pierre Serre special issue, part II,},
  date={2009},
  volume={5},
  number={1},
  pages={255\ndash 294},
}

\bib{flach06-2}{article}{
  author={Flach, M.},
  title={Cohomology of topological groups with applications to the Weil group},
  journal={Compositio Math.},
  date={2008},
  volume={144},
  number={3},
  pages={633\ndash 656},
}

\bib{fpr91}{article}{
  author={Fontaine, J.-M.},
  author={Perrin-Riou, B.},
  title={Autour des conjectures de Bloch et Kato: cohomologie galoisienne et valeurs de fonctions L},
  booktitle={In: Motives (Seattle)},
  series={Proc. Symp. Pure Math. \bf {55-1}},
  date={1994},
  pages={599\ndash 706},
}

\bib{fuji96}{article}{
  author={Fujiwara, K.},
  title={A Proof of the Absolute Purity Conjecture (after Gabber)},
  book={title={Algebraic Geometry 2000, Azumino (Hotaka)}, series={Advanced Studies in Pure Mathematics 36},date={2002}, publisher={Math. Soc. Japan}},
  pages={153\ndash 183},
}

\bib{geisser04}{article}{
  author={Geisser, T.},
  title={Weil-\'etale cohomology over finite fields},
  journal={Math. Ann.},
  volume={330},
  date={2004},
  pages={665\ndash 692},
}

\bib{geisser05}{article}{
  author={Geisser, T.},
  title={Arithmetic cohomology over finite fields and special values of $\zeta $-functions},
  journal={Duke Math. Jour.},
  volume={133},
  number={1},
  date={2006},
  pages={27\ndash 57},
}

\bib{sga4}{book}{
  author={Grothendieck, A.},
  author={Artin, M.},
  author={Verdier, J.L.},
  title={Theorie des Topos et Cohomologie Etale des Schemas (SGA 4)},
  publisher={Springer},
  date={1972},
  series={Lecture Notes in Math \bf {269, 270, 271}},
}

\bib{sga5}{book}{
  author={Grothendieck, A.},
  author={Illusie, L.},
  title={Cohomologie l-adique et Fonctions L (SGA 5)},
  publisher={Springer},
  date={1977},
  series={Lecture Notes in Math \bf {589}},
}

\bib{sga7}{book}{
  author={Grothendieck, A.},
  author={Rim, D.S.},
  author={Raynaud, M.},
  title={Groupes de Monodromie en G\'eom\'etrie Alg\'ebrique (SGA 7)},
  publisher={Springer},
  date={1972},
  series={Lecture Notes in Math \bf {288, 340}},
}

\bib{hartshorne}{book}{
  author={Hartshorne, R.},
  title={Algebraic Geometry},
  publisher={Springer},
  series={Graduate texts in Mathematics \bf {52}},
  date={1977},
}

\bib{illusie79}{article}{
  author={Illusie, L.},
  title={Complexe de de Rham-Witt et cohomologie cristalline},
  journal={Ann. Sci. de l'ENS},
  volume={12},
  number={4},
  date={1979},
  pages={501\ndash 666},
}

\bib{illusie09}{article}{
  author={Illusie, L.},
  title={On oriented products, fibre products and vanishing toposes},
  status={available from http://www.math.u-psud.fr/$\sim $illusie/},
}

\bib{johnstone02}{book}{
  author={Johnstone, P.},
  title={Sketches of an elephant: a topos theory compendium I, II},
  publisher={Clarendon Press},
  place={Oxford},
  date={2002},
}

\bib{knumum}{article}{
  author={Knudsen, F.},
  author={Mumford, D.},
  title={The projectivity of the moduli space of stable curves I: Preliminaries on `det' and `Div'},
  journal={Math. Scand.},
  volume={39},
  date={1976},
  pages={19\ndash 55},
}

\bib{li01}{article}{
  author={Lichtenbaum, S.},
  title={The Weil-\'Etale Topology on schemes over finite fields},
  journal={Compositio Math.},
  volume={141},
  date={2005},
  pages={689\ndash 702},
}

\bib{li04}{article}{
  author={Lichtenbaum, S.},
  title={The Weil-\'etale Topology for number rings},
  journal={Annals of Math.},
  date={2009},
  volume={170},
  number={2},
  pages={657\ndash 683},
}

\bib{miletale}{book}{
  author={Milne, J.S.},
  title={\'Etale Cohomology},
  series={Princeton Math. Series {\bf 17}},
  publisher={Princeton University Press},
  date={1980},
}

\bib{moerdijk88}{article}{
  author={Moerdijk, I.},
  title={The classifying topos of a continuous groupoid I},
  journal={Trans. Amer. Math. Soc.},
  volume={310},
  number={2},
  date={1988},
  pages={629\ndash 668},
}

\bib{morin}{thesis}{
  author={Morin, B.},
  title={Sur le topos Weil-\'etale d'un corps de nombres},
  date={2008},
  organization={Thesis, Universit\'e de Bordeaux},
}

\bib{morin09}{article}{
  author={Morin, B.},
  title={On the Weil-\'Etale Cohomology of Number Fields},
  status={to appear in Trans. Amer. Math. Soc.},
}

\bib{ray}{article}{
  author={Raynaud, M.},
  title={Sp\'ecialisation du foncteur du Picard},
  journal={Publ. Math. IHES},
  date={1970},
  volume={38},
  pages={27\ndash 76},
}

\bib{rz}{article}{
  author={Rapoport, M.},
  author={Zink, T.},
  title={\"Uber die lokale Zetafunktion von Shimuravariet\"aten. Monodromiefiltration und verschwindende Zyklen in ungleicher Charakteristik},
  journal={Invent. Math.},
  volume={68},
  date={1982},
  pages={53\ndash 108},
}

\bib{scholl90}{article}{
  author={Scholl, A.J.},
  title={Motives for modular forms},
  journal={Invent. Math.},
  volume={100},
  date={1990},
  pages={419\ndash 430},
}

\bib{serre69}{article}{
  author={Serre, J.P.},
  title={Facteurs locaux des fonctions zeta des vari\'et\'es alg\'ebriques},
  journal={S\'eminaire DPP, expos\'e 19},
  date={1969/70},
}

\bib{tsuji99}{article}{
  author={Tsuji, T.},
  title={$p$-adic \'etale cohomology and crystalline cohomology in the semi-stable reduction case},
  journal={Invent. Math.},
  volume={137},
  date={1999},
  pages={233\ndash 411},
}


\end{biblist}
\end{bibdiv}
\end{document}